\documentclass[a4wide,11pt]{article}%
\usepackage{geometry}                

\usepackage{graphicx}
\usepackage{authblk}
\usepackage{amssymb}
\usepackage{epstopdf}
  \usepackage{subfigure}  
\usepackage[sans]{dsfont}
\usepackage[applemac]{inputenc}
\usepackage[english]{babel}
\usepackage{latexsym}
\usepackage{mathrsfs}
\usepackage{amscd}
\usepackage{graphicx}
\usepackage{color}
\usepackage{float}
\usepackage{amsmath}
\usepackage{amsfonts}
\numberwithin{equation}{section}
	\usepackage{enumerate}
\usepackage{amsthm}
\usepackage[numbers]{natbib}
\usepackage[bookmarks=true,colorlinks=true,linkcolor={blue},urlcolor={blue}, citecolor={blue},pdfstartview={XYZ null null 1.22}]{hyperref}%

\newcommand{\N}{\mathbb N}

\def\be#1\ee{\begin{equation}#1\end{equation}}

\newtheorem{proposition}{Proposition}

\theoremstyle{definition}
\newtheorem{alg}{Algorithm}[section]
\newtheorem{remark}{Remark}

\def\RR{\mathbb R}

\def\cm{c_{\textrm{max}}}
\def\ps{\p_{\infty}}
\def\Ur{\textcolor{black}{V_r}}
\def\Ua{\textcolor{black}{V_a}}
\def\p{\rho}

\def\N{\mathcal N}
\def\L{\mathcal L}

\def\be{\begin{equation}}
\def\ee{\end{equation}}
\def\bea{\begin{eqnarray}}
\def\eea{\end{eqnarray}}

\title{Opinion dynamics over complex networks:\\ kinetic modeling and numerical methods}

\author{Giacomo Albi\thanks{Fakult\"at f\"ur Mathematik, Technische Univarsit\"at M\"unchen, Germany }\qquad  Lorenzo Pareschi\thanks{Department of Mathematics and Computer Science, University of
Ferrara, Italy}
\qquad 
Mattia Zanella\thanks{Department of Mathematics and Computer Science, University of
Ferrara, Italy}}

\begin{document}
\maketitle

%
%
%
%
%
%
\begin{abstract}
In this paper we consider the modeling of opinion dynamics over time dependent large scale networks. A kinetic description of the agents' distribution over the evolving network is considered which combines an opinion update based on binary interactions between agents with a dynamic creation and removal process of new connections. The number of connections of each agent influences the spreading of opinions in the network but also the way connections are created is influenced by the agents' opinion. The evolution of the network of connections is studied by showing that its asymptotic behavior is consistent both with Poisson distributions and truncated power-laws. In order to study the large time behavior of the opinion dynamics a mean field description is derived which allows to compute exact stationary solutions in some simplified situations. Numerical methods which are capable to describe correctly the large time behavior of the system are also introduced and discussed. Finally, several numerical examples showing the influence of the agents' number of connections in the opinion dynamics are reported.  
\end{abstract}



\section{Introduction}
In recent years, the importance of large scale social networks has grown enormously and their study has raised lots of attentions, with the aim to understand how their structure and connections may influence the spread of opinions and ideas through human networks \cite{Ace,APZc,Das,DL,Patt}. A major research topic is how to model the information exchange and, in particular, to understand and analyze the effects of interpersonal influence on processes such as opinion formation and creation and removal of new connections. The latter aspect is closely related to the construction of graph models for complex networks and has emerged as one of the most active research fields \cite{AB,ASBS,BA,N,S}. The empirical studies of technological and communication networks has been actively investigated thanks to a huge amount of data coming from the online platforms. From the theoretical point of view it is an unprecedented laboratory for testing the collective behavior of large populations of agents \cite{BAJ,WS}. The need to handle with millions, and often billions, of vertices implied a considerable shift of interest to large-scale statistical properties of graphs. 

In this context kinetic theory may play a major rule in designing effective models to characterize the statistical features of the opinion dynamics over such large collection of data. In particular, it can be used to analyze the so called \emph{stylized facts} of the dynamics, like the asymptotic degree distribution of the connections in the network and the large time opinion behavior. To this aim, in this paper, we extend the kinetic model of opinion formation introduced in \cite{T} to the case where each agent possesses a certain number of connections in the network. These connections evolve accordingly to a preferential attachment dynamics for the removal and creation of new connections. In this sense, the model here proposed fall in the general class of kinetic models for socio-economic problems where the dynamics of the model is influenced by additional characteristics of the agents, like personal conviction, leadership and knowledge \cite{APZa, BrTo, DT, PTa, PTb}.      
  
In principle, the modeling proposed here is not limited to a particular kind of opinion dynamics and one can adapt other models developed in the literature \cite{BMS, DMPW, SWS} to evolve over the network by following the ideas presented in this paper. We mention here that recently opinion models have been considered in the context of optimal control in \cite{AHP,APZa,APZc}. In a recent note \cite{APZc} we faced the solution of an optimal control problem for a model of opinion dynamics described by a system of ordinary differential equations over an evolving network. More precisely we considered a network with a fixed number of vertices and edges which modifies its configuration of connections in time through a preferential attachment rewiring process. 

A further contribution of the present manuscript is the development of numerical methods which are capable to describe correctly the large time behavior of the system. In particular we will focus on finite-difference schemes for the mean-field description of the opinion model over the network inspired by the well-known Chang-Cooper method \cite{BD, BCDS, CC, LLPS}. We remark that, at variance with the standard Chang-Cooper method, the Fokker-Planck model considered here is nonlinear. Similar schemes for nonlinear Fokker-Planck equations have been previously introduced in \cite{BCDS, LLPS}.  

The rest of the paper is organized as follows. In Section 2 we introduce the kinetic model and describe the evolution of the network of connections. The main properties of the network and the evolution of some macroscopic quantities, like the mean and the variance of the opinion over the network are discussed. Next in Section 3 we derive a Fokker-Planck model for the opinion dynamics under the classical quasi-invariant scaling. This permits to compute asymptotic stationary solutions of the opinion over the graph in some simplified situations. Section 4 is devoted to the construction of numerical methods for the above problems. Monte Carlo methods for the Boltzmann model and finite difference schemes for the Fokker-Planck model which are capable to describe correctly the steady states of the system are introduced. Finally in Section 5 several numerical examples illustrate our findings and show the behavior of the model. In separate Appendices we report proofs of related to the main properties of the network and to the positivity preservation property of the finite difference scheme. 

\section{The kinetic model}\label{sec:model}
In this section we introduce a general mathematical model based on a kinetic description for the study of the opinion formation on a large evolving network. 

\subsection{Opinion dynamic}
Let us consider a large system of agents interacting through a given network. 
We associate to each agent an opinion $w$, which varies continuously in a closed subset whose bounds denote two extreme and opposite opinions, and its number of connections $c$, as a discrete variable varying between $0$ and the maximum number of connections allowed by the network. Note that this maximum number typically is a fixed value which is several orders of magnitude smaller then the size the network.

We are interested in the evolution of the density function 
\begin{equation}\label{eq:def_f}
f=f(w,c,t), \qquad f: I\times {\mathcal C}\times \mathbb{R}^+ \rightarrow \mathbb{R}^+
\end{equation}
where $w\in I,I=[-1,1]$ is the opinion variable, $c\in {\mathcal C}=\{0,1,2,\ldots,\cm\}$ is a discrete variable describing the number of connections and $t\in\RR^+$ denotes as usual the time variable. For each time $t\ge 0$ we can compute the following marginal density
\begin{equation}\label{eq:integration_wc}
\p(c,t)=\int_I f(w,c,t)dw, 
\end{equation}
which defines the evolution of the number of connections of the agents or equivalently the degree distribution of the network. 
In the sequel we assume that the total number of agents is conserved, namely 
\be
\sum_{c=0}^{\cm}\p(c,t)=1.
\ee
The overall opinion distribution is defined likewise as the following marginal density function
\begin{equation}
g(w,t)=\sum_{c=0}^{\cm} f(w,c,t).
\end{equation}
We express the evolution of the opinions by a binary interaction rule. From a microscopic point of view we suppose that the agents modify their opinion through binary interactions which depend on opinions and number of connections. If two agents with opinion and number of connections $(w,c)$ and $(w_*,c_*)$ meet, their post-interaction opinion is given by
\begin{equation}\begin{cases}\label{eq:binary}
w' &=w-\eta P(w,w_*;c,c_*)(w-w_*)+\xi D(w,c), \\
w_*' &=w_*-\eta P(w_*,w;c_*,c)(w_*-w)+\xi_* D(w_*,c_*),
\end{cases}\end{equation}
where $w,w_*\in I=[-1,1]$ denote the pre-interaction opinions and $w',w_*'$ the opinion after the exchange of information between the two agents. Note that, in the present setting the compromise function $P(\cdot,\cdot;\cdot,\cdot)$ depends both on the opinions and on the number of connections of each agent. In \eqref{eq:binary} the nonnegative parameter $\eta$ influences the compromise rate while $\xi,\xi_*$ are centered random variables with the same distribution $\Theta$ with finite variance $\varsigma^2$   and taking values on a Borel set $\mathcal{B}\subset\mathbb{R}$. The function $D(\cdot,\cdot) \geq 0$ describes the local relevance of the diffusion for a given opinion and number of connections. We will consider by now a general interaction potential such that $0\le P(w,w_*,c,c_*)\le 1$.

In absence of diffusion, $\xi,\xi_*\equiv 0$, from \eqref{eq:binary} we have 
\begin{equation}
|w'-w_*'| = |1-\eta(P(w,w_*;c,c_*)+P(w_*,w;c_*,c))||w-w_*|, 
\end{equation}
then the post-exchange distances between agents are still in the reference interval $[-1,1]$ if we consider $\eta\in (0,1/2)$ and $0\le P(w,w_*,c,c_*)\le 1$. In agreement with \cite{AHP,APZa,DMPW,T} we can state the following result which derives the conditions on the noise term to ensure that the post-interaction opinions do not leave the reference interval.
\begin{proposition}
If we assume that $0<P(w,w_*;c,c_*)\le 1$ and 
\[
|\xi|< d, \qquad |\xi_*|< d
\]
where $$d=\min_{(w,c)\in I\times {\mathcal C}}\Big\{\dfrac{(1-w)}{D(w,c)},D(w,c)\ne 0\Big\},$$ then the binary interaction rule \eqref{eq:binary} preserves the bounds being the post interaction opinions $w,w_*$ contained in $I=[-1,1]$.
\end{proposition}
The evolution in time of the density function $f(w,c,t)$ is described by the following integro-differential equation of Boltzmann-type
\begin{equation}\label{eq:boltz_lin}
\dfrac{d}{d t}f(w,c,t)+\N[f(w,c,t)]=Q(f,f)(w,c,t),
\end{equation}
where $\N[\cdot]$ is an operator which is related to the evolution of the connections in the network and $Q(\cdot,\cdot)$ is the binary interaction operator defined as follows
\begin{equation}
Q(f,f) = \sum_{c_*=0}^{\cm}\int_{\mathcal{B}^2\times I} \left('B\dfrac{1}{J}f('w,c)f('w_*,c_*)-Bf(w,c)f(w_*,c_*)\right)dw_*d\xi d\xi_*,
\label{eq:Bo}
\end{equation}
where $('w,'w_*)$ are the pre-interaction opinions generated by the couple $(w,w_*)$ after the interaction. The term $J$ denotes the Jacobian of the transformation $(w,w_*)\rightarrow (w',w_*')$ and the kernels $'B,B$ define the binary interaction. Here and in the rest of the section, for notation simplicity, the explicit dependence from the time variable is omitted. 

We will consider interaction kernels of the following form 
\begin{equation}
B_{(w,w_*)\rightarrow (w',w_*')}=\lambda\Theta(\xi)\Theta(\xi_*)\chi(|w'|\le 1)\chi(|w_*'|\le 1),
\end{equation}
where $\lambda>0$ is a constant relaxation rate representing the interaction frequency. 

In order to write the collision operator $Q(\cdot,\cdot)$ in weak form we consider a test function $\psi(w)$ to get
\begin{equation}\begin{split}\label{eq:collisional_op}
 \int_I & Q(f,f)(w,c)\psi(w)dw =\\
 & \lambda \sum_{c_*=0}^{\cm}\left<\int_{I^2 } \left(\psi(w')-\psi(w)\right) f(w_*,c_*)f(w,c) dw dw_* \right>,
\end{split}\end{equation}
where the brackets $<\cdot>$ denotes the expectation with respect to the random variables $\xi,\xi_*$. Equation \eqref{eq:boltz_lin} assumes the following weak form
\begin{equation}\begin{split}\label{eq:boltz_weak}
 \dfrac{d}{dt}\int_I& f(w,c)\psi(w)dw+ \int_I \N[ f(w,c,t)]\psi(w)dw = \\
&\lambda \sum_{c_*=0}^{\cm}\left<\int_{I^2} \left(\psi(w')-\psi(w)\right)f(w_*,c_*)f(w,c)dwdw_*\right>.
\end{split}\end{equation}
An alternative form, obtained by symmetry is the following 
\begin{equation}\begin{split}\label{eq:boltz_weak2}
& \dfrac{d}{dt}\int_I f(w,c)\psi(w)dw+ \int_I \N[ f(w,c)]\psi(w)dw = \\
&\qquad\frac{\lambda}{2} \sum_{c_*=0}^{\cm}\left<\int_{I^2} \left(\psi(w')+\psi(w'_*)-\psi(w)-\psi(w_*)\right)f(w_*,c_*)f(w,c)dwdw_*
\right>.
\end{split}
\end{equation}

\subsection{Evolution of the network}\label{sec:main}

We introduced in the previous paragraph the operator $\N[\cdot]$ characterizing the evolution of the agents in the discrete space of connections. This, of course, corresponds to the evolution of the underlaying network of connections between the agents. Here we will specify the details of the model considered in the present paper, inspired by \cite{XZW}. 

The operator $\N[\cdot]$ is defined through a combination of preferential attachment and uniform processes describing the evolution of the connections of the agents by removal and adding links in the network. These processes are strictly related to the generation of stationary scale-free distributions \cite{BA}.

More precisely, for each $c=1,\ldots,\cm-1$ we define
\begin{equation}\begin{split}\label{eq:master2}
\N[f(w,c,t)] =& -\dfrac{2\Ur(f;w)}{\gamma+\beta}\left[(c+1+\beta)f(w,c+1,t)-(c+\beta)f(w,c,t)\right]\\
&-\dfrac{2\Ua(f;w)}{\gamma+\alpha}\left[(c-1+\alpha)f(w,c-1,t)-(c+\alpha)f(w,c,t)\right],
\end{split}\end{equation} 
where
$\gamma=\gamma(t)$ is the mean density of connectivity defined as 
\begin{equation}
\gamma(t) = \sum_{c=0}^{\cm}c\p(c,t),
\label{eq:gammad}
\end{equation}
$\alpha, \beta>0$ are attraction coefficients, and $\Ur(f;w)\geq 0$, $\Ua(f;w)\geq 0$ are characteristic rates of the removal and adding steps, respectively. The first term in \eqref{eq:master2} describes the net gain of $f(w,c,t)$ due to the connection removal between agents whereas the second term represents the net gain due to the connection adding process. The factor $2$ has been kept in evidence since connections are removed and created pairwise. 


At the boundary we have the following equations
\begin{equation}\label{eq:BD_master2}
\begin{split}
\N[f(w,0,t)] =&  -\dfrac{2\Ur(f;w)}{\gamma+\beta}(\beta+1)f(w,1,t)+\dfrac{2\Ua(f;w)}{\gamma+\alpha}\alpha f(w,0,t), \\
\N[f(w,\cm,t)] =& \dfrac{2\Ur(f;w)}{\gamma+\beta}(\cm+\beta) f(w,\cm,t)\\
& -\dfrac{2\Ua(f;w)}{\gamma+\alpha}(\cm-1+\alpha)f(w,\cm-1,t),
\end{split}
\end{equation}
which are derived from (\ref{eq:master2}) taking into account the fact that, in the dynamics of the network, connections cannot be removed from agents with $0$ connections and cannot be added to agents with $\cm$ connections.

\begin{remark}\label{rmk:CR}
If one defines the characteristic rates as
\be
\Ur(f;w)=U_r\frac{\gamma+\beta}{\gamma_f+\beta g(w,t)},\quad \Ua(f;w)=U_a\frac{\gamma+\alpha}{\gamma_f+\alpha g(w,t)}
\label{eq:rates}
\ee
where
\be
\gamma_f(w,t)=\sum_{c=0}^{\cm}c f(w,c,t),
\ee
and $U_a$, $U_r$ constants, the dynamics in (\ref{eq:master2}) correspond to a combination of a preferential attachment processes ($\alpha,\beta\approx 0$) and a uniform processes ($\alpha,\beta \gg 1$) for each agent with opinion $w$, with respect to the probability density of connections $f(w,c,t)/g(w,t)$.
\end{remark}

The evolution of the network of connections can be recovered taking $\psi(w)=1$ in the master equation \eqref{eq:boltz_weak}. From equation \eqref{eq:integration_wc} we have 
\begin{equation}\label{eq:Ldef2}
\dfrac{d}{dt}\p(c,t) + \int_I \N[f(w,c,t)]\,dw=0.
\end{equation}
From the above definition of the network operator $\N[\cdot]$ it follows that 
\begin{equation}
\dfrac{d}{dt}\sum_{c=0}^{\cm}\p(c,t)=0. 
\label{eq:tnc}
\end{equation}
Then for the collisional operator defined in \eqref{eq:collisional_op} and the choice of $\N[\cdot]$ in \eqref{eq:master2} the total number of agents is conserved. 

Let us take into account the evolution of the mean density of connectivity $\gamma$ defined in (\ref{eq:gammad}). We can prove that, for each $t\ge0$ 
\begin{equation}
\begin{split}
&\dfrac{d}{dt}\gamma(t)=-2\int_I\Ur(f;w)\frac{\gamma_f+\beta g(w,t)}{\gamma+\beta}dw+2\int_I\Ua(f;w)\frac{\gamma_f+\alpha g(w,t)}{\gamma+\alpha}dw\\
&\,\,\,\,\,+\dfrac{2\beta}{\gamma+\beta}\int_I \Ur(f;w) f(w,0,t)\,dw
-\dfrac{2(\cm+\alpha)}{\gamma+\alpha}\int_I \Ua(f;w) f(w,\cm,t)\,dw
\end{split}
\label{eq:mdc}
\end{equation}
Therefore $\gamma$ is not conserved in general. Asymptotically, conservation is recovered in the case $\beta=0$, if the characteristic rates are given by (\ref{eq:rates}) with $U_a=U_r$ or are constants with $V_a=V_r$, and for a sufficiently fast decay of the density function $f(w,\cm,t)$.

In Appendix \ref{app:mass}-\ref{app:mean_evo} we report the explicit computations of the conservation of the total number of connections (\ref{eq:tnc}) and of the evolution of the mean density of connectivity (\ref{eq:mdc}).

In the particular case where $\Ua$ and $\Ur$ are constants independent of $f$ and $w$ then the operator $\N[\cdot]$ is linear and will be denoted by $\L[\cdot]$. In this case, the evolution of the network of connections is independent from the opinion and we get the closed form  
\begin{equation}\label{eq:Ldef}
\dfrac{d}{dt}\p(c,t) + \L[\p(c,t)]=0,
\end{equation}
where 
\begin{equation}\begin{split}\label{eq:master}
\L[\p(c,t)] =& -\dfrac{2\Ur}{\gamma+\beta}\left[(c+1+\beta)\p(c+1,t)-(c+\beta)\p(c,t)\right]\\
&-\dfrac{2\Ua}{\gamma+\alpha}\left[(c-1+\alpha)\p(c-1,t)-(c+\alpha)\p(c,t)\right],
\end{split}\end{equation}
and at the boundary 
\begin{equation}\label{eq:BD_master}
\begin{split}
\L[\p(0,t)]& =  -\dfrac{2\Ur}{\gamma+\beta}(\beta+1)\p(1,t)+\dfrac{2\Ua}{\gamma+\alpha}\alpha \p(0,t), \\
\L[\p(\cm,t)]& = \dfrac{2\Ur}{\gamma+\beta}(\cm+\beta) \p(\cm,t)-\dfrac{2\Ua}{\gamma+\alpha}(\cm-1+\alpha)\p(\cm-1,t).
\end{split}
\end{equation}
Note that the dynamics in (\ref{eq:master}) corresponds again to a combination of preferential attachment processes ($\alpha,\beta\approx 0$) and uniform processes ($\alpha,\beta \gg 1$) with respect to the probability density of connections $\rho(c,t)$. 


Concerning the large time behavior of the network of connections, in the linear case with $\Ur=\Ua$, $\beta=0$ and now denoting by $\gamma$ the asymptotic value of the density of connectivity, it is possible to prove that  
\begin{proposition}
\label{pr:1}
For each $c\in {\mathcal C}$ the stationary solution to (\ref{eq:Ldef}) or equivalently 
\be\label{prop:st}
(c+1)\ps(c+1) = \dfrac{1}{\gamma+\alpha}\left[(c(2\gamma+\alpha)+\gamma\alpha)\ps(c)-\gamma(c-1+\alpha)\ps(c-1)\right]
\ee
is given by 
\be\label{prop:sol}
\ps(c) = \left(\dfrac{\gamma}{\gamma+\alpha}\right)^c \dfrac{1}{c!}\alpha(\alpha+1)\cdots (\alpha+c-1)\ps(0)
\ee
where
\be
\ps(0) = \left(\dfrac{\alpha}{\alpha+\gamma}\right)^{\alpha}.
\ee
\end{proposition}
Detailed computations are given in Appendix \ref{app:st}.

Further approximations are possible in the cases $\alpha\gg 1$ and $\alpha\approx 0$. For big values of $\alpha$ the preferential attachment process described by the master equation \eqref{eq:master} is destroyed and the network approaches to a random network, whose degree distribution is the Poisson distribution. In fact, in the limit $\alpha\rightarrow+\infty$ we have $(\alpha+\gamma)^c\approx \alpha(\alpha+1)\cdots(\alpha+c-1)$ and
\[
\ps(c) =\lim_{\alpha\rightarrow+\infty} \left(1+\dfrac{\gamma}{\alpha}\right)^{-\alpha}\gamma^c = \dfrac{e^{-c}}{c!}\gamma^c.
\]
In the second case, for $\gamma\ge1$ and small values of $\alpha$, the distribution can be correctly approximated with a truncated power-law with unitary exponent
\[
\ps(c) = \left(\dfrac{\alpha}{\gamma}\right)^{\alpha}\dfrac{\alpha}{c}.
\]

\begin{figure}
\centering
\includegraphics[scale=0.319]{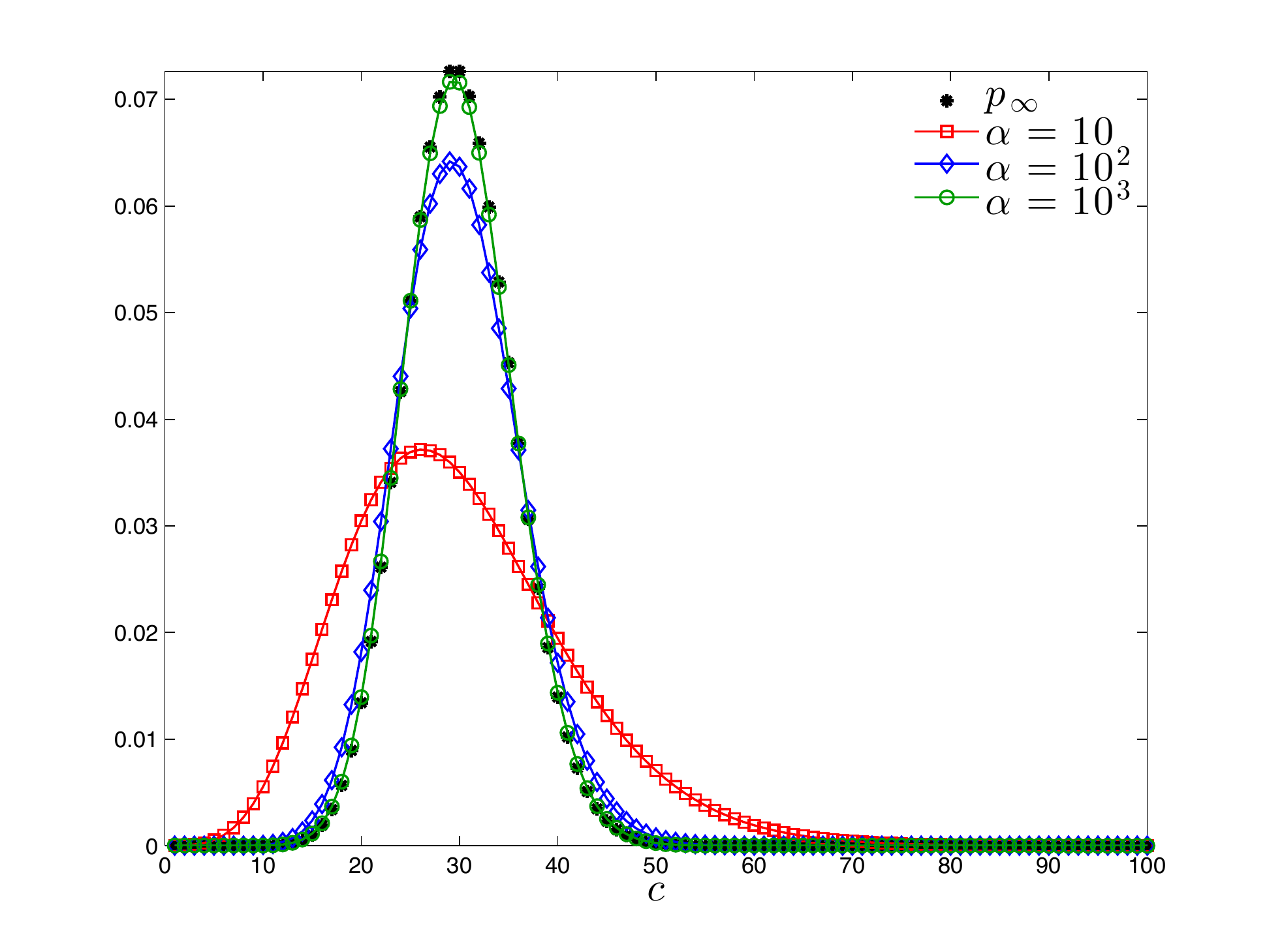}
\includegraphics[scale=0.319]{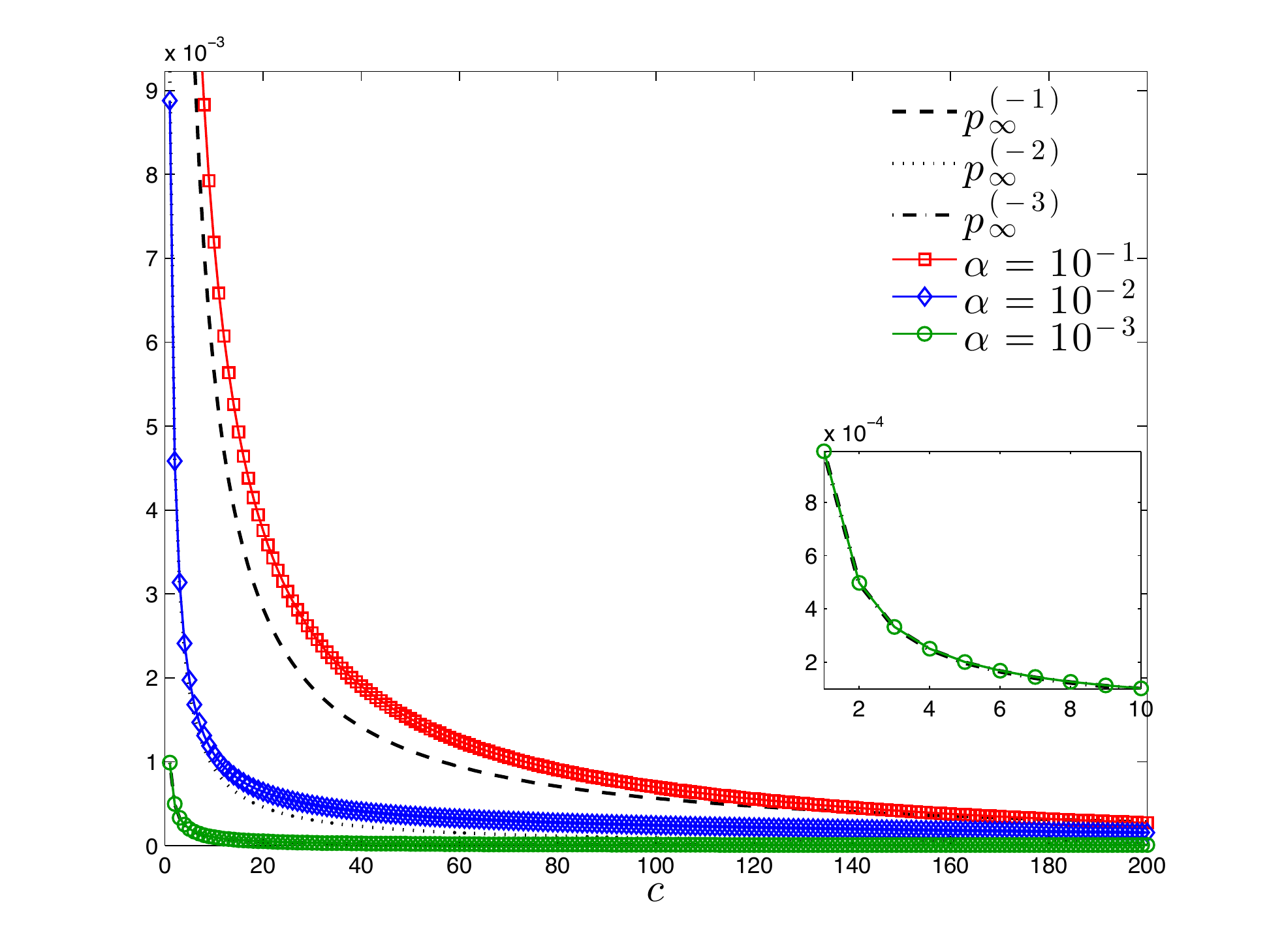}
\caption{Stationary states of \eqref{eq:Ldef} with relaxation coefficients  $\Ur=\Ua=1$, mean density of connectivity $\gamma=30$, $\cm=1500$ and several values of the attraction parameters $\alpha$, and having fixed $\beta = 0$. Left: convergence toward the Poisson distribution for big values of $\alpha$. Right: convergence toward a power-law distribution in the limit $\alpha\rightarrow 0$, we indicated with $p_{\infty}^{(-k)}, k=1,2,3$ the $\alpha-$dependent stationary solutions for $\alpha=10^{-1},10^{-2},10^{-3}$, respectively. 
}
\end{figure}

\subsection{Evolution of the moments}
\label{secmac}
In order to study the evolution of the mean opinion, i.e.
\[
m_w (c,t) =  \int_I w f(w,c,t)dw
\]
we consider $\psi(w)=w$ in \eqref{eq:boltz_weak2} 
\begin{equation}\begin{split}
 \dfrac{d}{dt}\int_I& wf(w,c,t)dw+ \int_I w N\left[f(w,c,t)\right]dw = \\
&\dfrac{\lambda}{2} \sum_{c_*=0}^{\cm}\left<\int_{I^2} \left(w'+w_*'-w-w_*\right)f(w_*,c_*,t)f(w,c,t)dwdw_*\right>.
\end{split}\end{equation}
We obtain 
\begin{equation*}\begin{split}
\dfrac{d}{dt}&m_w(c,t)+\int_I w N\left[f(w,c,t)\right]dw=\\
&\dfrac{\eta\lambda}{2}\sum_{c_*=0}^{\cm}\int_{I^2}(w-w_*)\left[P(w_*,w;c_*,c)-P(w,w_*;c,c_*)\right]f(w_*,c_*,t)f(w,c,t)dw_*dw.
\end{split}\end{equation*}
Of course, if the compromise function $P(\cdot,\cdot;\cdot,\cdot)$ is symmetric with respect to the pairs $(w,w_*)$ and $(c,c_*)$ we have conservation of the overall opinion on the network
\[
\frac{d}{dt}\sum_{c=0}^{\cm} m_w(c,t)=0.
\] 
In addition, if the network operator is linear the evolution of the mean opinion obeys the same closed differential equation of the network of connections 
\begin{equation}
\label{eq:mom}
\dfrac{d}{dt}m_w(c,t)+\L[m_w(c,t)]=0,
\end{equation}
and therefore all the conclusions of the previous section hold true also for the mean opinion on the network. 

More generally we will consider compromise functions $P(\cdot,\cdot;\cdot,\cdot)$ with the following form
\begin{equation}\label{eq:factorHK}
P(w,w_*;c,c_*)=H(w,w_*)K(c,c_*),
\end{equation}
where $0\le H(\cdot,\cdot)\le 1$ represents the positive compromise propensity and $0\le K(\cdot,\cdot)\le 1$ a function taking into account the influence of number connections in the opinion exchange process. Note that, in this case, even if we consider a symmetric compromise function $H$ and a linear network operator we have
\begin{equation}\begin{split}
\dfrac{d}{dt}&m_w(c,t)+\L[m_w(c,t)]=\\
& \dfrac{\eta\lambda}{2}\sum_{c_*=0}^{\cm} B(t,c,c_*)\left[K(c_*,c)-K(c,c_*)\right]\\
& \qquad\quad \int_{I^2}H(w_*,w)(w-w_*)f(w_*,c_*,t)f(w,c,t)dw_*dw,
\end{split}\end{equation}
and the evolution of the mean opinion cannot be expressed in closed form due to the influence of the different connections that the agents possess. This is a fundamental difference compared to classical kinetic models of opinion \cite{T}. 

In the case of the second moment of the opinion $\phi(w)=w^2$, if we assume a symmetric function $P$, by denoting
\[
E_w (c,t) =  \int_I w^2 f(w,c,t)dw
\]
we get 
\begin{equation}\begin{split}
\dfrac{d}{dt}&E_w(c,t)+\int_I w^2 \N[f(w,c,t)]\,dw=\\
&{\eta\lambda}\sum_{c_*=0}^{\cm}\int_{I^2}P(w_*,w;c_*,c)^2(w-w_*)^2 f(w_*,c_*,t)f(w,c,t)dw_*dw\\
&+\lambda\varsigma^2\int_I D^2(c,w)f(w,c,t) dw,
\end{split}\end{equation}
which, in the case of a linear operator $\L[\cdot]$ with $P=1$ and in absence of noise $D=0$, simplifies to
\begin{equation}
\label{eq:ene}\begin{split}
\dfrac{d}{dt}&E_w(c,t)+\L[E_w(c,t)]=\\
&{\eta\lambda}\left(E_w(c,t)+\p(c,t)\sum_{c_*=0}^{\cm}E_w(c_*,t)-2m_w(c,t)\sum_{c_*=0}^{\cm}m_w(c_*,t)\right).
\end{split}\end{equation}
Equation (\ref{eq:ene}) together with (\ref{eq:Ldef}) and (\ref{eq:mom}) form a closed system for the evaluation of the second order moment of the opinion.

\section{Fokker-Planck modelling}
\label{sec:FP}
In general it is difficult to obtain analytic results on the large time behavior of the opinion for the kinetic equation introduced in the previous section. A step toward the simplification of the analysis is the derivation of asymptotic states of the Boltzmann-type equation derived from a simplified Fokker-Planck-type models \cite{PTa}. Here we recall briefly the approach usually referred to as the quasi-invariant opinion limit \cite{APZa,BrTo,T}. 

\subsection{Derivation of the model}
The idea is to rescale the interaction frequency $\lambda$, the interaction propensity $\eta$ and the diffusion variance $\varsigma^2$ at the same time, in order to to maintain asymptotically the memory of the microscopic interactions. Let su introduce the scaling parameter $\varepsilon>0$ and consider the scaling 
\begin{equation}\label{eq:scaling}
\eta=\varepsilon, \qquad \lambda=\dfrac{1}{\varepsilon}, \qquad \varsigma^2=\varepsilon\sigma^2.
\end{equation}
The above scaling corresponds to the case where the interaction kernel concentrates on binary interactions producing very small changes in the agents' opinion but at the same time the number of interactions becomes very large. From a modelling point of view we require that the scaling \eqref{eq:scaling} preserves the macroscopic properties of the kinetic system in the limit $\varepsilon\rightarrow 0$, namely the evolution of the mean and the variance of the opinion derived in Section \ref{secmac}. 

The scaled equation \eqref{eq:boltz_weak} reads 
\begin{equation}\begin{split}\label{eq:scaled1}
\dfrac{d}{dt}\int_{I}&f(w,c,t)\psi(w)dw+\int_I N\left[f(w,c,t)\right]\psi(w)\,dw = \\
& \dfrac{1}{\varepsilon}\sum_{c_*=0}^{\cm}\left< \int_{I^2}(\psi(w')-\psi(w))f(w_*,c_*,t)f(w,c,t)dwdw_*\right>,
\end{split}\end{equation}
{with scaled binary interactions given by 
\begin{equation}
w' -w= \varepsilon P(w,w_*;c,c_*)(w_*-w)+\xi_{\varepsilon}D(w)+O(\varepsilon^2),
\end{equation}
where $\xi_{\epsilon}$ is a centered random variable with variance $\varepsilon \sigma^2$.}
Since as $\varepsilon\to 0$ we have $w'\to w$ we can consider the Taylor expansion of $\psi$ around $w$ to get
\begin{equation}
\psi(w')-\psi(w)=(w'-w)\psi'(w)+\dfrac{1}{2}(w'-w)^2\psi''(\bar w),
\end{equation}
where for some $\theta\in [0,1]$
\[
\bar w=\theta w+(1-\theta)w.
\]
{Inserting this expansion in the binary interaction term of \eqref{eq:scaled1} we obtain
\begin{equation}\begin{split}
\dfrac{1}{\varepsilon}\sum_{c_*=0}^{\cm}\Bigg< \int_{I^2}& (w'-w)\psi'(w)+\dfrac{1}{2}(w'-w)^2\psi''( w) \\
& f(w_*,c_*,t)f(w,c,t)dwdw_*\Bigg> + R(\varepsilon),
\end{split}\end{equation}
where $R(\varepsilon)$ indicates the remainder, given by
\begin{equation}\label{eq:rest}
R(\varepsilon) = \dfrac{1}{2\varepsilon}\sum_{c_*=0}^{\cm}\Bigg< \int_{I^2} (w'-w)^2(\psi''( \bar w)-\psi''( w))f(w_*,c_*,t)f(w,c,t)dwdw_* \Bigg>.
\end{equation}
Therefore, the scaled binary interaction term reads 
\begin{equation}\begin{split}
\sum_{c_*=0}^{\cm}  \int_{I^2}& \Bigg[P(w,w_*;c,c_*)(w_*-w)\psi’(w) \\
&+\dfrac{\sigma^2}{2}D(w,c)^2 \psi''( w)\Bigg] f(w_*,c_*,t)f(w,c,t)dwdw_*  +R(\varepsilon)+O(\varepsilon).
\end{split}\end{equation}
By similar arguments of \cite{T} it can be shown rigorously that $R(\varepsilon)$ in \eqref{eq:rest} decays to zero in the limit $\varepsilon\rightarrow 0$.}
Thus, as $\varepsilon\rightarrow 0$ we recover
\begin{equation}\begin{split}\label{eq:scaled2}
\dfrac{d}{dt}\int_{I}&f(w,c,t)\psi(w)dw+\int_I N\left[f(w,c,t)\right]\psi(w) dw = \\
&\sum_{c_*=0}^{\cm} \left[\int_{I^2}P(w,w_*;c,c_*)(w_*-w)\psi'(w)f(w_*,c_*,t)f(w,c,t)dw_*dw \right. \\
& \left.+\dfrac{\sigma^2}{2}\int_I D(w,c)^2\psi^{\prime\prime}(w)f(w,c,t)dw\right].
\end{split}\end{equation}
Integrating backwards by parts equation \eqref{eq:scaled2} we obtain the following Fokker-Planck differential equation for the evolution of the opinions' distribution through the evolving network 
\begin{equation}\begin{split}\label{eq:FP}
\dfrac{\partial}{\partial t}f(w,c,t)+N\left[f(w,c,t)\right]=\dfrac{\partial }{\partial w}\mathcal{P}[f]f(w,c,t)+\dfrac{\sigma^2}{2}\dfrac{\partial^2}{\partial w^2}(D(w,c)^2 f(w,c,t))
\end{split}\end{equation}
where
\begin{equation}\label{eq:K}
\mathcal{P}[f](w,c,t) = \sum_{c_*=0}^{\cm}\int_I P(w,w_*;c,c_*)(w_*-w)f(w_*,c_*,t)dw_*.
\end{equation}

\subsection{Stationary solutions}
In this section we will show how in some cases it is possible to compute explicitly steady state solutions of the Fokker-Planck system \eqref{eq:FP}. We restrict to linear operators $\L[\cdot]$ and asymptotic solutions of the following form 
\begin{equation}\label{eq:steady}
f_{\infty}(w,c)=g_{\infty}(w)\ps(c),
\end{equation} 
where $\ps(c)$ is the steady state distribution of the connections (see Proposition \ref{pr:1}) and 
\begin{equation}
\int_I f_{\infty}(w,c)dw=\ps(c), \qquad \qquad \sum_{c=0}^{\cm}f_{\infty}(w,c)=g_{\infty}(w).
\end{equation}
From the definition of the linear operator $\L[\cdot]$ we have $\L[\ps(c)]=0$, so stationary solutions of type \eqref{eq:steady} satisfy the following equation
\begin{equation}
\label{eq:steadye}
\dfrac{\partial }{\partial w}\mathcal{P}[f_{\infty}]f_{\infty}(w,c)+\dfrac{\sigma^2}{2}\dfrac{\partial^2}{\partial w^2}(D(w,c)^2 f_{\infty}(w,c))=0.
\end{equation}
Under some simplifications we can analytically solve equation \eqref{eq:steadye} as it has been shown in \cite{APZa,T}. If we assume \eqref{eq:factorHK}, i.e. $P(w,w_*;c,c_*)=H(w,w_*)K(c,c_*)$, the operator $\mathcal{P}[f_\infty]$ can be written as follows 
\begin{equation}\label{eq:K2}
\begin{split}
\mathcal{P}[f_\infty](w,c)& =  \left(\sum_{c_*=0}^{\cm}{K}(c,c_*)\ps(c_*)\right)\left(\int_I H(w,w_*)(w_*-w)g_\infty(w_*)dw_*\right)\\
&=:\mathcal{K}[\ps](c)\mathcal{H}[g_\infty](w),
\end{split}
\end{equation}
and if we further assume that $K(c,c_*)={\bar K}(c_*)$ is independent of $c$, and $H(w,w_*)={\bar H}(w)$ independent of $w_*$
we have 
\[
\mathcal{K}[\ps] = \sum_{c_*=0}^{\cm}{\bar K}(c_*)\ps(c_*) =: \kappa,\quad \mathcal{H}[g_\infty]=\bar{H}(w)\left( w-\bar m_w\right),
\] 
where $\bar m_w=\sum_{c=0}^{\cm} m_w(c,t)$.

Finally, taking $D(w,c)=D(w)$ independent of $c$, equation  \eqref{eq:steadye} reads
\begin{equation}\label{eq:steady2}
\left(\kappa\dfrac{\partial }{\partial w}\bar{H}(w)\left( w-\bar m_w\right) g_{\infty}(w)+\dfrac{\sigma^2}{2}\dfrac{\partial^2 }{\partial w^2}D(w)^2g_{\infty}(w)\right)\ps(c)=0.
\end{equation}
Therefore, on the support of $\ps(c)$, 
stationary solutions can be derived from the following equation
\begin{equation}\label{eq:st_g3}
\kappa\bar{H}(w)\left(w-\bar m_w\right)g_{\infty}(w)+\dfrac{\sigma^2}{2}\dfrac{\partial }{\partial w}D(w)^2g_{\infty}(w)=0,
\end{equation}
which corresponds to the solution of the following ordinary differential equation
\begin{equation}\label{eq:st_g4}
\dfrac{d g_{\infty}}{d w}=2\left(\dfrac{\kappa}{\sigma^2}\frac{\bar{H}(w-\bar m_w)}{D^2}-\frac{D'}{D}\right)g_{\infty},
\end{equation}
thus
\begin{equation}
g_{\infty}(w)= \frac{C_0}{D(w)^2} \exp\left\{\dfrac{2\kappa}{\sigma^2}\int^w\frac{\bar{H}(v)}{D(v)^2}(\bar m_w-v)\,dv\right\},
\end{equation}
where the constant $C_0$ is chosen such that the total mass of $g_{\infty}$ is equal to one. 

Some explicit examples are give below.
\begin{enumerate}
 \item In the case $H\equiv 1$ and $D(w)=1-w^2$ the steady state solution $g_{\infty}$ is given by 
\begin{equation}\label{eq:stat_1}
g_{\infty}(w)= C_0(1+w)^{-2+\bar m_w\kappa/\sigma^2}(1-w)^{-2-m_w\kappa/\sigma^2}\exp\Big\{-\dfrac{\kappa(1-\bar m_w w)}{\sigma^2(1-w^2)}\Big\},
\end{equation}
\item For $H(w,w_*)=1-w^2$ and $D(w)=1-w^2$ the steady state solution $g_{\infty}$ is given by 
\begin{equation}\label{eq:stat_2}
g_{\infty}(w)=C_0(1-w)^{-2+(1-\bar m_w){\kappa}/{\sigma^2}}(1+w)^{-2+(1+\bar m_w){\kappa}/{\sigma^2}},
\end{equation}
\end{enumerate}
In Figure \ref{Fig:stat1} as an example we report the stationary solution $f_\infty(w,c) = g_\infty(w)\rho_\infty(c)$, where $ g_\infty(w)$ is given by \eqref{eq:stat_1} with $\kappa =1$, $m_w = 0$, $\sigma^2 = 0.05$ and $p_\infty(c)$ defined by \eqref{prop:sol}, with $\Ur=\Ua=1$, $\gamma = 30$ and $\alpha = 10$ on the left and $\alpha = 0.01$ on the right.

\begin{figure}[t]
\centering
\includegraphics[scale= 0.3]{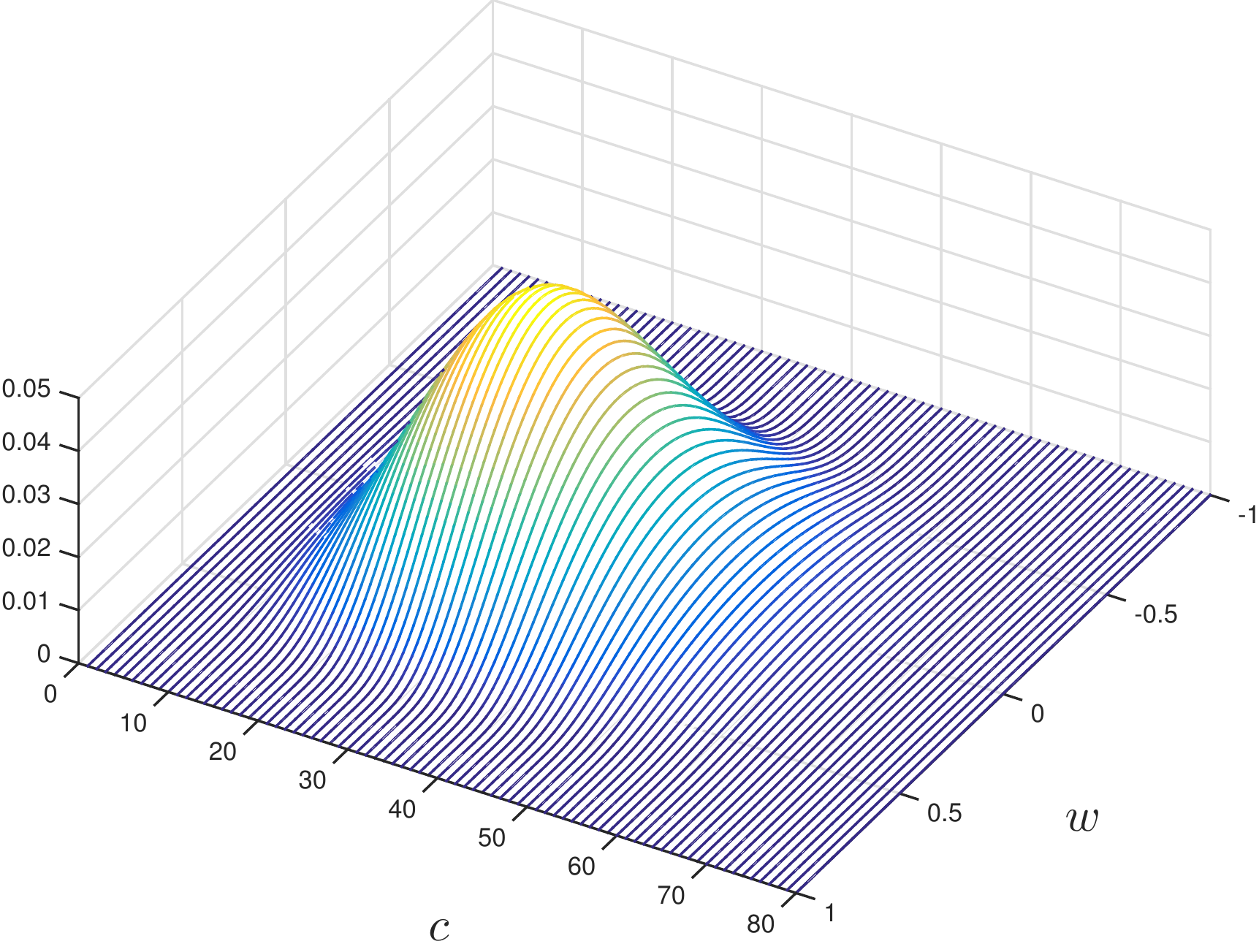}
\hskip +1cm
\includegraphics[scale= 0.3]{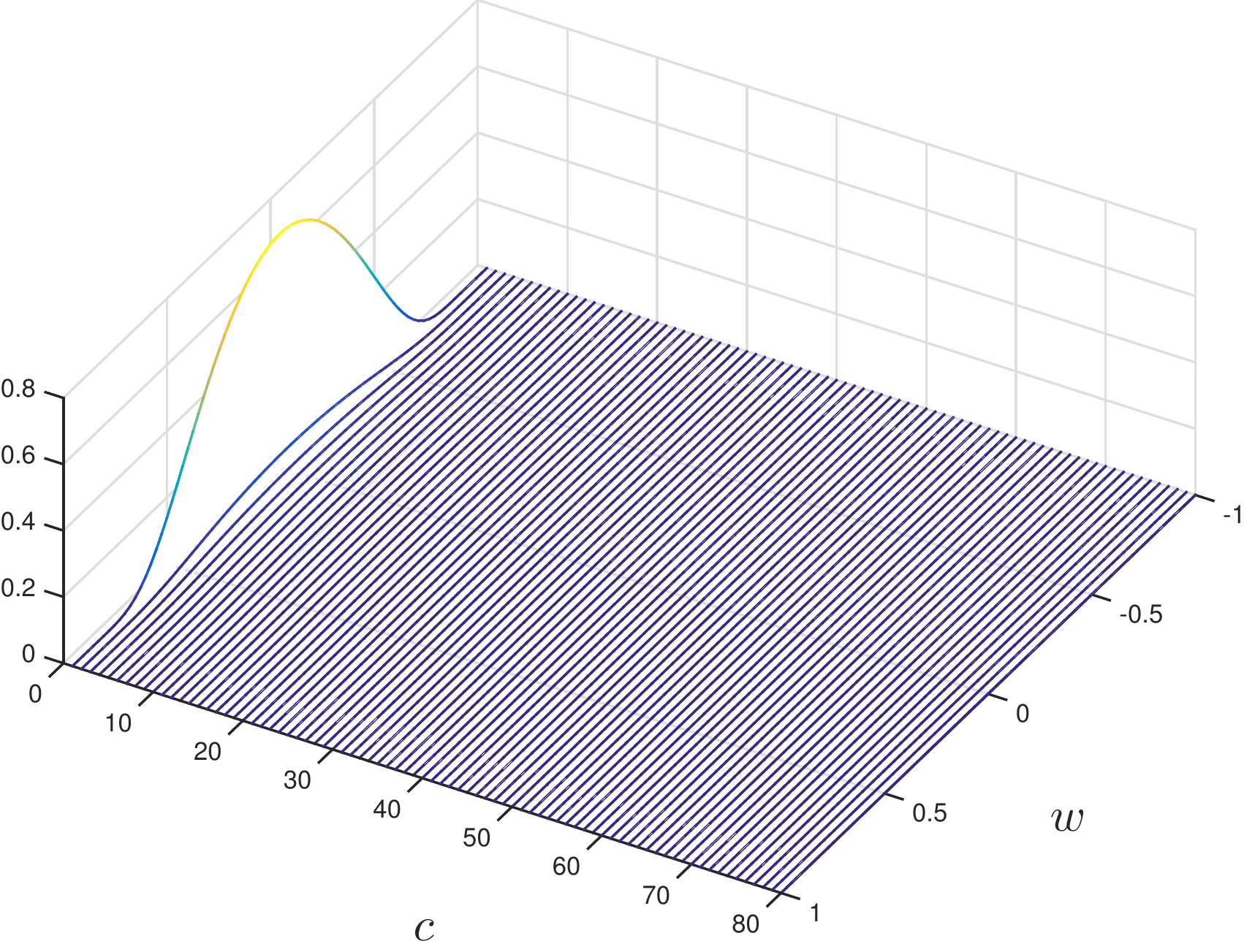}
\caption{Stationary solutions of type $f_\infty(w,c) = g_\infty(w)p_\infty(c)$, where $ g_\infty(w)$ is given by \eqref{eq:stat_1} with $\kappa =1$, $m_w = 0$, $\sigma^2 = 0.05$ and $p_\infty(c)$ defined by \eqref{prop:sol}, with $\Ur=\Ua=1$, and $\alpha = 10$ on the left and $\alpha = 0.1$ on the right.
 }\label{Fig:stat1}
\end{figure}

\section{Numerical methods}
In this section we consider the development of numerical methods for the kinetic models studied in the previous sections. First we consider direct simulation Monte Carlo methods for the Boltzmann model (\ref{eq:Bo}) introduced in Section \ref{sec:model}. Here the major difficulty is to consider a probabilistic interpretation of the dynamics induced by the network operator $\N[\cdot]$, whereas the opinion interaction follows the standard binary sampling approach (see \cite{PTa} for details). Next we consider the derivation of numerical schemes for the Fokker-Planck model (\ref{eq:FP}) derived in Section \ref{sec:FP}. In particular we will focus on the construction of finite-difference methods which are capable to describe correctly the large time behavior of the model. To this aim we will consider a nonlinear version of the Chang-Cooper type discretization which has the nice feature of preserving the steady states and the non negativity of the numerical solution \cite{CC, LLPS}. 
\subsection{Direct simulation Monte Carlo}
One of the most common approaches to solve Boltzmann-type equations is based Monte Carlo methods. Let us consider the initial value problem given by equation  \eqref{eq:boltz_lin} with initial condition $f(w,c,t=0)=f_0(w,c)$, the solution at time $t^n=n\cdot\Delta t^n, n\ge1$ is obtained as a composition of the solutions of the following problems: we first integrate the network term for all $c\in\mathcal{C}$ along the time interval $[t^{n},t^{n+1}]$
\begin{equation}\begin{cases}\vspace{0.5em}
\dfrac{d}{dt}\tilde f(w,c,t) + \N[\tilde f(w,c,t)]=0, \\
\tilde f(w,c,0)=f_0(w,c)
\end{cases}\end{equation}
then we solve the interaction step 
\begin{equation}\begin{cases}\vspace{0.5em}\label{eq:Coll}
\dfrac{d}{dt}f(w,c,t) = Q(f,f)(w,c,t), \\
f(w,c,0)  = \tilde f(w,c,t^n).
\end{cases}\end{equation}
The described process may be iterated in order to obtain the numerical solution of the initial equation at each time step. At variance with standard Monte Carlo methods for opinion dynamics, see for example \cite{PTa}, here we face the additional difficulty of the network evolution. In the sequel we describe the details of the Monte Carlo method for the network evolution in the simplified case $\N[\cdot]=\L[\cdot]$.

Let $f^n = f(w,c,t^n)$ the empirical density function for the density of agents at time $t^n$ with opinion $w\in[-1,1]$ and connections $c\in\mathcal{C}$.  For a any given opinion $w$ the solution of the transport step is given for each $c>0$ and $c<\cm$ by
\begin{equation}\label{eq:MCmaster}
\begin{split}
f^{n+1}(w,c) =&\left(1-\Delta t\dfrac{{\Ur}(c+\beta)}{\gamma^n+\beta}-\Delta t\dfrac{ {\Ua}(c+\alpha)}{\gamma^n+\alpha}\right)f^n(w,c)\\
&+\Delta t  \dfrac{{\Ur}(c+\beta)}{\gamma^n+\beta}f^n(w,c-1)+\Delta t\dfrac{ {\Ua}(c+\alpha)}{\gamma^n+\alpha}f^n(w,c+1),
\end{split}\end{equation}
with boundary conditions 
\begin{equation}\begin{split}
f^n(w,0) & = \left(1-\Delta t\dfrac{ {\Ua}(c+\alpha)}{\gamma^n+\alpha}\right)f^n(w,0)+\Delta t\dfrac{ {\Ua}(c+\alpha)}{\gamma^n+\alpha}f^n(w,1), \\
f^n(w,\cm) & = \left(1-\Delta t\dfrac{{\Ur}(c+\beta)}{\gamma^n+\beta}\right)f^n(w,\cm)+\Delta t  \dfrac{{\Ur}(\cm+\beta)}{\gamma^n+\beta}f^n(w,\cm-1),
\end{split}\end{equation}
and temporal discretization such that 
\begin{equation}\label{eq:timeL}
\Delta t\le \min \left\{ \dfrac{\gamma^n+\beta}{{\Ur}(\cm+\beta)},\dfrac{\gamma^n+\alpha}{{\Ua}(\cm+\alpha)} \right\}.
\end{equation}
An algorithm to simulate the above equation reads as follows
\begin{alg}[]\label{MCmaster}~
  \begin{enumerate}
  \item Sample $(w^0_i,c^0_i)$, with $i=1,\ldots,N_s$, from the distribution $f^{0}(w,c)$.
  \item \texttt{for} $n=0$ \texttt{to} $n_{tot}-1$  
  \begin{enumerate}
  \item Compute   $\gamma^n =\frac{1}{N_s} \sum_{j=1}^{N_s}c^n_j$;
  \item fix $\Delta t$ such that condition \eqref{eq:timeL} is satisfied.
   \item \texttt{for} $k=1$ \texttt{to} $N_s$  
 	 \begin{enumerate}
	 \item compute the  following probabilities rates 
	 \[ p_k^{(a)} =\frac{\Delta t V_a(c_k^n+\alpha)}{\gamma^n+\alpha},\qquad p_k^{(r)} = \frac{\Delta t V_r(c_k^n+\beta)}{\gamma^n+\beta},\]
	 \item set $c^*_k =c_k^n$.
	 \item \texttt{if} $0 \leq c_k^*\leq \cm-1$,\\ \quad with probability $p_k^{(a)}$ add a connection: $c_k^* = c_k^* +1$; 
	  \item  \texttt{if} $1 \leq c_k^*\leq \cm$,\\ \quad  with probability $p_k^{(r)}$ remove a connection: $c_k^* = c_k^* -1$;
	 \end{enumerate}
	 \item[]\texttt{end for}
  \item set $c^{n+1}_i  = c^*_i$, for all $i= 1,\ldots, N_s$.
\end{enumerate}
\item[]\texttt{end for}
  \end{enumerate}
  \label{ANMCS}
\end{alg}\medskip

The collision step may be solved through binary interaction algorithm \cite{AP, PTa}, where the basic idea is to solve the binary exchange of information described by \eqref{eq:binary}, under the quasi-invariant opinion scaling \eqref{eq:scaling}. 

The time-discrete scheme reads
\begin{align}\label{eq:MCBoltz}
f^{n+1}(w,c) = \left(1- \frac{\Delta t}{\varepsilon} \right)f^n(w,c) + \frac{\Delta t }{\varepsilon} {Q_\varepsilon^{+}(f^n,f^n)}(w,c),
\end{align}
where we have made explicit the dependence of $Q(f,f)$ on the frequency of interactions $1/\varepsilon$ and with $Q^{+}_\varepsilon(f^n,f^n)$ we denoted the gain part, namely it accounts the density of opinions gained at position $w$ after the binary interaction  \eqref{eq:binary}.
The collisional step \eqref{eq:MCBoltz} is a convex combinations of probability density under the time step constrain $\Delta t\leq \varepsilon$, which has to be coupled with \eqref{eq:timeL}. For further details on the algorithm we refer to \cite{AP, PTa}.

In Figure \ref{Fig:stat2} we show the two stationary states, already presented in Figure \ref{Fig:stat1}, computed through the Monte Carlo procedure just described, where we use $N_s = 2\times10^4$ samples to reconstruct the density and scaling parameter $\varepsilon = 0.01$ and $\Delta t = \varepsilon$.
\begin{figure}[t]
\centering
\includegraphics[scale= 0.3]{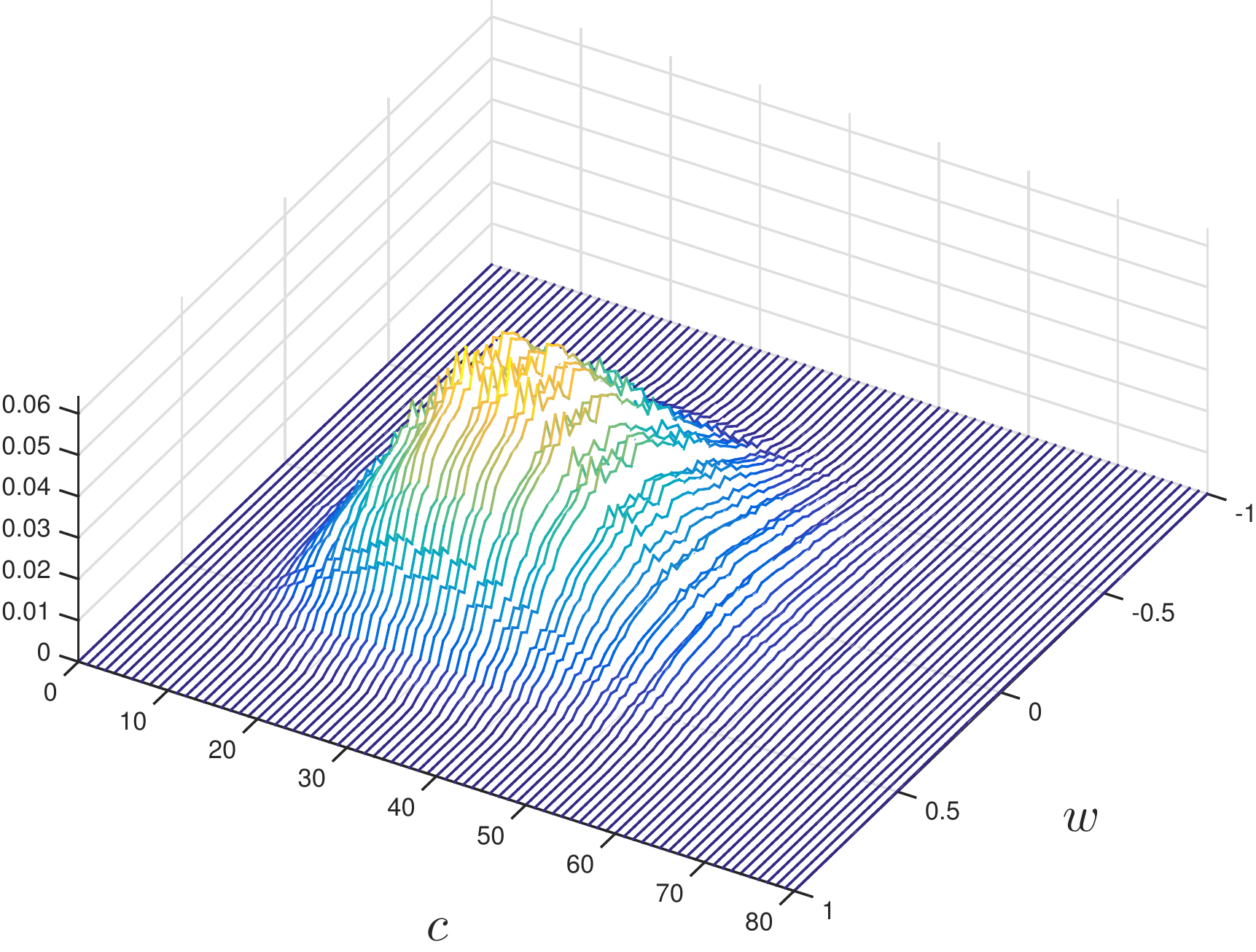}
\hskip +1cm
\includegraphics[scale= 0.3]{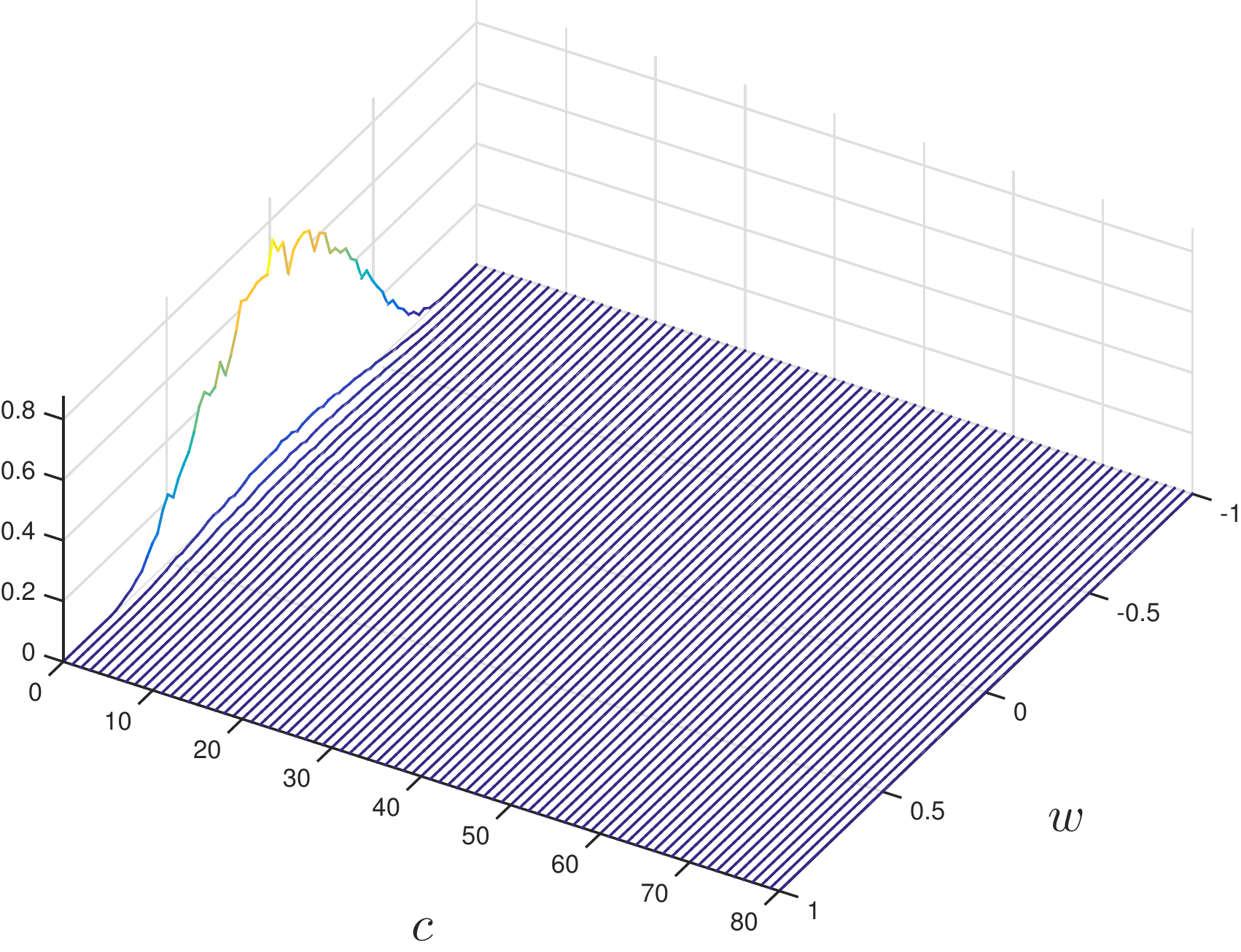}
\caption{ Stationary solutions captured via Monte Carlo simulations, with $N_s=2\times10^4$ samples. Parameters of the model are chosen as follows $\sigma^2 = 0.05$, $\Ur=\Ua=1$, $\beta = 0$, $\alpha = 10$ on the right hand side and $\alpha =0.1$ on the left hand side.
 }\label{Fig:stat2}
\end{figure}

\subsection{Chang-Cooper type numerical schemes}
We consider the Fokker-Planck system (\ref{eq:FP}) that we will rewrite in the form
\begin{equation}\begin{split}\label{eq:FP2}
\dfrac{\partial}{\partial t}f(w,c,t)+N\left[f(w,c,t)\right]=\dfrac{\partial }{\partial w}\mathcal{F}[f]
\end{split}\end{equation}
where
\begin{equation}\label{eq:F}
\mathcal{F}[f] = \left(\mathcal{P}[f] + \sigma^2 D'(w,c)D(w,c)\right)f(w,c,t)+\dfrac{\sigma^2}{2}D(w,c)^2 \dfrac{\partial}{\partial w}f(w,c,t),
\end{equation}
and $\mathcal P[f]$ is given by (\ref{eq:K}). 

The above equation is complemented with the initial data $f(w,c,0)=f_0(w,c)$ and considered in the domain $(w,c)\in I\times{\mathcal C}$ with zero flux boundary condition on $w$. Note that in the variable $c$ the equation is in discrete form and therefore the discretization we will consider acts only on the opinion variable $w$.

Let us introduce a uniform grid $w_{i}=-1+i\Delta w$, $i=0,\ldots,N$ with $\Delta w = 2/N$, we denote by $w_{i \pm 1/2}=w_i \pm \Delta w/2$ and define
\[
f_{i}(c,t)=\frac{1}{\Delta w}\int_{w_{i+1/2}}^{w_{i-1/2}} f(w,c,t)\,dw.
\]
Integrating equation (\ref{eq:FP2}) yields
\begin{equation}\begin{split}\label{eq:FPd}
\dfrac{\partial}{\partial t}f_{i}(c,t)+N\left[f_{i}(c,t)\right]=\frac{\mathcal{F}_{i+1/2}[f]-\mathcal{F}_{i-1/2} [f]}{\Delta w},
\end{split}\end{equation} 
where $\mathcal{F}_{i}[f]$ is the flux function characterizing the numerical discretization.

We assume a flux function as a combination of upwind and centered discretization as in the classical Chang-Cooper flux
\be
\begin{split}
\label{eq:flux}
\mathcal{F}_{i+1/2}[f]=&\left((1-\delta_{i+1/2})(\mathcal{P}[f_{i+1/2}] + \sigma^2 D'_{i+1/2}D_{i+1/2})+\frac{\sigma^2}{2\Delta w}D^2_{i+1/2} \right)f_{i+1}\\
&+\left(\delta_{i+1/2} (\mathcal{P}[f_{i+1/2}] + \sigma^2 D'_{i+1/2}D_{i+1/2})-\frac{\sigma^2}{2\Delta w} D^2_{i+1/2} \right)f_{i},
\end{split}
\ee  
where $D_{i+1/2}=D(w_{i+1/2},c)$ and $D'_{i+1/2}=D'(w_{i+1/2},c)$.

The weights $\delta_{i+1/2}$ have to be chosen in such a way that a steady state solution is preserved. Moreover, as it is shown in Appendix \ref{app:B}, this choice permits also to preserve nonnegativity of the numerical density.

Preservation of the steady states corresponds to assume that the numerical flux vanishes when $f$ is at the steady state $f^{\infty}$. Imposing the numerical flux equal to zero from (\ref{eq:flux}) we get
\be
\label{eq:cc}
\frac{f_{i+1}}{f_{i}}=\displaystyle  \frac{-\delta_{i+1/2} (\mathcal{P}[f_{i+1/2}] + \sigma^2 D'_{i+1/2}D_{i+1/2})+\frac{\sigma^2}{2\Delta w} D^2_{i+1/2}}{(1-\delta_{i+1/2})(\mathcal{P}[f_{i+1/2}] + \sigma^2 D'_{i+1/2} D_{i+1/2})+\frac{\sigma^2}{2\Delta w}D^2_{i+1/2}}.
\ee
Solving with respect to $\delta_{i+1/2}$ yields
\be
\delta_{i+1/2} = \frac{{\sigma^2} D^2_{i+1/2}}{2\Delta w(\mathcal{P}[f_{i+1/2}]+\sigma^2 D'_{i+1/2}D_{i+1/2})}+\frac{1}{1-f_{i}/f_{i+1}}.
\ee
On the other hand the same computation directly on the flux (\ref{eq:F}) gives the differential equation
\be
\frac{\sigma^2 D(w,c)^2}{2}\dfrac{\partial}{\partial w}f(w,c,t)=-\left(\mathcal{P}[f] + \sigma^2 D'(w,c)D(w,c)\right)f(w,c,t),
\ee 
which in general cannot be solved, except is some special cases as discussed in the previous section, due to the nonlinear term on the right hand side. A possible way to overcome this difficulty is to consider a quasi steady-state approximation as follows. We first integrate the previous equation in the cell  $[w_i,w_{i+1}]$ to get
\[
\int_{w_i}^{w_{i+1}} \left(\frac{1}{f}\dfrac{\partial}{\partial w}f\right)(w,c,t)\,dw = -
\frac{2}{\sigma^2}\int_{w_i}^{w_{i+1}} \frac{1}{D(w,c)^2}\left(\mathcal{P}[f] + \sigma^2 D'(w,c)D(w,c)\right)\,dw,
\]
and then
\[
\frac{f_{i+1}}{f_i} = \exp\left(-
\frac{2}{\sigma^2}\int_{w_i}^{w_{i+1}} \frac{1}{D(w,c)^2}\left(\mathcal{P}[f] + \sigma^2 D'(w,c)D(w,c)\right)\,dw\right).
\]
Next we can approximate the integral on the right hand side with a suitable quadrature formula. Because of singularities at the boundaries $w=\pm 1$ of the integrand function we can resort on open formula of Newton-Cotes type. For example, using the simple midpoint rule a second order approximation is obtained
\be
\label{eq:mid}
\frac{f_{i+1}}{f_i} \approx \exp\left(-
\frac{2\Delta w}{\sigma^2}\frac{1}{D^2_{i+1/2}}\left(\mathcal{P}[f_{i+1/2}] + \sigma^2 D'_{i+1/2}D_{i+1/2}\right)\right).
\ee
Now by equating (\ref{eq:mid}) and (\ref{eq:cc}) we recover the following expression of the weight functions
\be
\label{eq:weights}
\delta_{i+1/2}=\frac1{\lambda_{i+1/2}}+\frac{1}{1-\exp(\lambda_{i+1/2})}, 
\ee 
where
\be
\lambda_{i+1/2}=\frac{2\Delta w}{\sigma^2}\frac{1}{D^2_{i+1/2}}\left(\mathcal{P}[f_{i+1/2}] + \sigma^2 D'_{i+1/2}D_{i+1/2}\right).
\ee
Note that here, at variance with the standard Chang-Cooper scheme \cite{CC}, the weights depend on the solution itself as in \cite{LLPS}. Thus we have a nonlinear scheme which preserves the steady state with second order accuracy. In particular, by construction, the weight in (\ref{eq:weights}) are nonnegative functions with values in $[0,1]$. 

Higher order accuracy of the steady state can be recovered using a more general numerical flux given by 
\be
\begin{split}
\label{eq:fluxe}
\mathcal{F}_{i+1/2}[f]&=\frac{D^2_{i+1/2}}{\Delta w}\left((1-\delta_{i+1/2})\int_{w_i}^{w_{i+1}}\frac{\mathcal{P}[f] + \sigma^2 D'(w,c)D(w,c)}{D(w,c)^2}\,dw+\frac{\sigma^2}{2} \right)f_{i+1}\\
&+\frac{D^2_{i+1/2}}{\Delta w}\left(\delta_{i+1/2} \int_{w_i}^{w_{i+1}}\frac{\mathcal{P}[f] + \sigma^2 D'(w,c)D(w,c)}{D(w,c)^2}\,dw-\frac{\sigma^2}{2} \right)f_{i},
\end{split}
\ee  
and taking
\be
\lambda_{i+1/2}=\frac{2}{\sigma^2}\int_{w_i}^{w_{i+1}} \frac{1}{D(w,c)^2}\left(\mathcal{P}[f] + \sigma^2 D'(w,c)D(w,c)\right)\,dw.
\ee

\section{Numerical Results}
In this section we perform several numerical test to validate our modeling and numerical setting. We focus on the case $\alpha< 1$, since it represents the most relevant case in complex networks \cite{AB, XZW}, for this range of the parameter we have emergence of power law distributions for the network's  connectivity. 
Except for the first test case where we analyze the numerical convergence of the Boltzmann model in the quasi-invariant limit, in all the other test the opinion dynamics evolves according to \eqref{eq:FP}. The compromise function $P(c,c_*;w,w_*)$ and the local diffusion function $D(w,c)$ will be specified in the various tests.
The choice of parameters for the different tests is summarized in Table \ref{tab:par}, where other parameters are introduced additional details will be reported.
\begin{table}[htb]
\caption{Parameters in the various test cases}
\label{tab:all_parameters}
\begin{center}
\begin{tabular}{cccccccccc}
\hline
Test & $\sigma^2$ & $\sigma^2_F$ & $\sigma^2_L$ & $\cm$  & $\Ur$ &$\Ua$ &  $\gamma_0$ & $\alpha$ & $\beta$ \\
\hline
\hline
\#1 & $5\times10^{-2}$  & $6\times10^{-2}$ & $-$ & $250$   & $1$& 1 &  $30$ & $1\times10^{-1} $ & $0$ \\
\hline
\#2 & $5\times10^{-2}$  & $6\times10^{-2}$ & $-$ & $250$   & $-$& $-$ &  $30$ & $1\times10^{-1} $ & $0$\\
\hline
\#3 & $5\times10^{-3}$  & $4\times10^{-2}$ & $2.5\times10^{-2}$ & 250  & $1$ & $1$& 30 & $1\times10^{-4}$ & $0$\\
\hline
\#4 & $1\times10^{-3}$  & $-$ & $-$ & 250  & $1$ & $1$& 30 & $1\times10^{-1}$ & $0$ \\
\hline
\end{tabular}\label{tab:par}
\end{center}
\end{table}

\subsection{Test \#1}
We first consider the simple test case where the opinion evolves independently of the network and we validate the Chang-Cooper type scheme comparing its convergence with respect to the Monte-Carlo methods.

We simulate the dynamics with the linear compromise function  
$P(w,w_*;c,c_*) = 1$, and $D(w,c) = 1-w^2$, thus we can use the results \eqref{eq:stat_1}, to compare the solutions obtained through the numerical scheme with the analytical one. The other parameters of the model are reported in Table \ref{tab:all_parameters} and we define the following initial data
\begin{align}\label{eq:g0}
g_0(w)      = \frac{1}{2\sqrt{2\pi\sigma_F^2 }}(\exp\{-(w+1/2)^2/{(2\sigma_F^2)}\}+\exp\{-(w-1/2)^2/{(2\sigma_F^2)}\}).
\end{align}


In Figure \ref{Fig:MC}, on the left hand-side, we report the qualitative convergence of the Monte-Carlo methods, where we consider $N_s = 10^5$ samples to reconstruct the opinion's density, $g(w,t)$, on a grid of $N = 80$ points. The figure shows that for decreasing values of the scaling parameter $\varepsilon = \{0.5, 0.05, 0.005\}$, we have convergence to the reference solutions, \eqref{eq:stat_1} of the Fokker-Planck equation.
On the right we report the convergence to the stationary solution of the connectivity distribution, \eqref{prop:st}, for $\alpha = 0.1$ and $V = 1$ and with $\cm = 250$. In this case we show two different qualitative behaviors for an increasing number of samples $N_s=\{10^3, 10^5\}$ and for sufficient large times.
\begin{figure}[t]
\centering
\includegraphics[scale= 0.3]{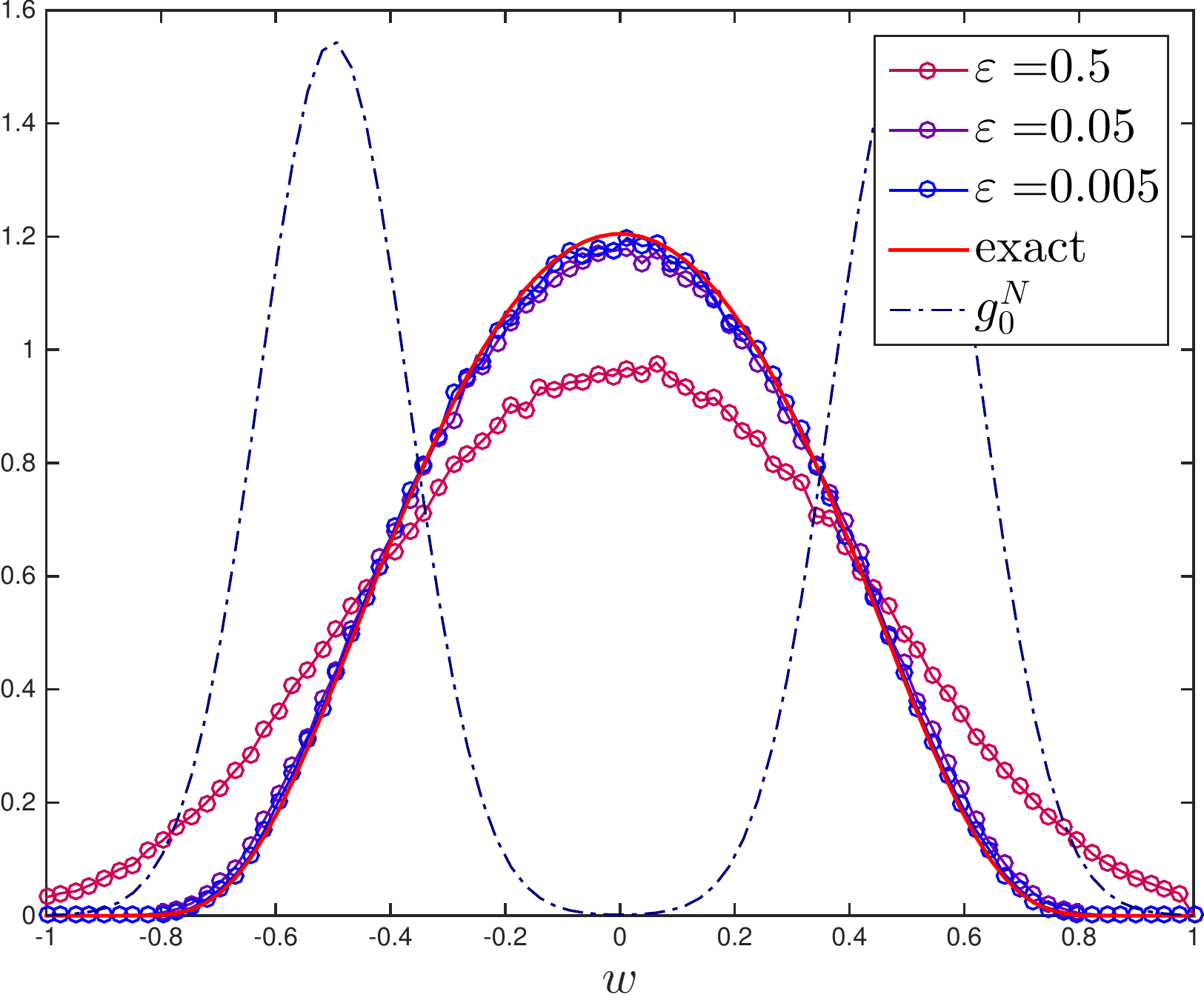}
\hskip +0.2cm
\includegraphics[scale= 0.3]{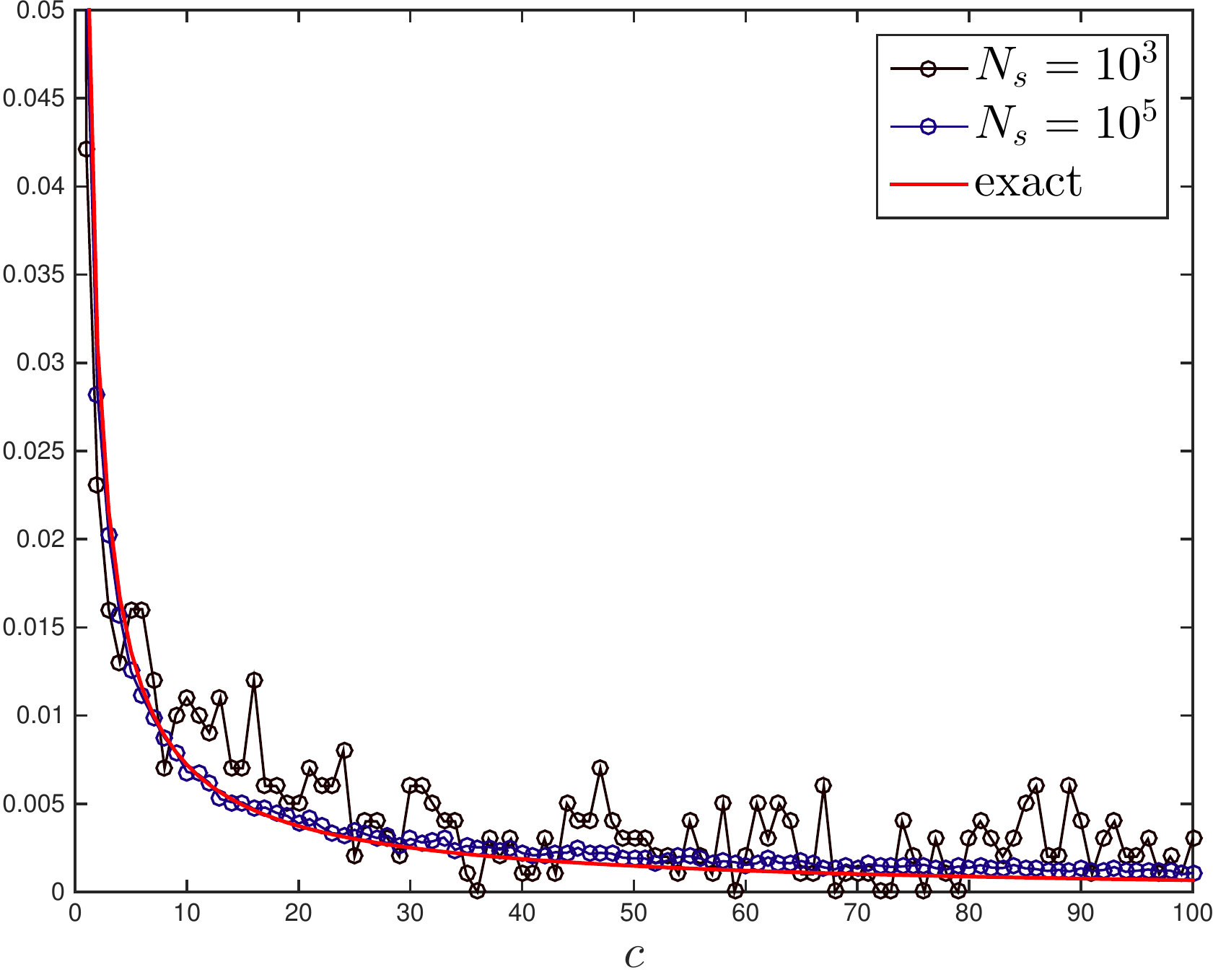}
\caption{Test \#1.
One-dimensional setting: on the left, convergence of \eqref{eq:MCBoltz} to the stationary solution \eqref{eq:stat_1}, of the Fokker-Planck equation, for decreasing values of the parameter $\varepsilon$, $g^N_0$ represent the initial distribution. On the right, convergence of the Monte-Carlo \eqref{eq:MCmaster} to the reference solution \eqref{prop:st}  for increasing value of the the number of  samples $N_s$.
 }\label{Fig:MC}
\end{figure}

In Figure \ref{Fig:CC}, on the left-hand side we report the qualitative solution of the Chang--Cooper type scheme integrated with the explicit Euler method, on the right-hand side we depict the decay of the $L_1$ relative error to the reference solution, \eqref{eq:stat_1}, i.e.
\begin{align}\label{eq:L1err}
\frac{\|{g}^N - g_\infty\|_1}{\|g_\infty\|_1},
\end{align}
with $g^N$ representing the approximated solution of the numerical scheme.
We test the scheme's convergence for different quadrature rules, and additionally we compared  them with the error of the Monte-Carlo simulation.
The right-hand plot shows how the Change Cooper scheme is able to reach high order of accuracy, for high order quadrature rules in \eqref{eq:fluxe}.
For the discretization we consider the following parameters: we account $N=80$ and we define $\Delta w = 2/N$ and time step $\Delta t = (\Delta w)^2/(4\sigma^2)$,  for a final time of $T = 10$. The Binary Interaction algorithm is performed with $\varepsilon = 0.0005$ and with a number of sample $N_s = 10^5$.
\begin{figure}[t]
\centering
\includegraphics[scale= 0.3]{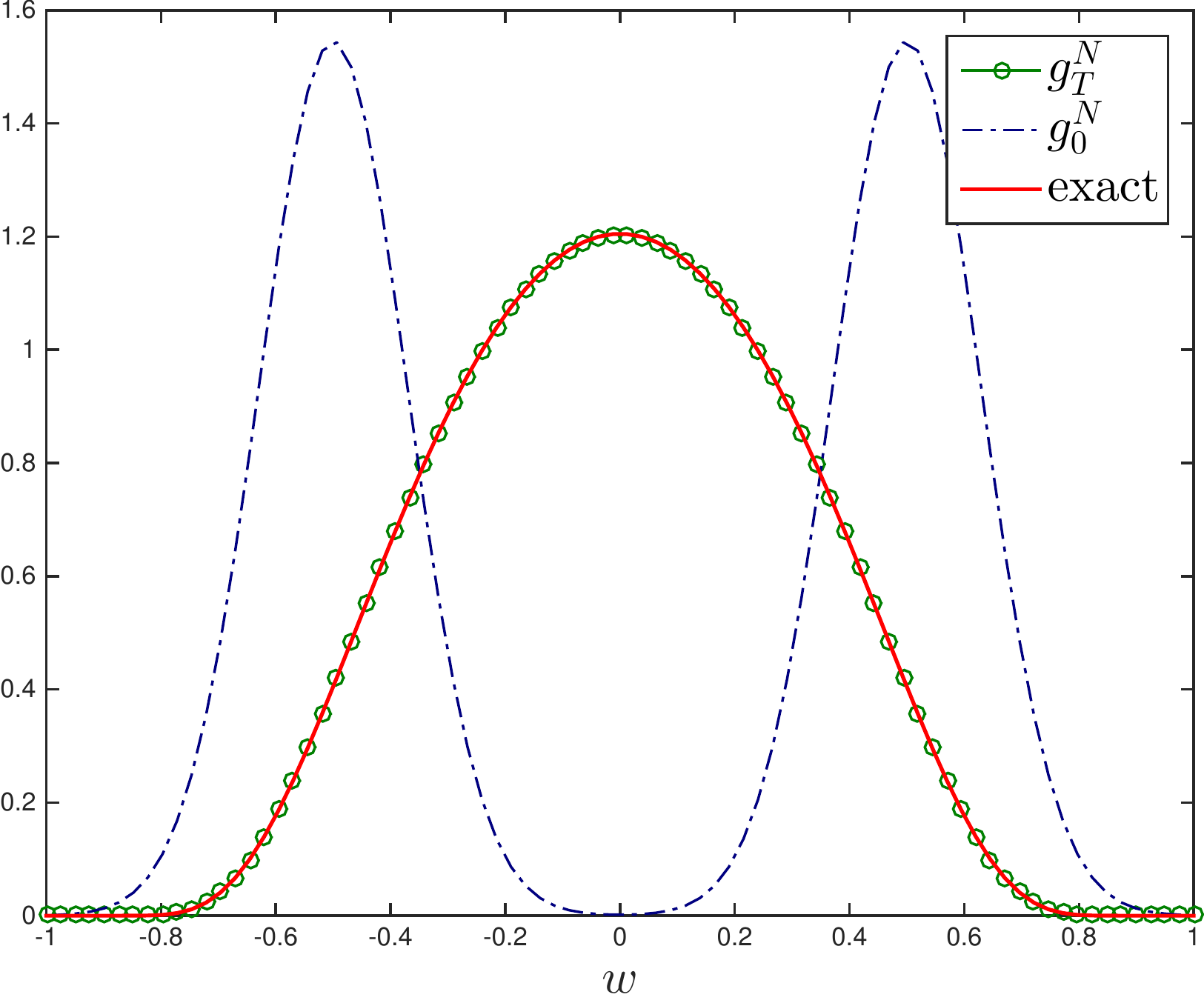}
\hskip +0.2cm
\includegraphics[scale= 0.3]{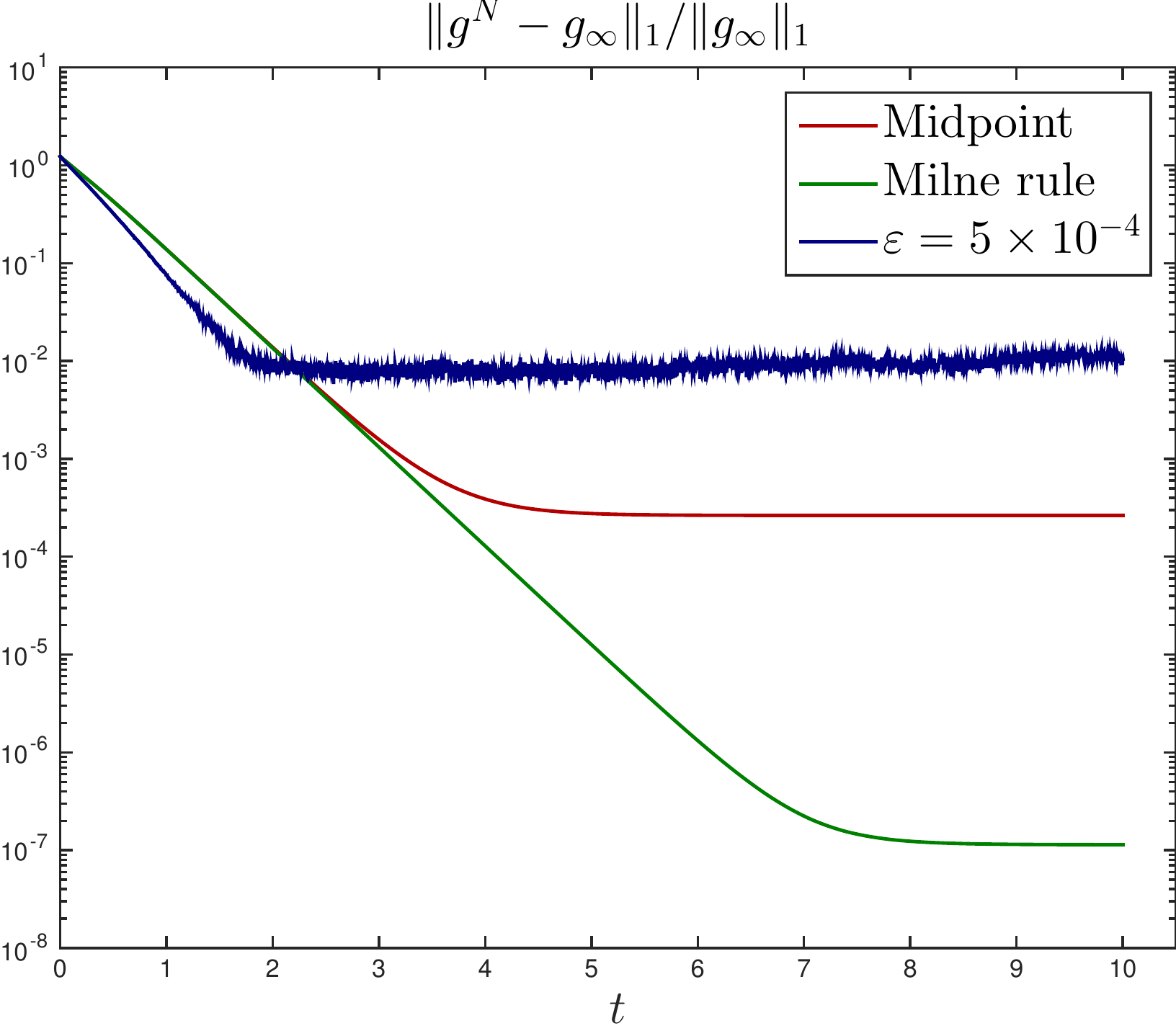}
\caption{Test \#1.
One-dimensional setting: on the left, The solution of the Chang--Cooper type scheme, is indicated with $g^N_T$ and compared with the stationary solution \eqref{eq:stat_1}, also the initial data $g^N_0$, \eqref{eq:g0}, is reported. On the right, decay of the $L^1$ relative error, \eqref{eq:L1err}, for the different choice of the quadrature rule, mid--point rule, \eqref{eq:mid}, and Milne's rule, (respectively of $2^{nd}$ and $4^{th}$ order). 
 }\label{Fig:CC}
\end{figure}

 \subsection{Test \#2}
 Next we consider a second validation test in the full case where the opinion and the network evolution are coupled, again with  linear compromise function, 
$P(w,w_*;c,c_*) = 1$, and $D(w,c) = 1-w^2$. In this case we are able again to characterize the analytical solution of the model as the product of the two stationary solution for the opinion variable and the connectivity, i.e. $f_\infty(w,c) = \rho_\infty(c)g_\infty(w)$. 
We define an initial data as follows,
\begin{equation}
\begin{aligned}\label{eq:f0}
f_0(w,c)      = \frac{2}{3}p_0(c)g_0^+(w)+\frac{1}{3}p_0(c-c_0)g_0^-(w),
\end{aligned}
\end{equation}
where 
\begin{align}
p_0(c)        =  k_0\max\{c(c-2\gamma_0),0\}, \qquad g_0^\pm(c) =  \frac{1}{\sqrt{2\pi\sigma_F^2 }}(\exp\{-(w\pm1/2)^2/{(2\sigma_F^2)}\}.
\end{align}
and with coefficient $c_0=20$ and  $k_0 = 3/(20\gamma_0^3)$.
Parameters are defined in Table \ref{tab:all_parameters}, except for the characteristic rates, which will be defined in two different ways.
We want to study the decay of the $L^1$ relative error with respect to the time, as depicted in Figure \ref{Fig:CC}.

In the first case we consider constant characteristic rates, i.e. $V=\Ua=\Ur$, showing that for increasing values of $V$ the convergence of the numerical scheme is faster. This is not surprising since for larger values of $V$ the dynamics of the connectivity distribution relaxes faster towards the stationary state.
 
We report in Figure \ref{Fig:stat} the evolution of the density $f(w,c,t)$,  in the time frame $[0,T]$, with $T=20$, where 
on the (z,c)-axis the distribution of the connections, $\rho(c,t)$, is represented in order to better enlighten the convergence to the power-law like distribution.
\begin{figure}[t]
\centering
\subfigure[t=0]{\includegraphics[scale= 0.3]{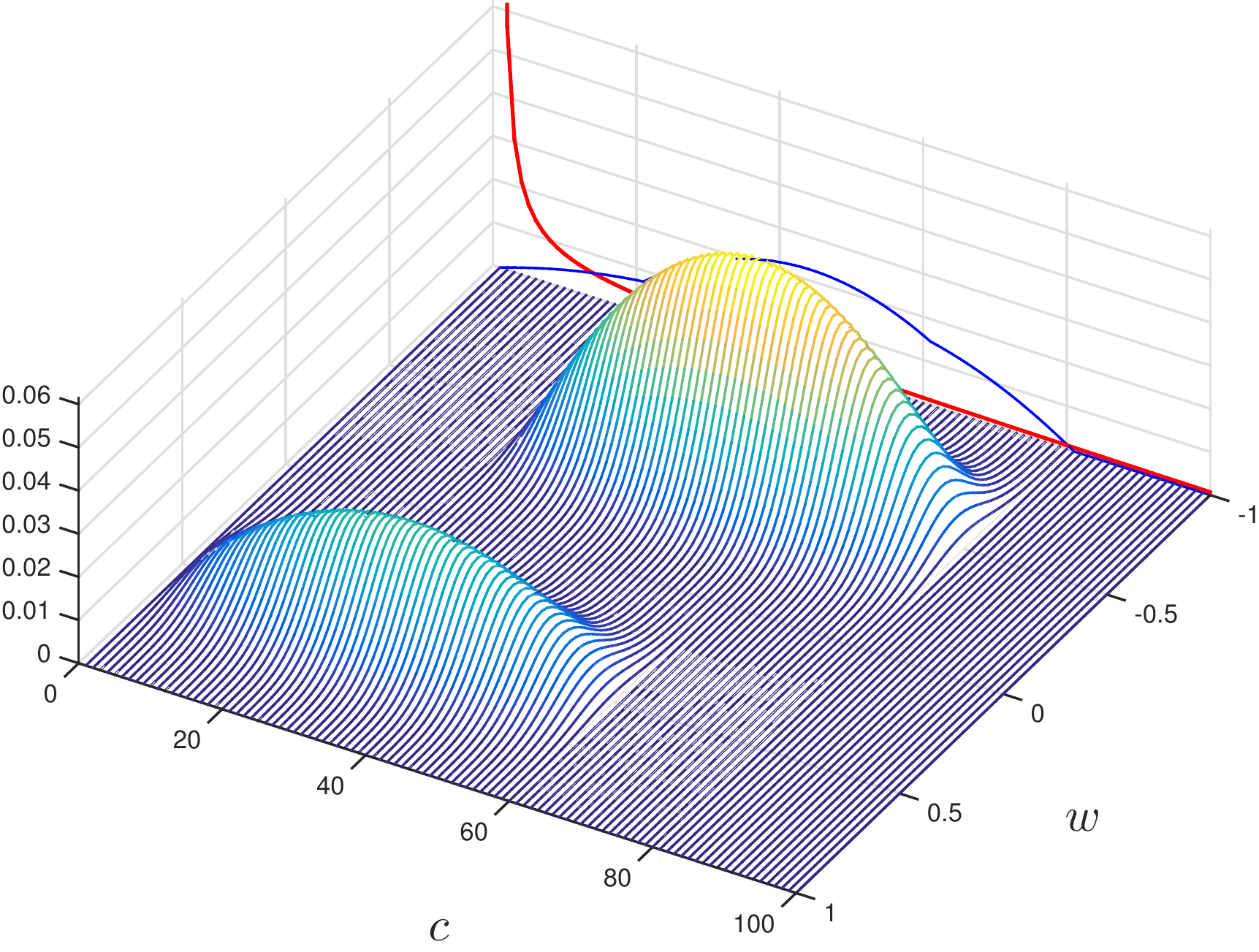}}
\hskip +0.2cm
\subfigure[t=2]{\includegraphics[scale= 0.3]{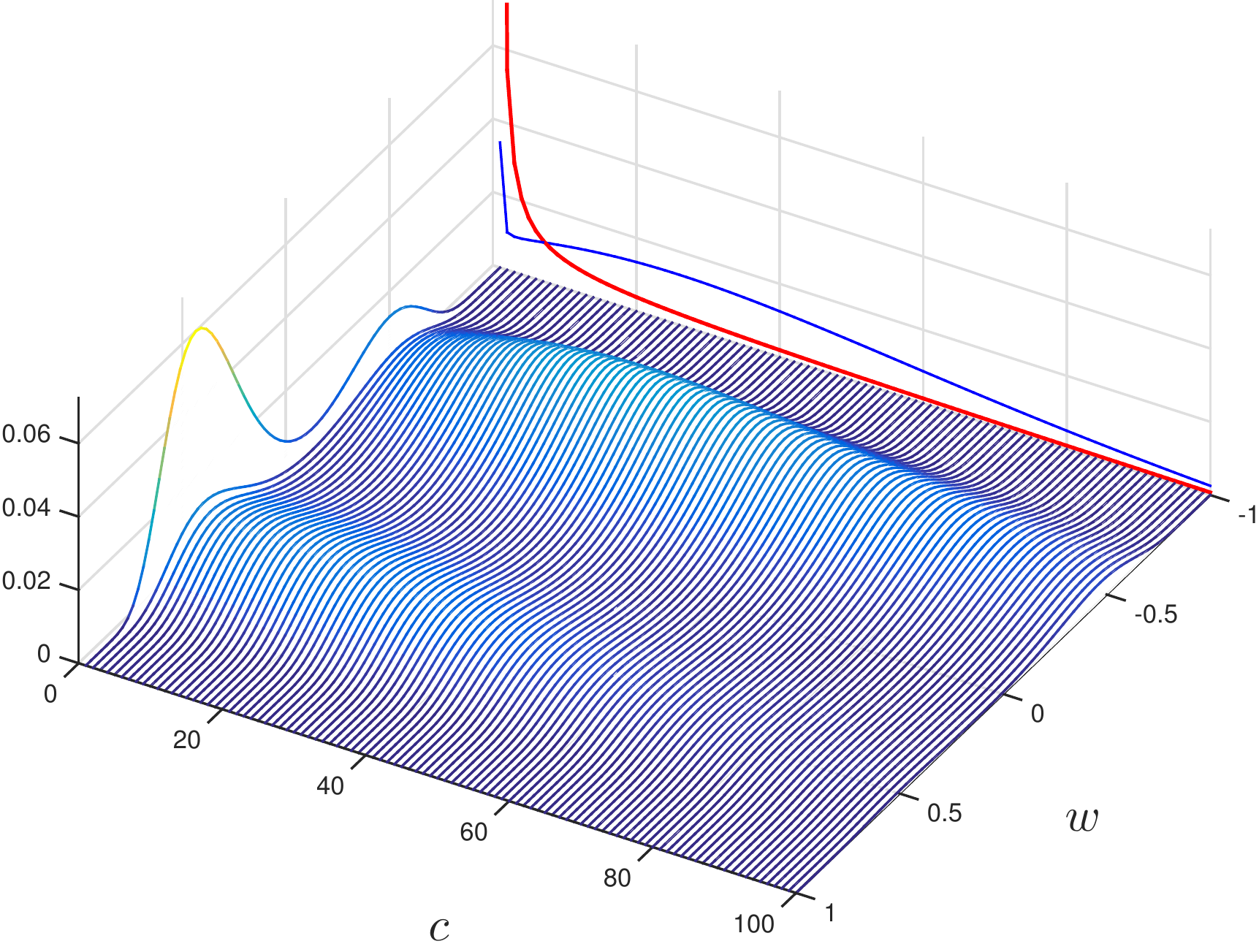}}
\\
\subfigure[t=5]{\includegraphics[scale= 0.3]{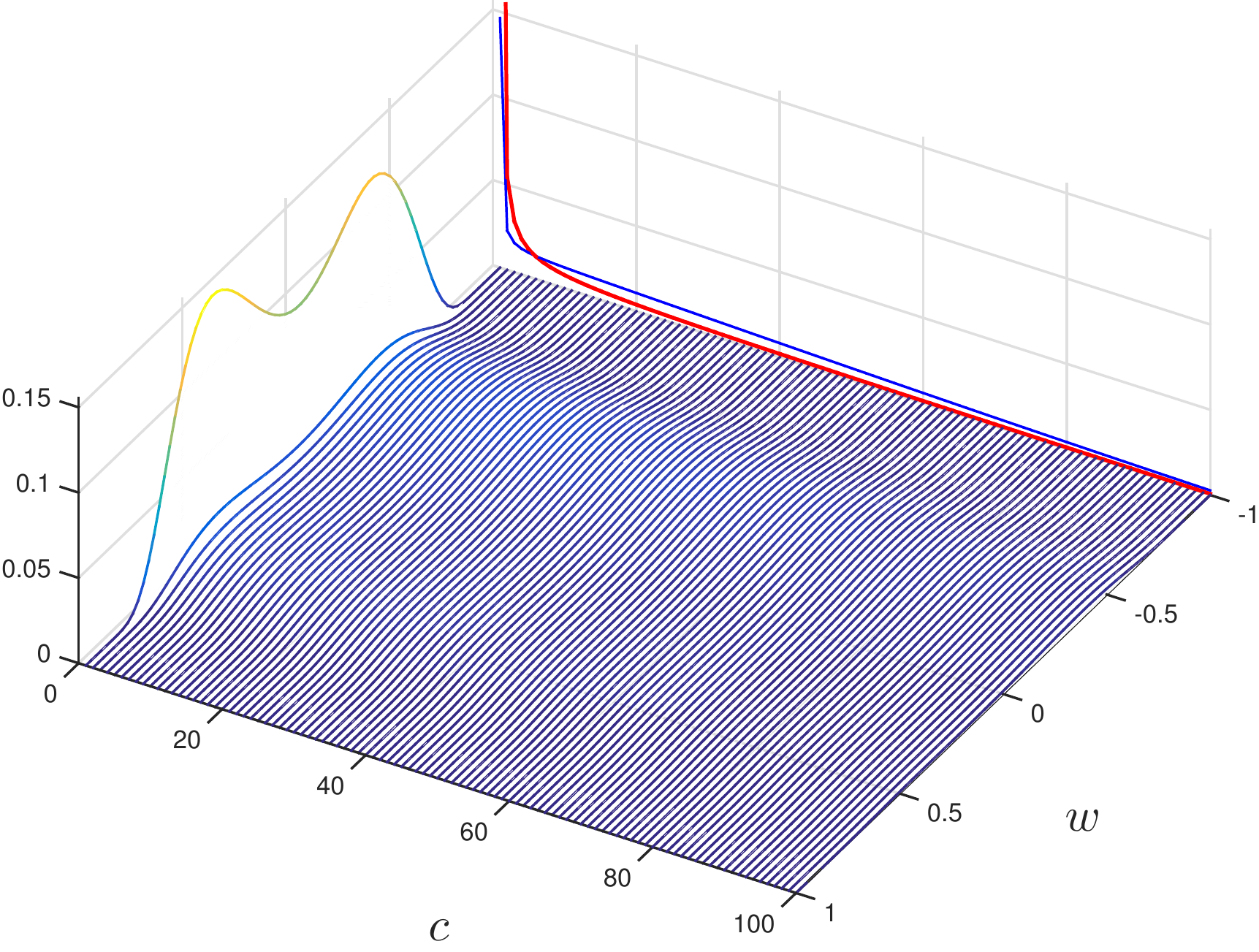}}
\hskip +0.2cm
\subfigure[t=20]{\includegraphics[scale= 0.3]{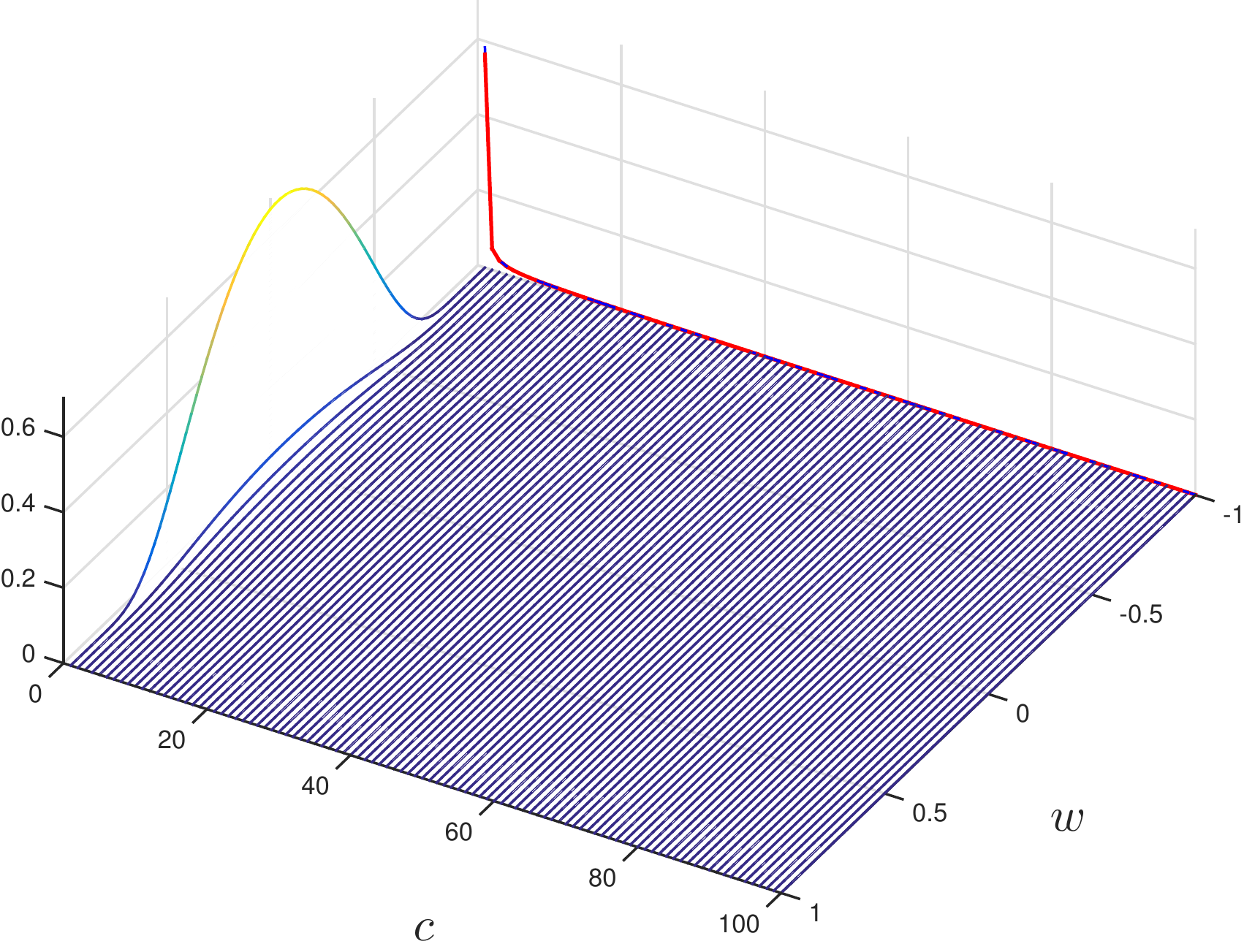}}
\caption{Test \#2. From left to right and from the top to the bottom, evolution of the density $f(w,c,t)$. Where plot $(a)$ represents the initial data , $f_0(w,c)$, \eqref{eq:f0}, and plot $(d)$ the stationary solution.
On the $(c,z)-$ plane we depict with a blu line the marginal distribution of the solution at time $t$, $p(c,t)$, with red line we represent the reference marginal distribution of the stationary solution.
 }\label{Fig:stat}
\end{figure}

In a second test we performed the same simulation, but with characteristic rates defined as in Remark \ref{rmk:CR}, thus
\be
\Ur(f;w)=U_r\frac{\gamma+\beta}{\gamma_f+\beta g(w,t)},\quad \Ua(f;w)=U_a\frac{\gamma+\alpha}{\gamma_f+\alpha g(w,t)}
\label{eq:rates2}
\ee
with $U=U_a=U_r$, $\beta = 0$ and $\gamma_f(t) = \sum_{c=0}^{\cm} cf(w,c,t)$. Simulations shows that in this case the same stationary solution are obtained.

We report in Figure \ref{Fig:error2d}, the decay of the errors for different values of the characteristic rates, in the two different cases,
$V=\{10^3,10^4,10^5\}$ for the constant rate and $U =\{10^3,10^4,10^5\}$ for the variable rates. In both cases we observe a faster convergence to the stationary solution for increasing values of the characteristic rates. Observe that the same order of accuracy of the mid--point rule in Figure \ref{Fig:CC} is recovered, on the other hand tiny differences in the decay are observed between the two cases.
In Figure \ref{Fig:Trans}, we enlighten the different evolution of the transient solution at time $t = 1$, of the simulation in Figure 6. On the left we depict the solution with constant characteristic rates, on the right with variable characteristic rates, which
shows that lower density in the opinion leads to faster spread on the connections.
\begin{figure}[t]
\centering
\includegraphics[scale= 0.3]{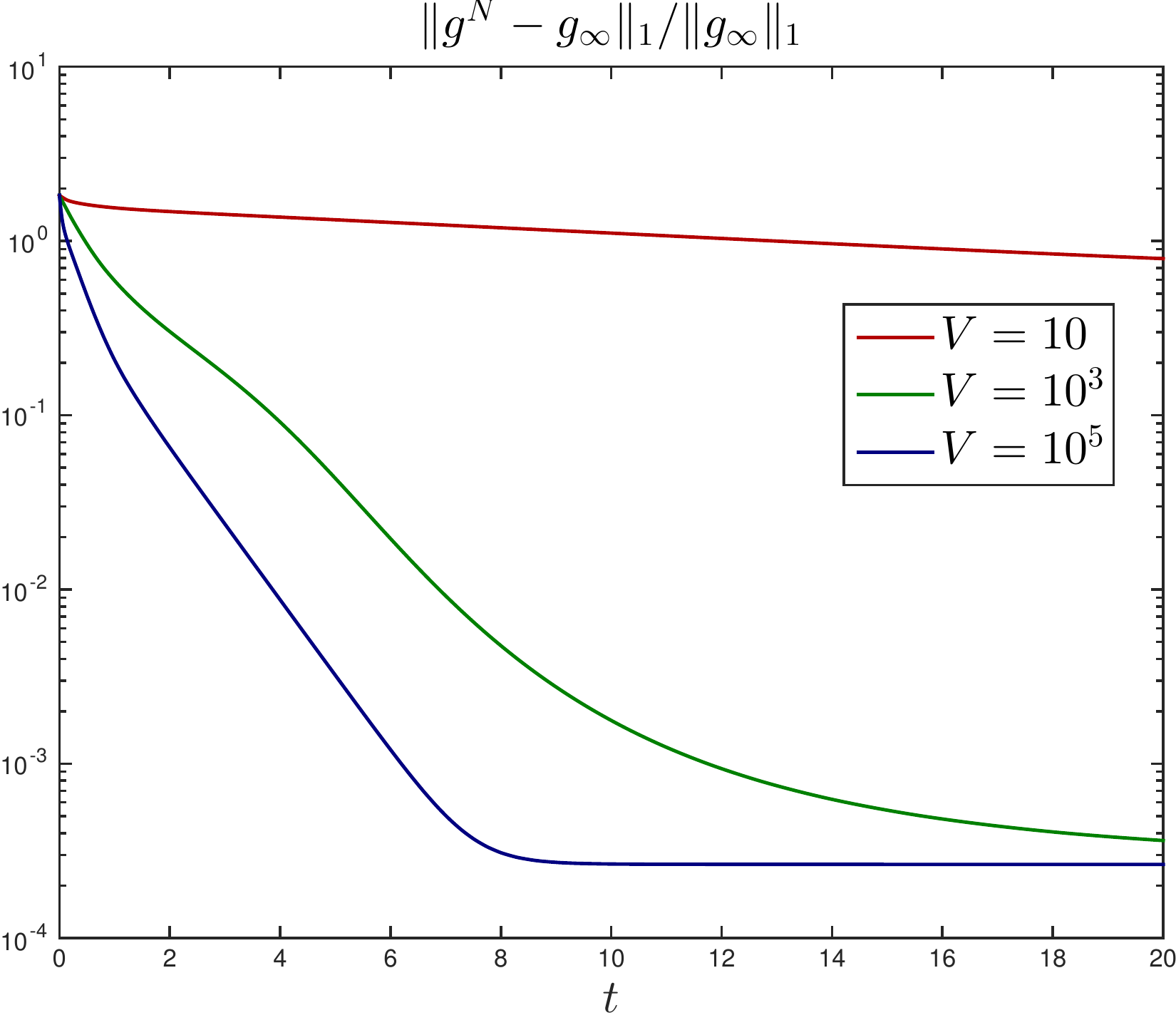}
\hskip +0.2cm
\includegraphics[scale= 0.3]{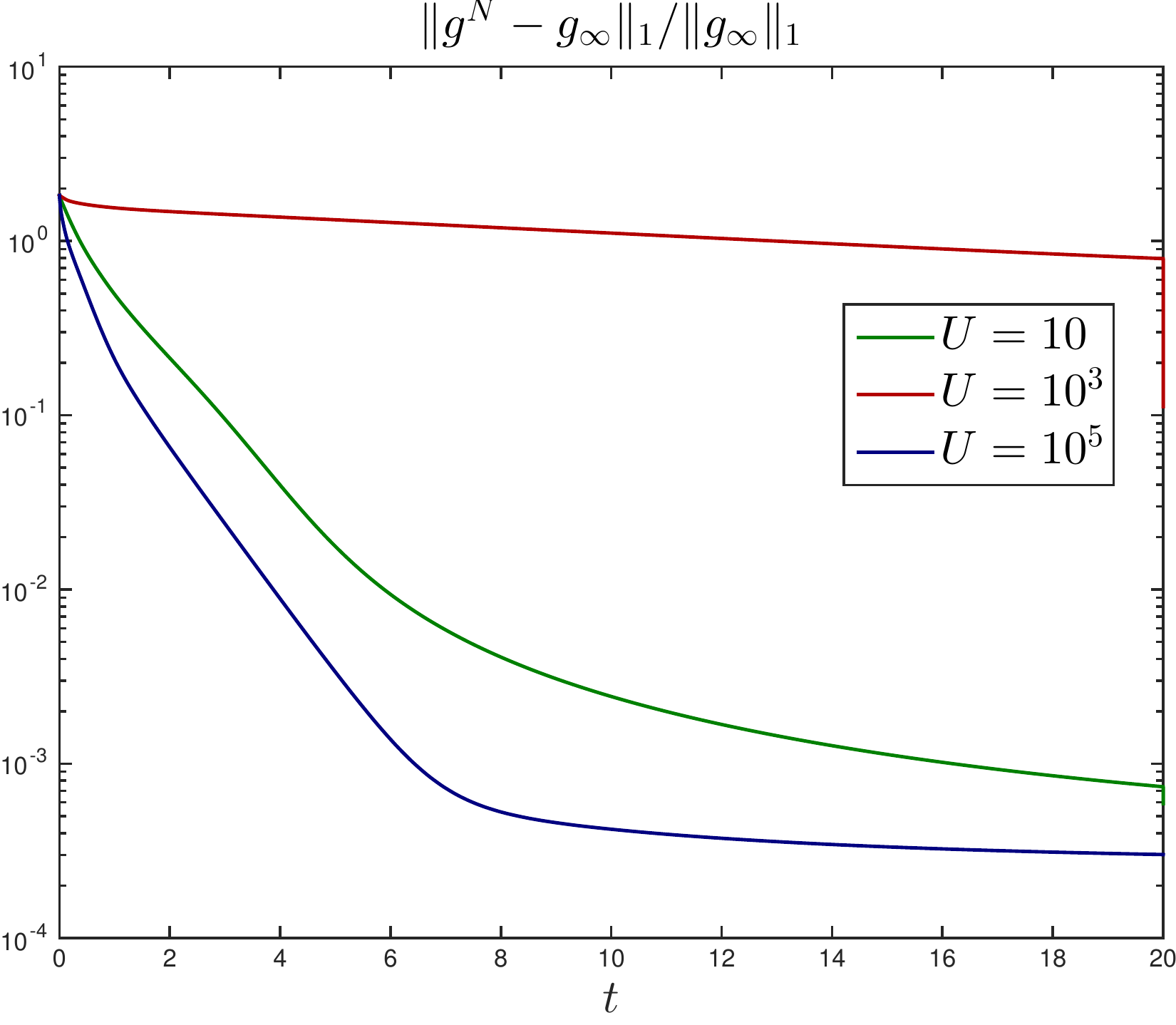}
\caption{Test \#2. Decay of the  $L^1$ relative error with respect to the stationary solution \eqref{eq:stat_1}. On the left, fixed characteristic rates $V=\{10^3,10^4,10^5\}$, on the right variable characteristic rates defined as in  \eqref{eq:rates2}, with $U =\{10^3,10^4,10^5\}$. In both cases for increasing values of the characteristic rate $V$ and $U$ the stationary state is reached faster. }\label{Fig:error2d}
\end{figure}

\begin{figure}[t]
\centering
\includegraphics[scale= 0.3]{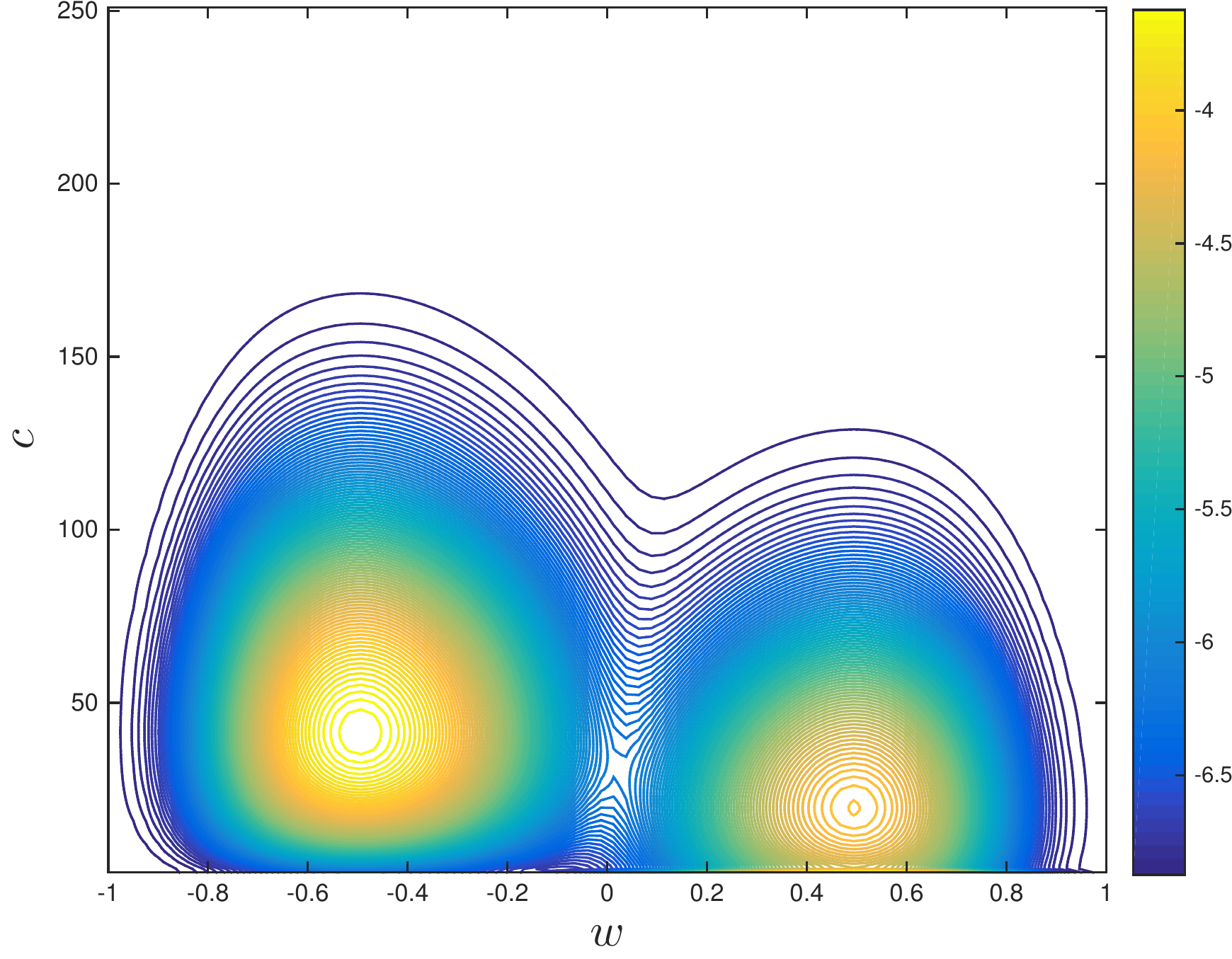}
\hskip +0.2cm
\includegraphics[scale= 0.3]{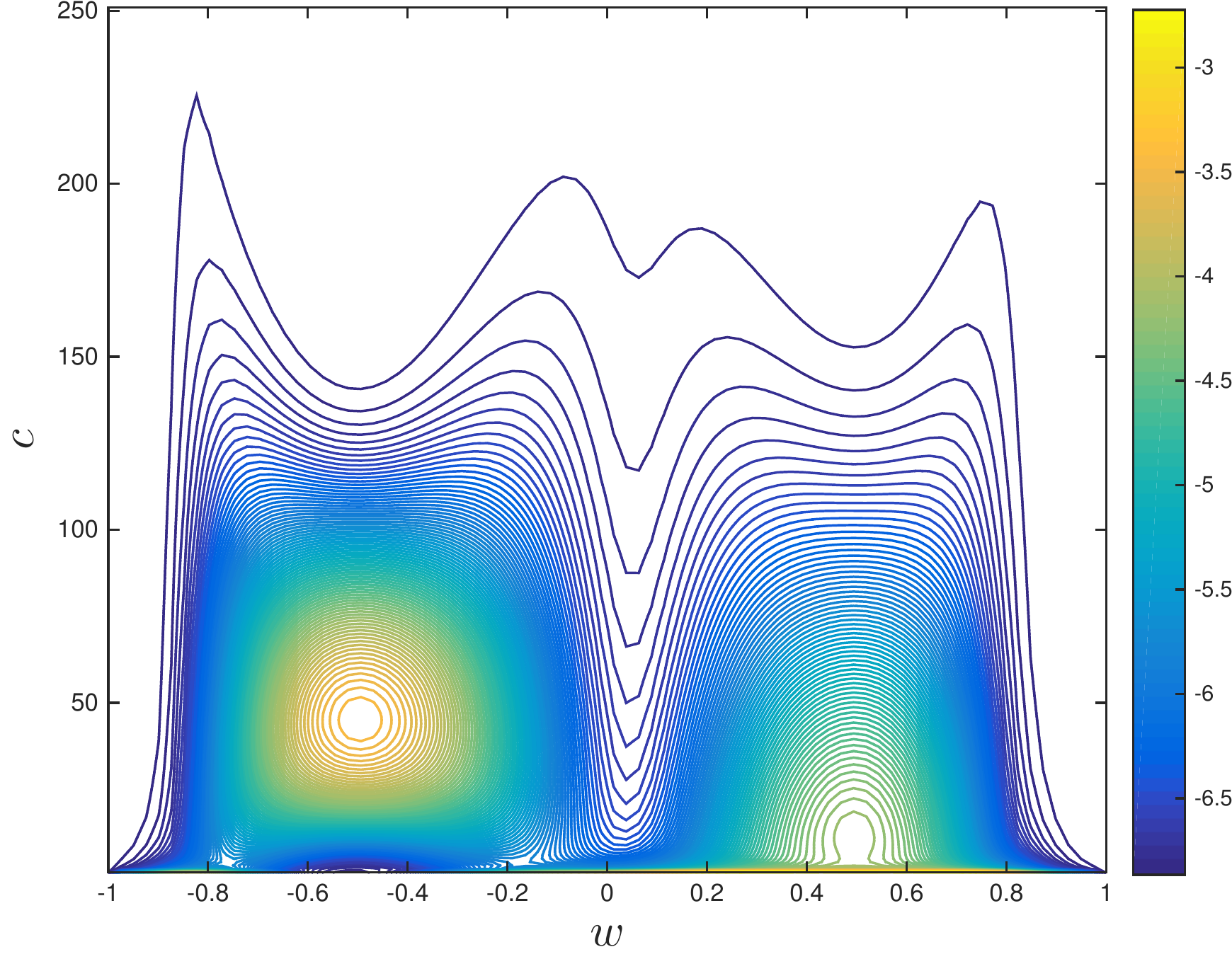}
\caption{Test \#2. Evolution at time $t = 1$ of the initial data,
$f_0(w, c)$, \eqref{eq:f0}, as isoline plot. On the left in the case of constant
characteristic rate on the right variable characteristic rates defined
as in \eqref{eq:rates2}. The right plot shows that for lower opinion's density the evolution along the connection is faster and slower where the opinions are more concentrated.}\label{Fig:Trans}
\end{figure}

\subsection{Test \#3}
In this test we analyze the influence of the connections over the opinion dynamics, by considering a compromise function of the type 
\[
P(w,w_*;c,c_*) = H(w,w_*)K(c,c_*), 
\]
where $H(w,w_*) = 1-w^2$ and  $K(\cdot,\cdot)$, which models the influence of the connectivity on the opinion evolution, is defined as follows
 \begin{align}\label{eq:Kc}
\quad K(c,c_*) = \left(\frac{c}{\cm}\right)^{-a} \left(\frac{c_*}{\cm}\right)^{b},
\end{align}
for $a,b >0$. This type of kernel assigns higher relevance into the opinion dynamics to higher connectivity, and low influence to low connectivity.
The diffusivity is weighted by $D(w,c) = 1-w^2$.

We perform a first computation with initial data 
\begin{align}\label{eq:T2f0}
f_0(w,c) =C_0
\begin{cases}
 \rho_\infty(c)\exp\{-(w+\frac{1}{2})^2)/(2\sigma_F^2)\},&\quad \textrm{ if }  0\leq c\leq 20,\\
 \rho_\infty(c)\exp\{-(w-\frac{3}{4})^2/(2\sigma_L^2)\},&\quad \textrm{ if }   60\leq c\leq 80,\\
 0, &\quad \textrm{ otherwise},
\end{cases}
\end{align}
where the parameters' choice is reported in the third line of Table \ref{tab:par}, and $a=b=3$ for the interaction function $K(\cdot,\cdot)$, \eqref{eq:Kc}.
The evolution is performed through the Chang-Cooper type scheme with $\Delta w = 2/N$, with $N = 80$. We study the evolution of the system in the time interval $[0,T]$, with $T=2.5$.

In Figure \ref{Fig:F3} we report the result of the simulation. The initial configuration is is split in two parts, the majority concentrated around opinion $\bar{w}_F = - 1/2$ and only a small portion concentrated around $\bar{w}_L= 3/4$, from the upper-right and bottom-left figures we observe that, because of the anisotropy induced by $K(c,c_*)$, the density with a low level of connectivity is immediately influenced by the small concentration of density around $w_L$ with a large level of connectivity; the bottom-right plot shows the final configuration.

In Figure \ref{Fig:F4} we depict, on the left-hand side, the initial and final marginal density of the opinion, respectively $g(w,0)$ and $g(w,T)$, showing the change of the total opinion. On the right we enlighten the change of opinion plotting  the evolution of the average opinion.
\begin{figure}[t]
\centering
\subfigure[t=0]{\includegraphics[scale= 0.3]{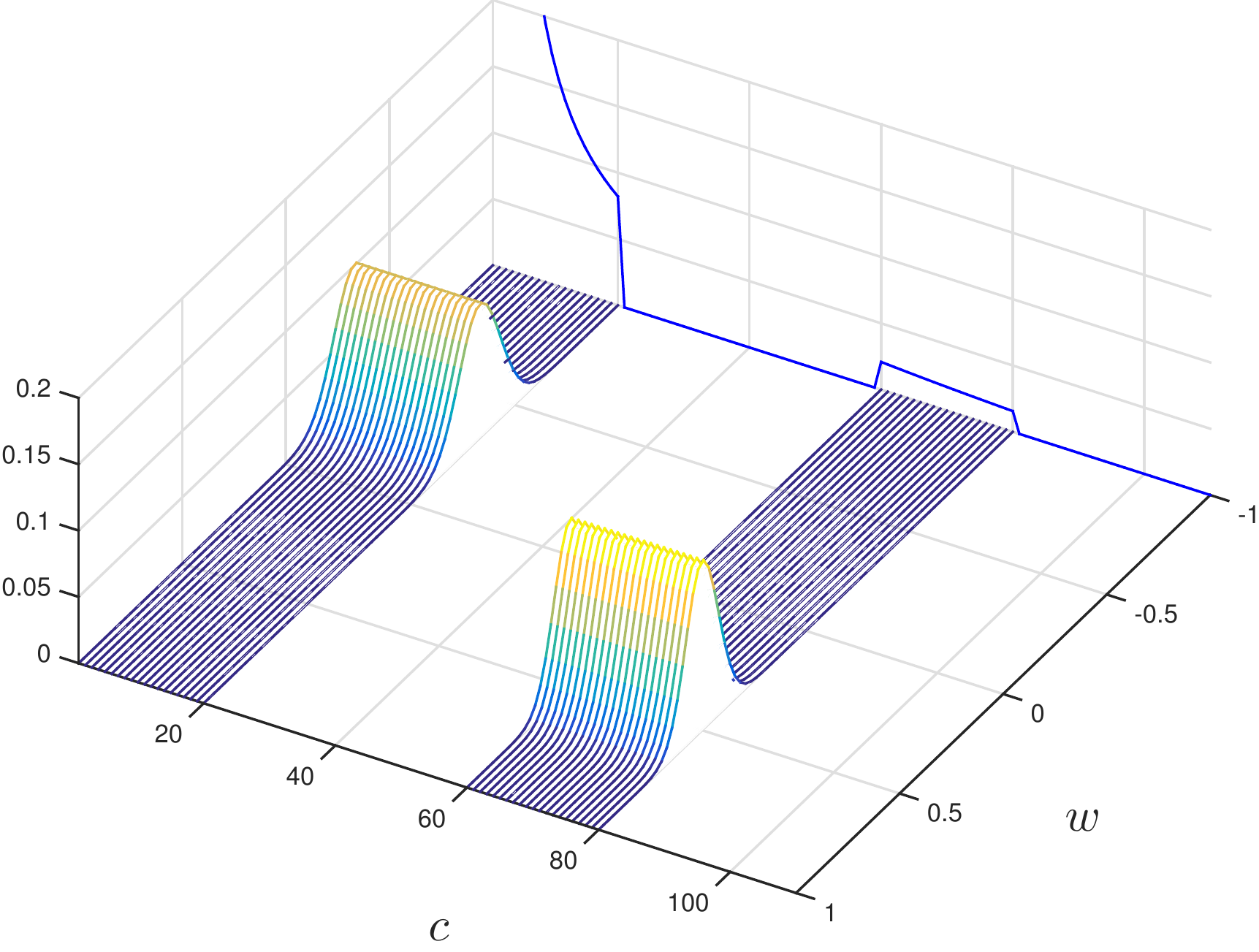}}
\hskip +0.2cm
\subfigure[t=0.2]{\includegraphics[scale= 0.3]{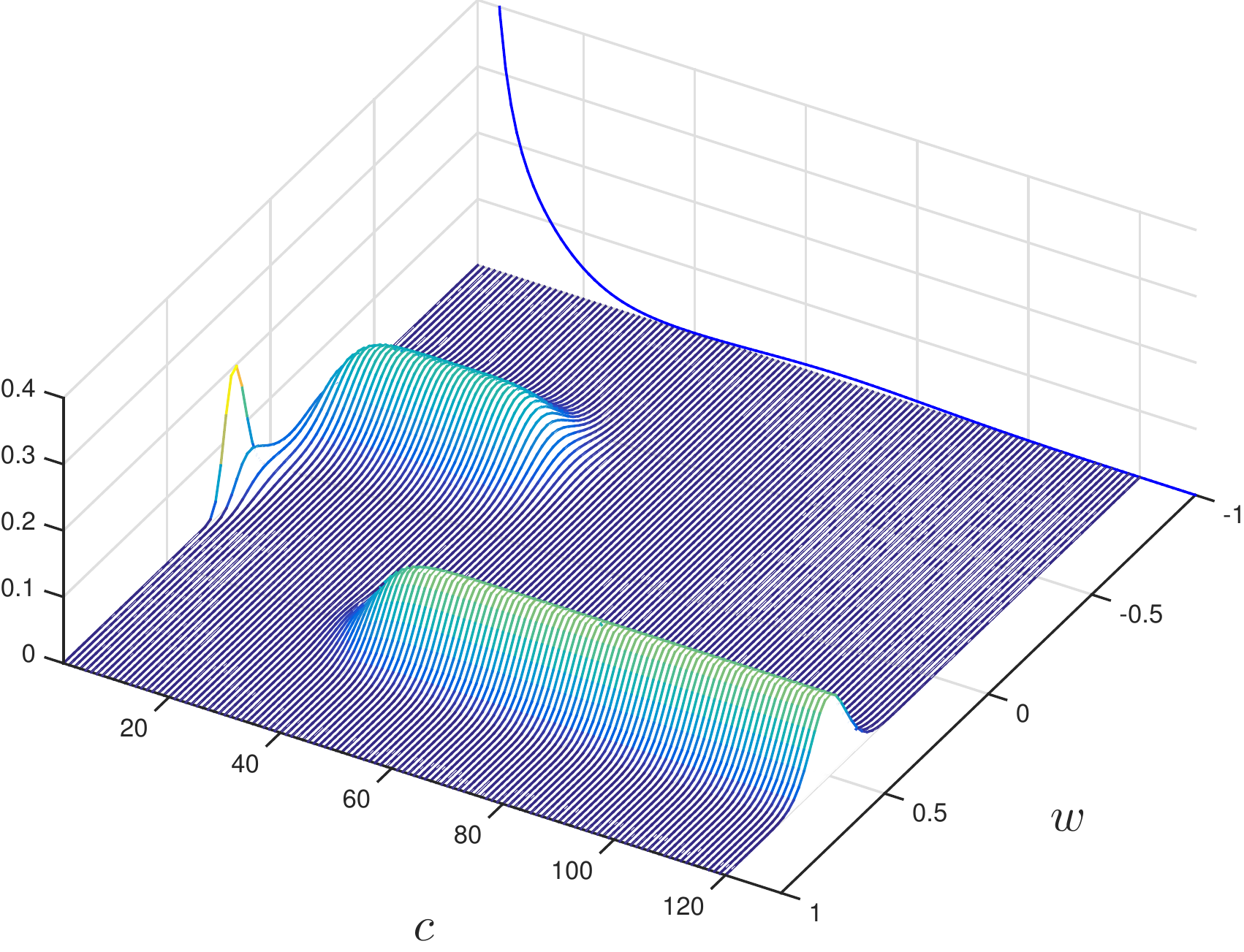}}
\\
\subfigure[t=1.5]{\includegraphics[scale= 0.3]{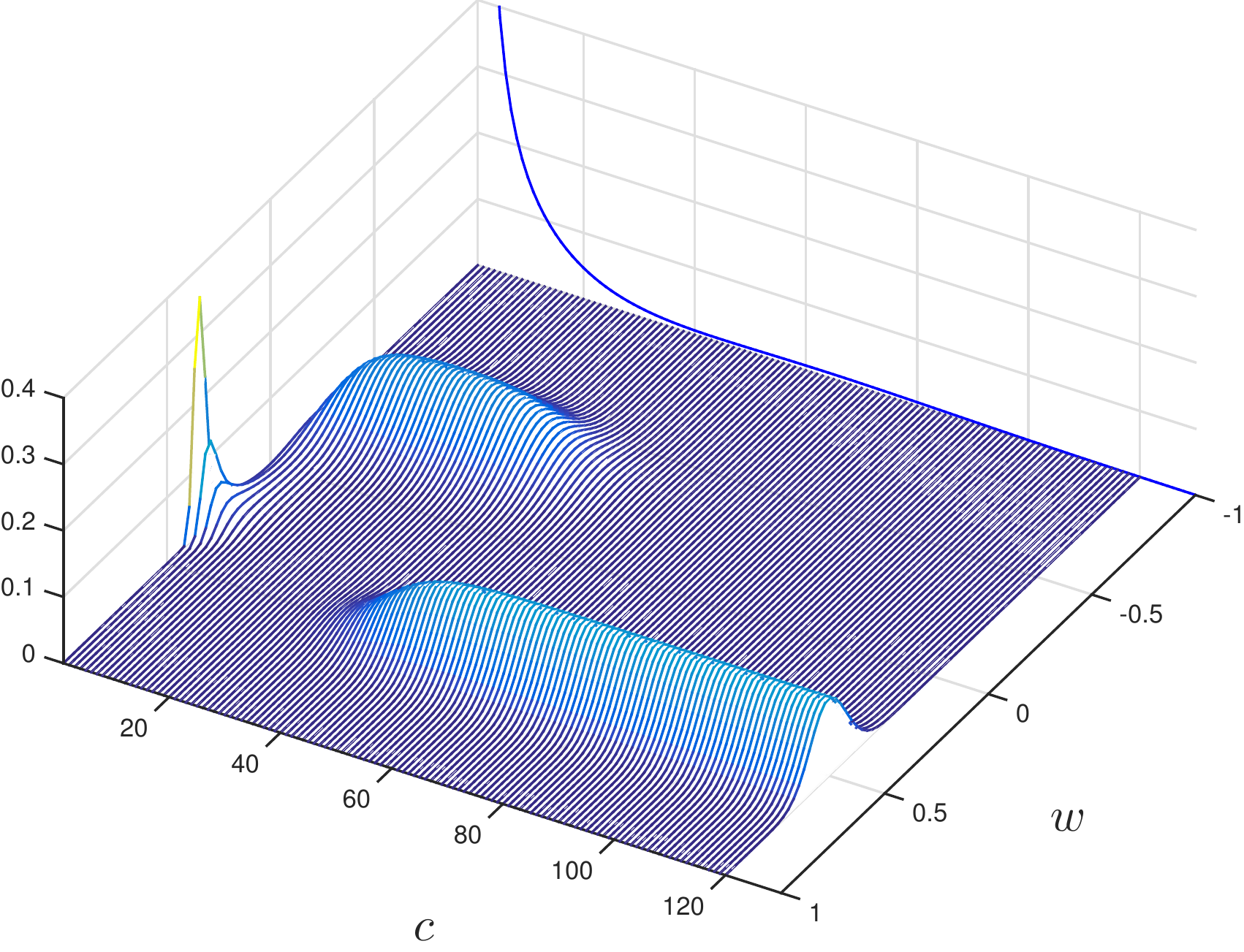}}
\hskip +0.2cm
\subfigure[t=2.5]{\includegraphics[scale= 0.3]{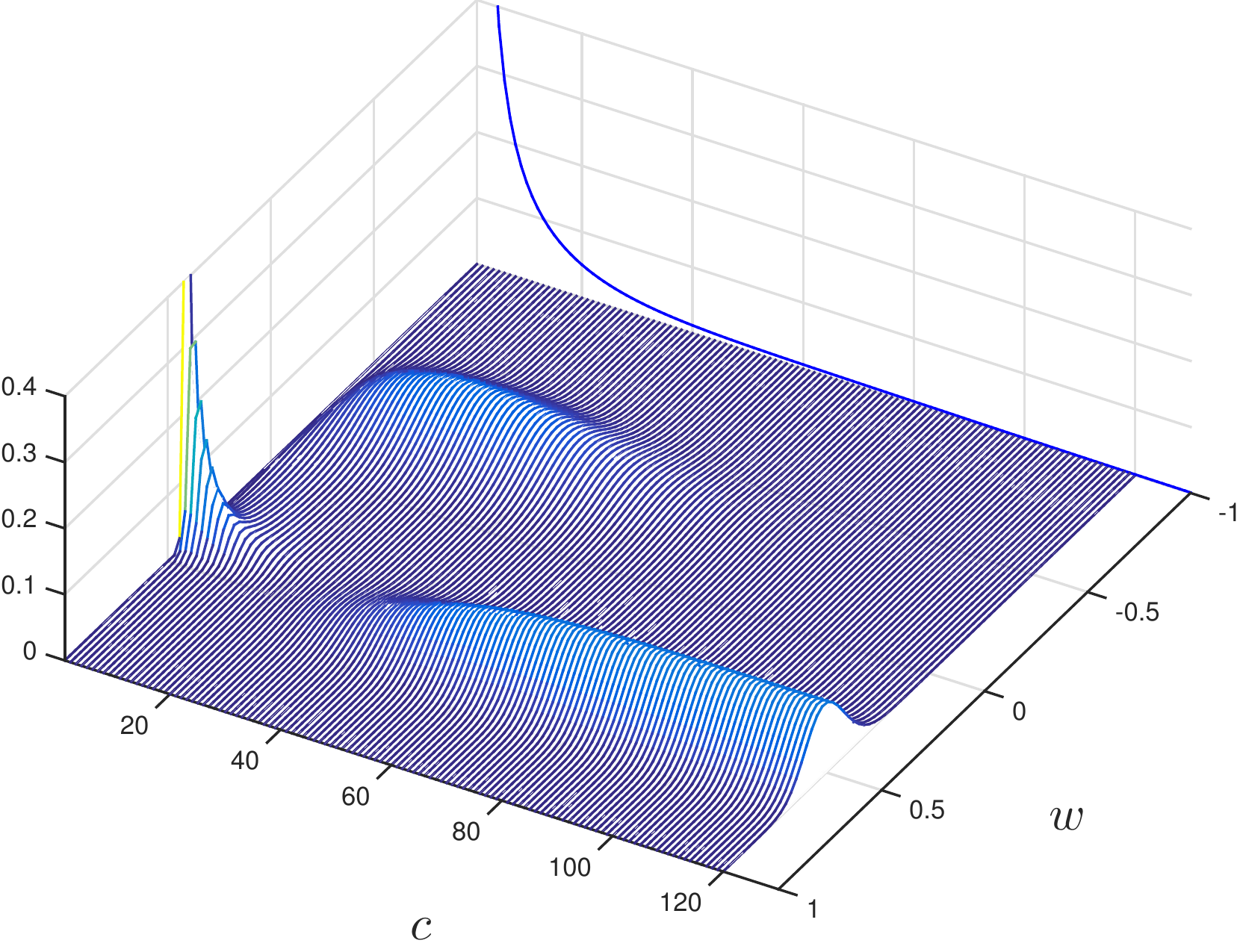}}
\caption{Test \#3. From left to right and from the first row to the second row, evolution of the initial data \eqref{eq:T2f0} in time frame $[0,T]$, with $T = 2$. The evolution shows how a small portion of density with high connectivity can bias the majority of the population towards their position. 
(Note: The density is scaled according to the marginal distribution $\rho(c,t)$ in order to better show its evolution, the actual marginal density $\rho(c,t)$ is depicted in the background, scaled by a factor 10).
}\label{Fig:F3}
\end{figure}
\begin{figure}[t]
\centering
\includegraphics[scale=0.3]{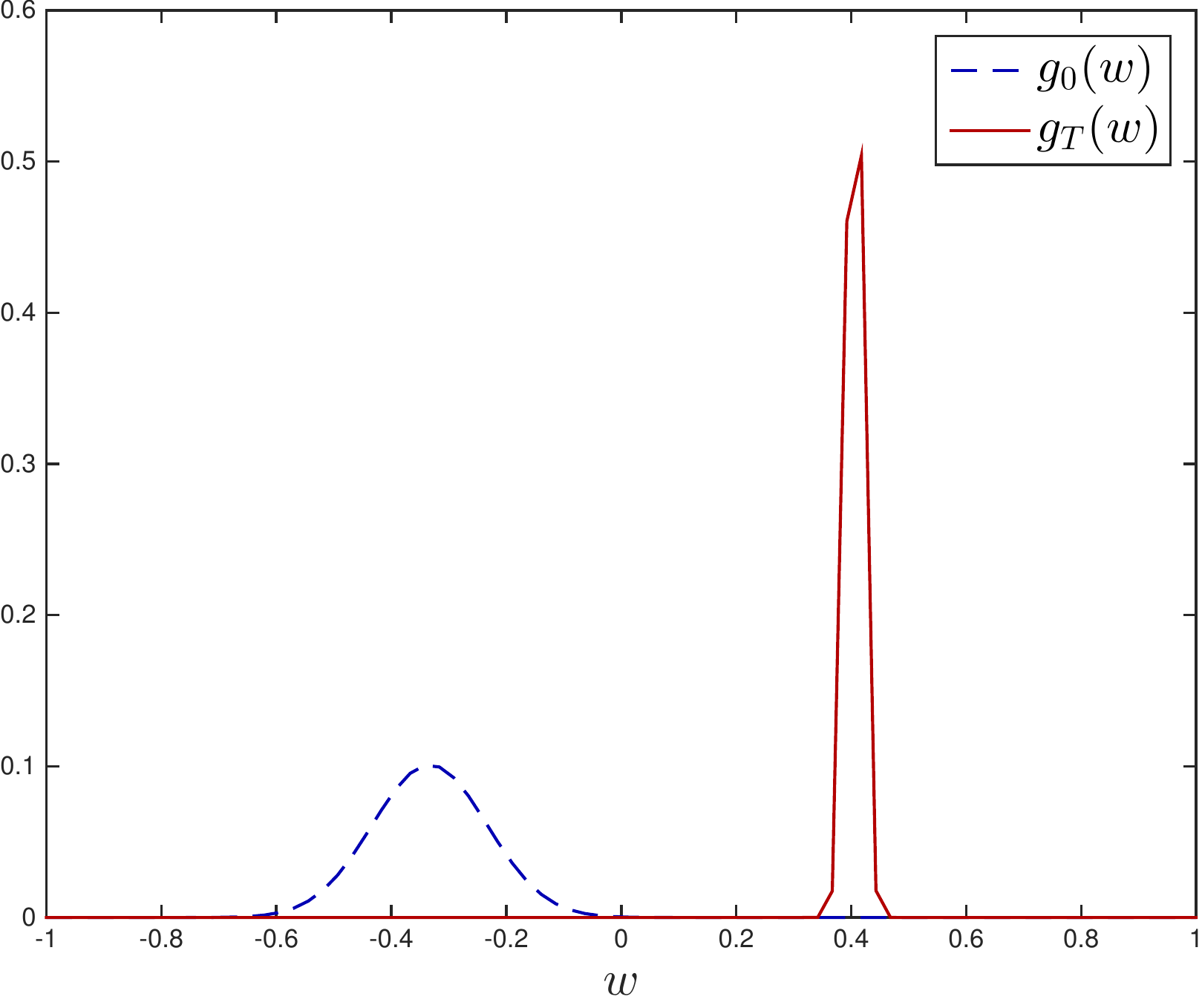}
\hskip +0.5cm
\includegraphics[scale = 0.3]{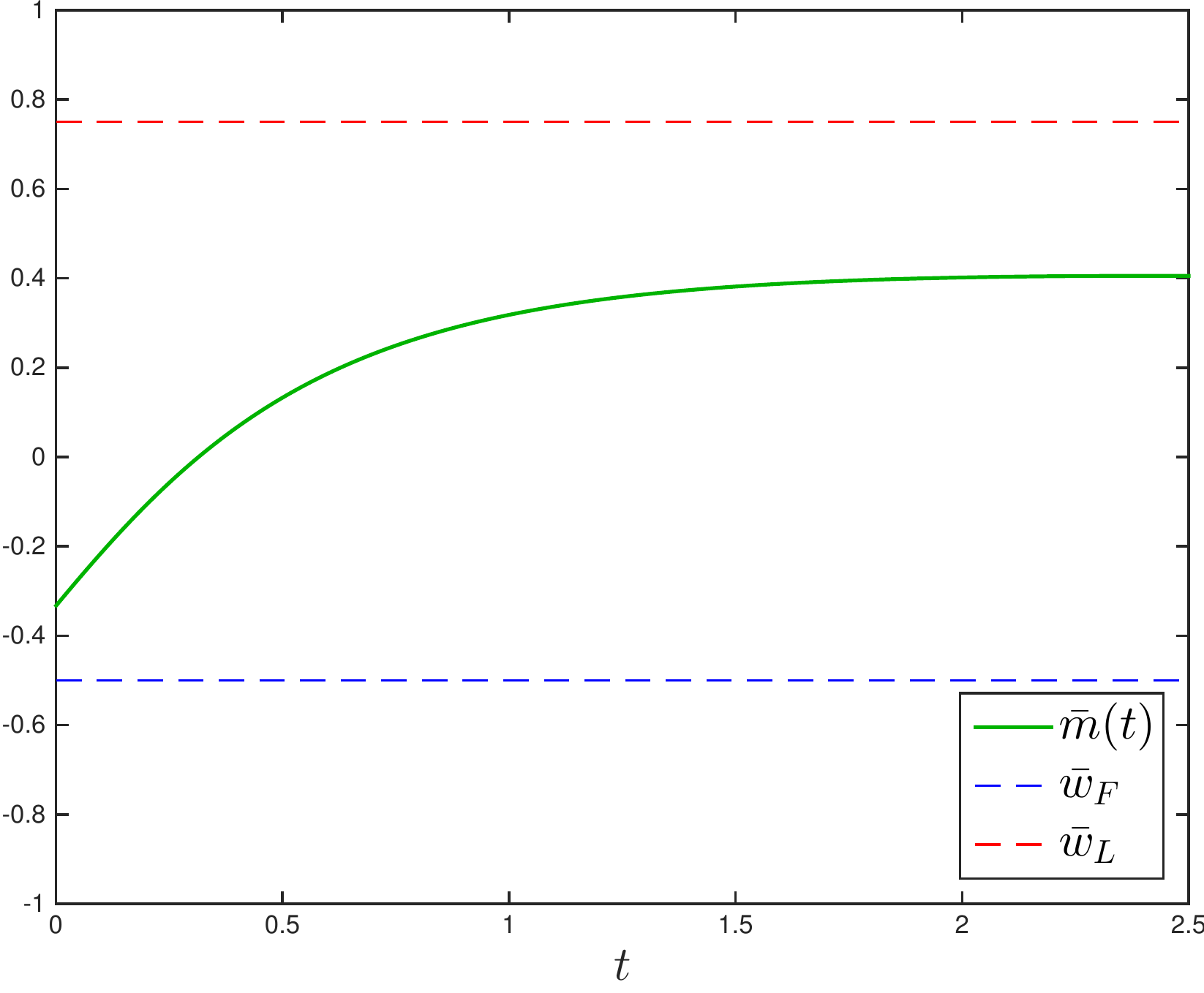}
\caption{Test \#3. On the left-hand side final and initial state of the marginal distribution $g(w,t)$ of the opinion, the green line represents the evolution of the average opinion $\bar{m}(t)$, the red and blue dashed lines represent respectively the opinions $\bar{w}_L=0.75$ and $\bar{w}_F=-0.5$, which are the two leading opinions of the initial data \eqref{eq:T2f0}.}\label{Fig:F4}
\end{figure}

\subsection{Test \#4}
In the last test case, we consider the Hegselmann-Krause model, \cite{HK}, known also as bounded confidence model, where agents interact only with agents whose opinion lays within a certain
range of confidence. Thus we define the following compromise function
\begin{align}\label{eq:HK}
P(w,w_*;c,c_*) = \chi_{\{|w-w_*|\leq\Delta(c)\}}(w_*),
\end{align}
where $\Delta(c)$ is the confidence level and we assume that in general it depends on the number of connections.
We define the initial data 
\begin{align}\label{eq:initD}
f_0(w,c) = \frac{1}{2}\rho_\infty(c),
\end{align}
therefore the opinion is uniformly distributed on the interval $I=[-1,1]$ and it decreases along $c\in[0,\cm]$ following $\rho_\infty(c)$, as in \eqref{prop:st}, with parameters defined in Table \ref{tab:all_parameters} and $D(w,c) = 1-w^2$. 
The evolution is performed through the Chang-Cooper type scheme with $\Delta w = 2/N$, with $N = 80$. We consider the evolution of the system in the time interval $[0,T]$, with $T=100$.

We study first a confidence level independent from the number of connections, therefore  we set $\Delta(c)=\Delta =0.25$. In Figure \ref{Fig:F5} the evolution of the initial data \eqref{eq:initD} shows the classical behavior of  Hegselmann-Krause model, where opinions' clusters emerge.
\begin{figure}[t]
\centering
\subfigure[t=0]{\includegraphics[scale= 0.3]{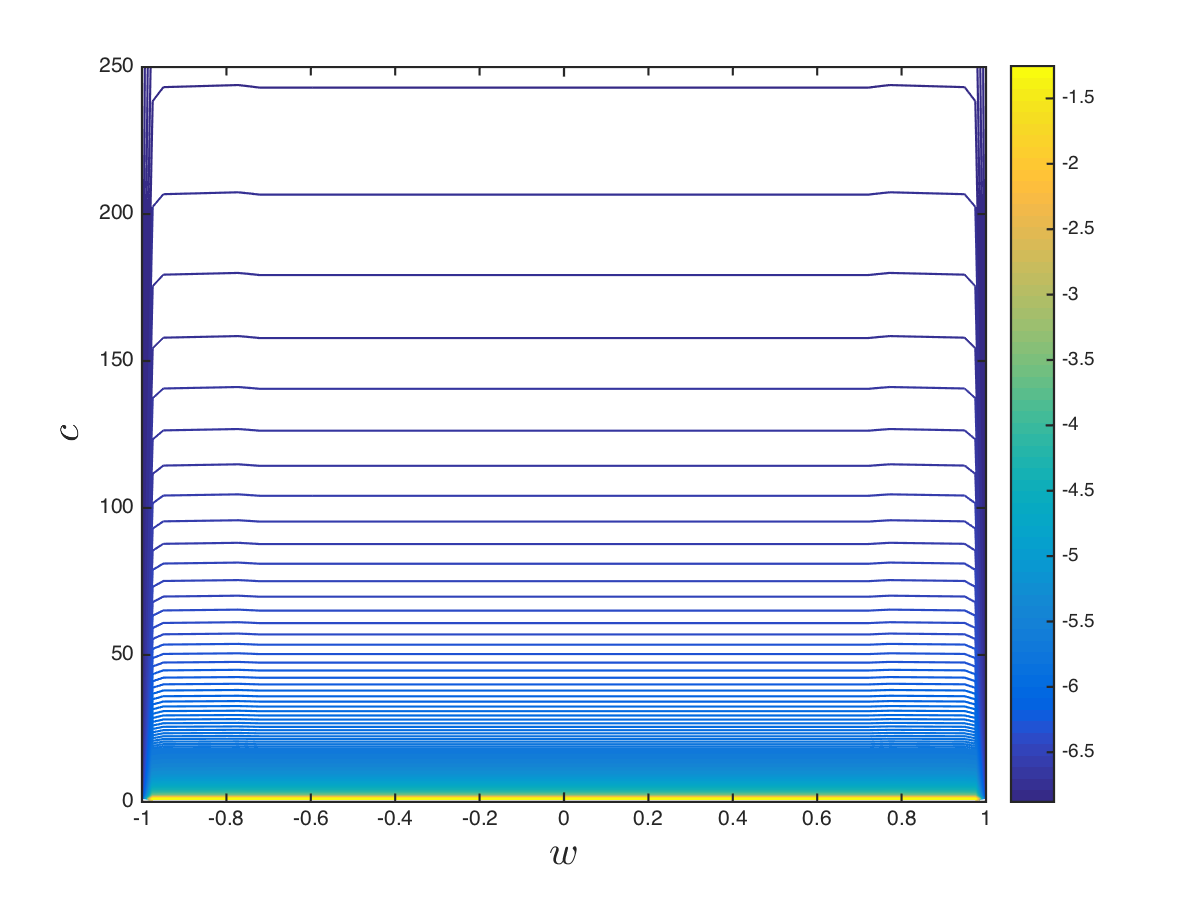}}
\hskip +0.2cm
\subfigure[t=10]{\includegraphics[scale= 0.3]{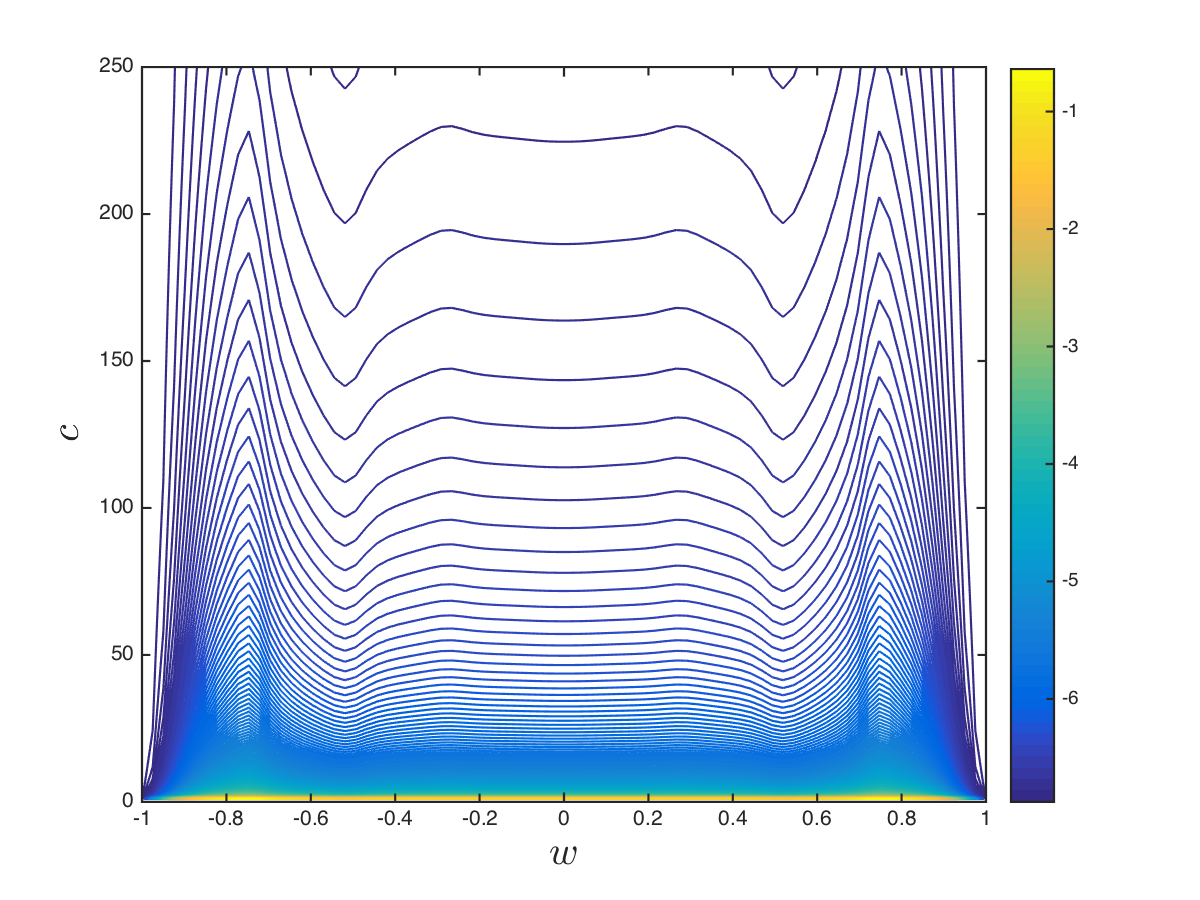}}
\\
\subfigure[t=50]{\includegraphics[scale= 0.3]{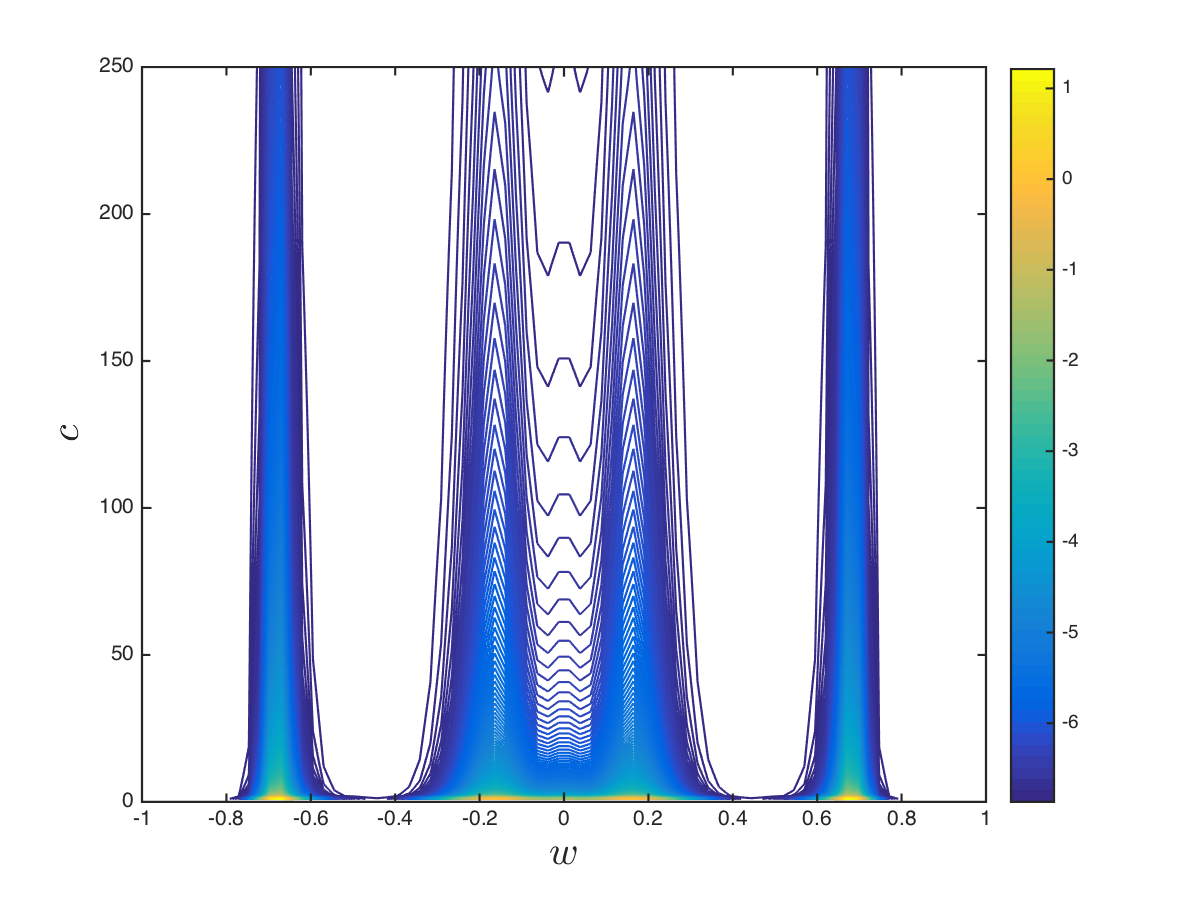}}
\hskip +0.2cm
\subfigure[t=100]{\includegraphics[scale= 0.3]{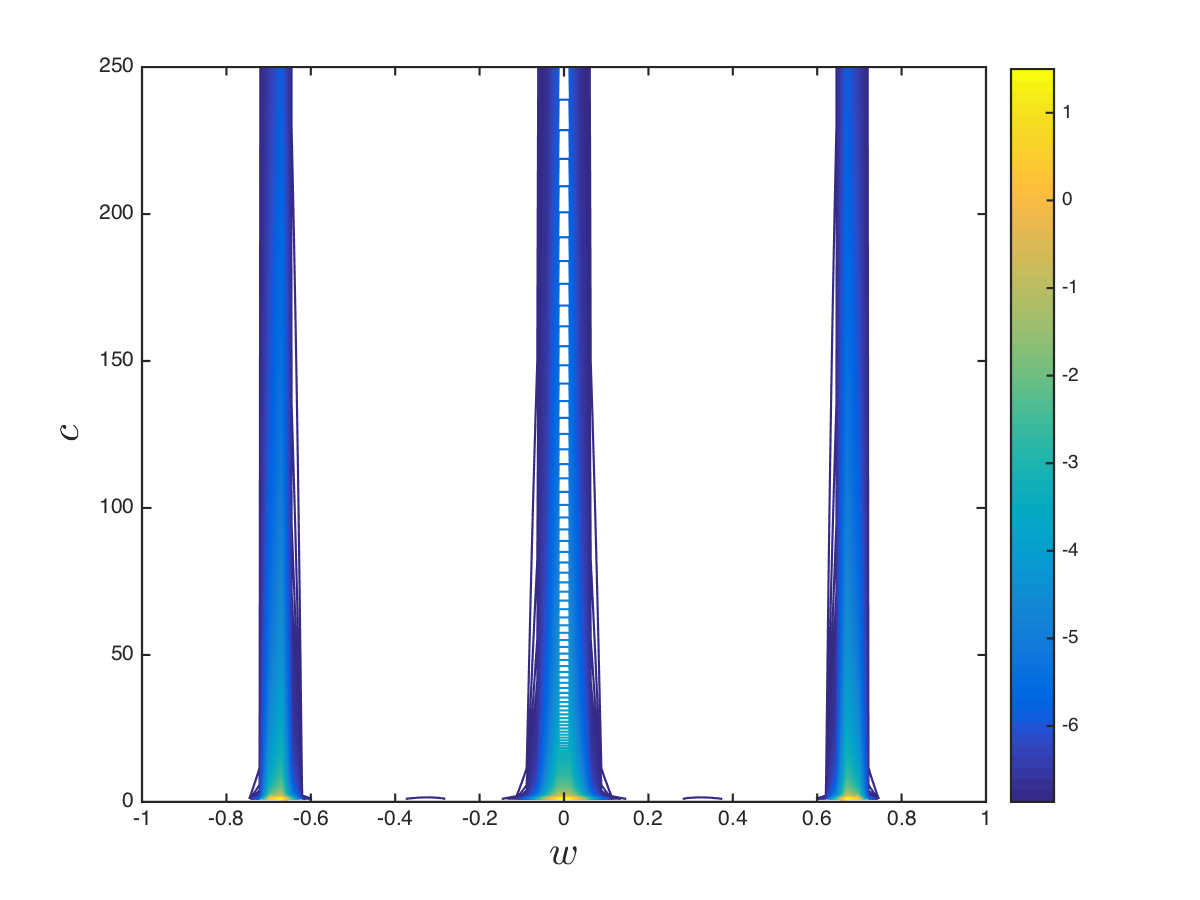}}
\caption{Test \#4. Evolution of the Fokker-Planck model \eqref{eq:FP} where the interaction are described by \eqref{eq:HK}, with $\Delta=0.25$, in the time frame $[0,T]$, with $T = 100$.
The evolution shows the emergence of three main opinion clusters, which are not affected by the connectivity variable. 
(Note: In order to better show its evolution, we represent the solution as $\log(f(w,c,t)+\epsilon)$, with $\epsilon = 0.001$.)
}\label{Fig:F5}
\end{figure}

\begin{figure}[h!]
\centering
\subfigure[t=0]{\includegraphics[scale= 0.3]{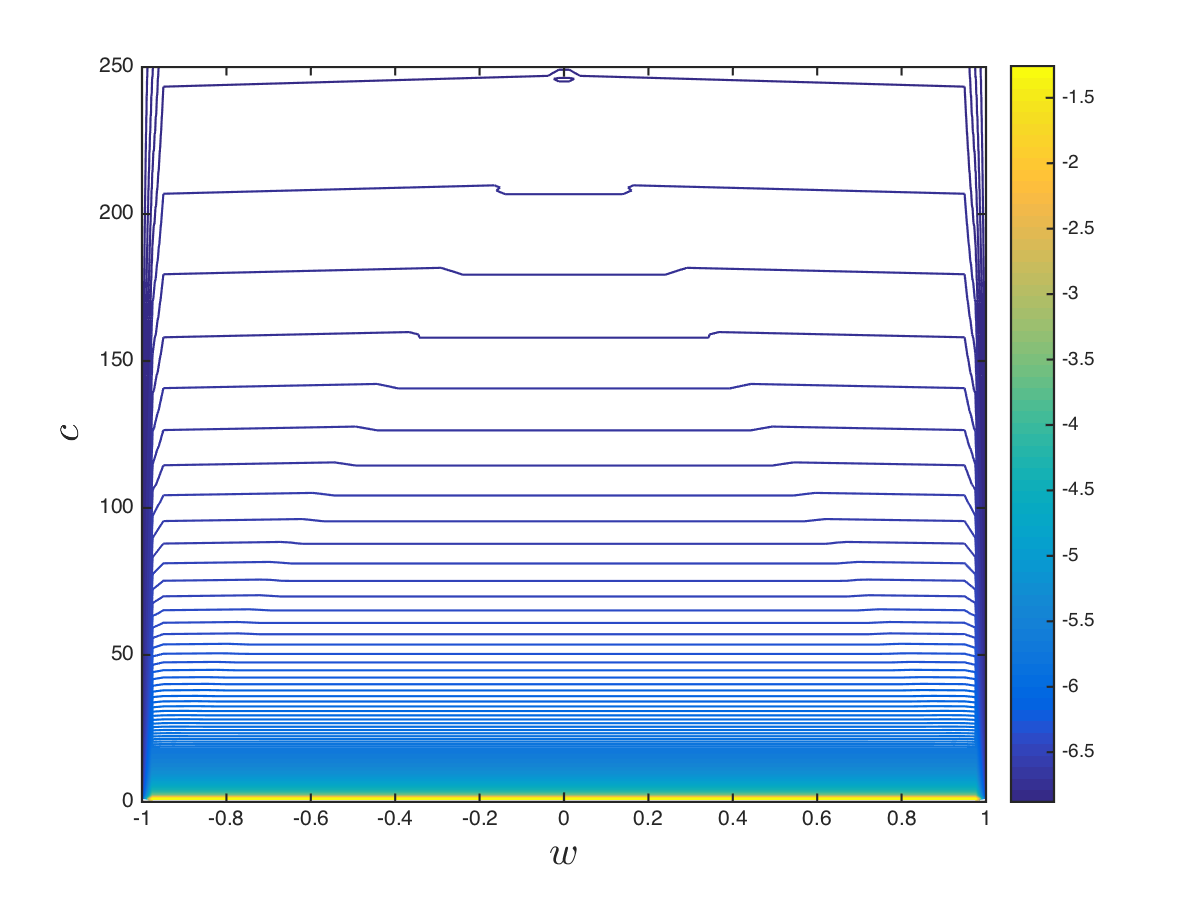}}
\hskip +0.2cm
\subfigure[t=10]{\includegraphics[scale= 0.3]{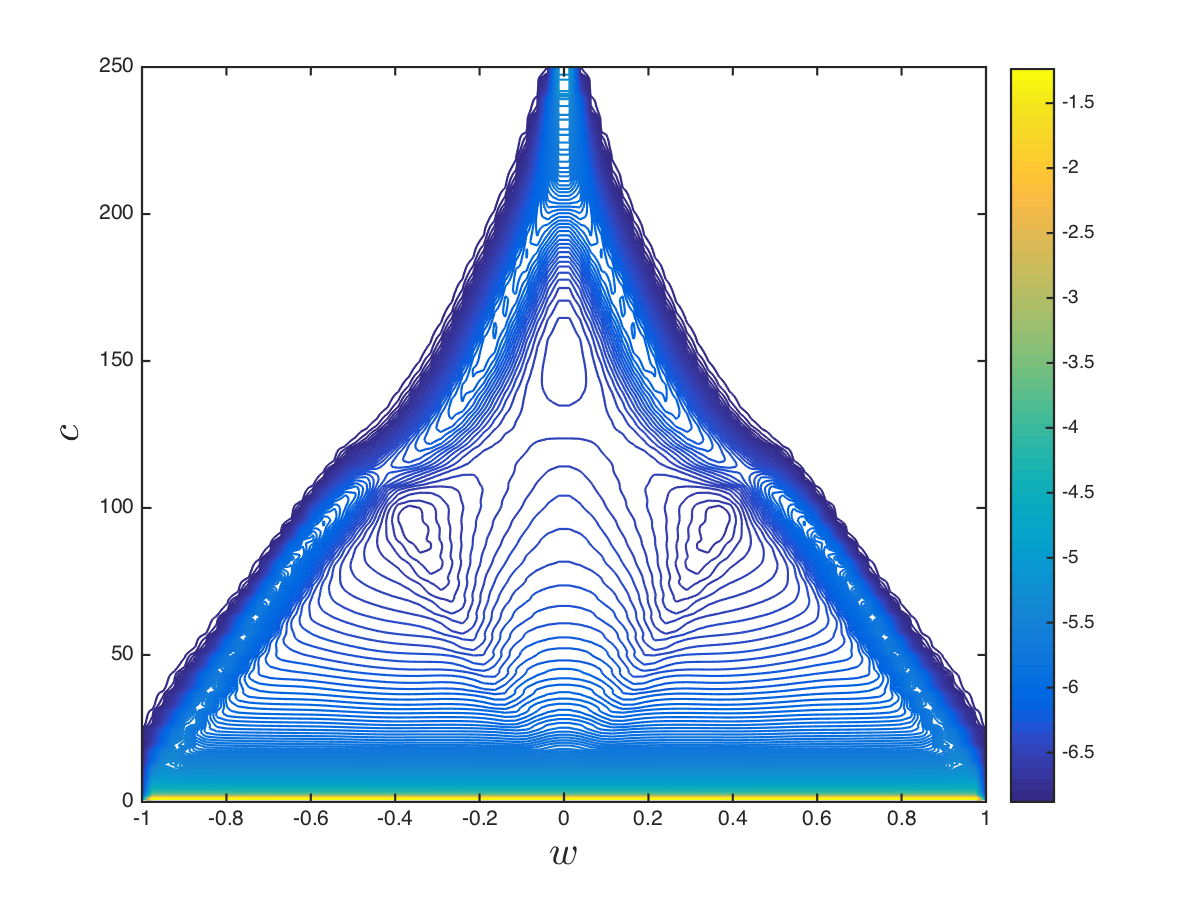}}
\\
\subfigure[t=50]{\includegraphics[scale= 0.3]{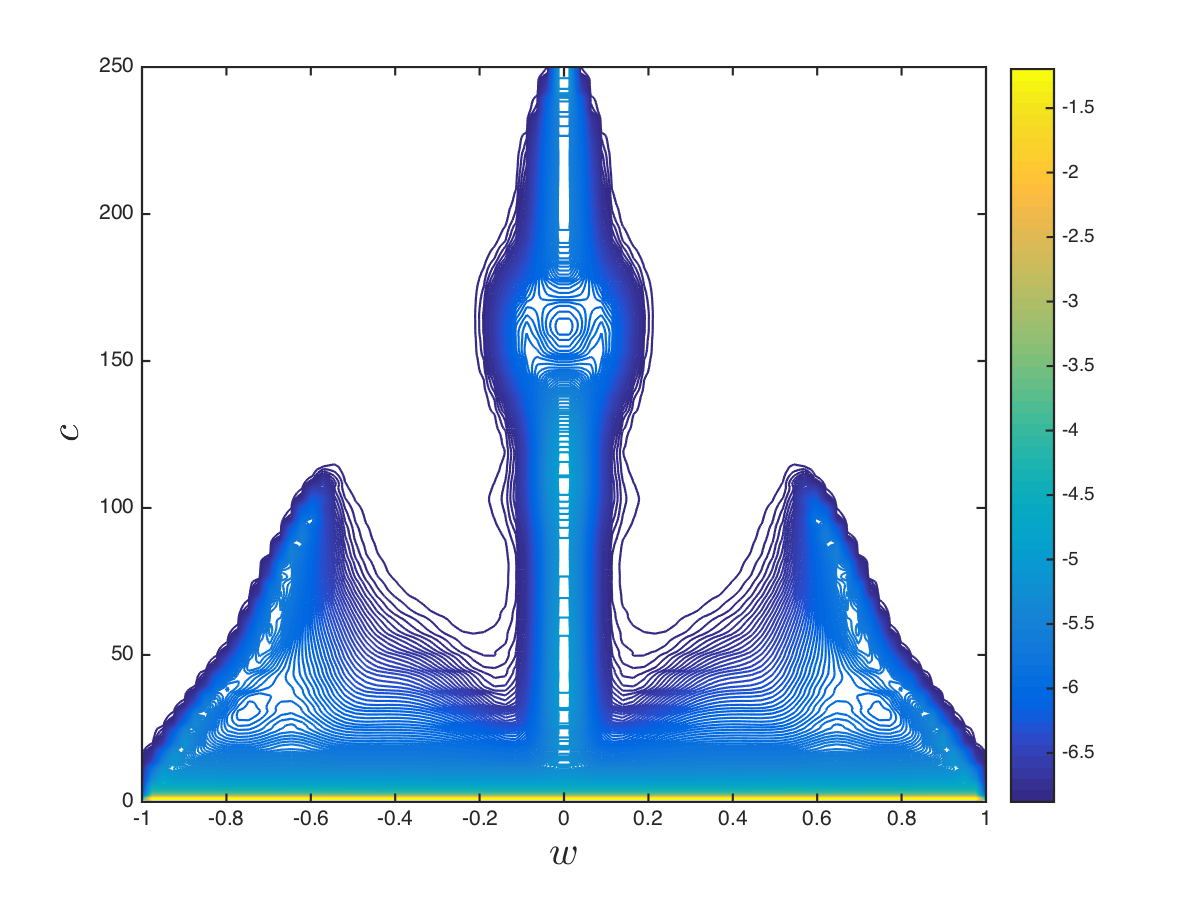}}
\hskip +0.2cm
\subfigure[t=100]{\includegraphics[scale= 0.3]{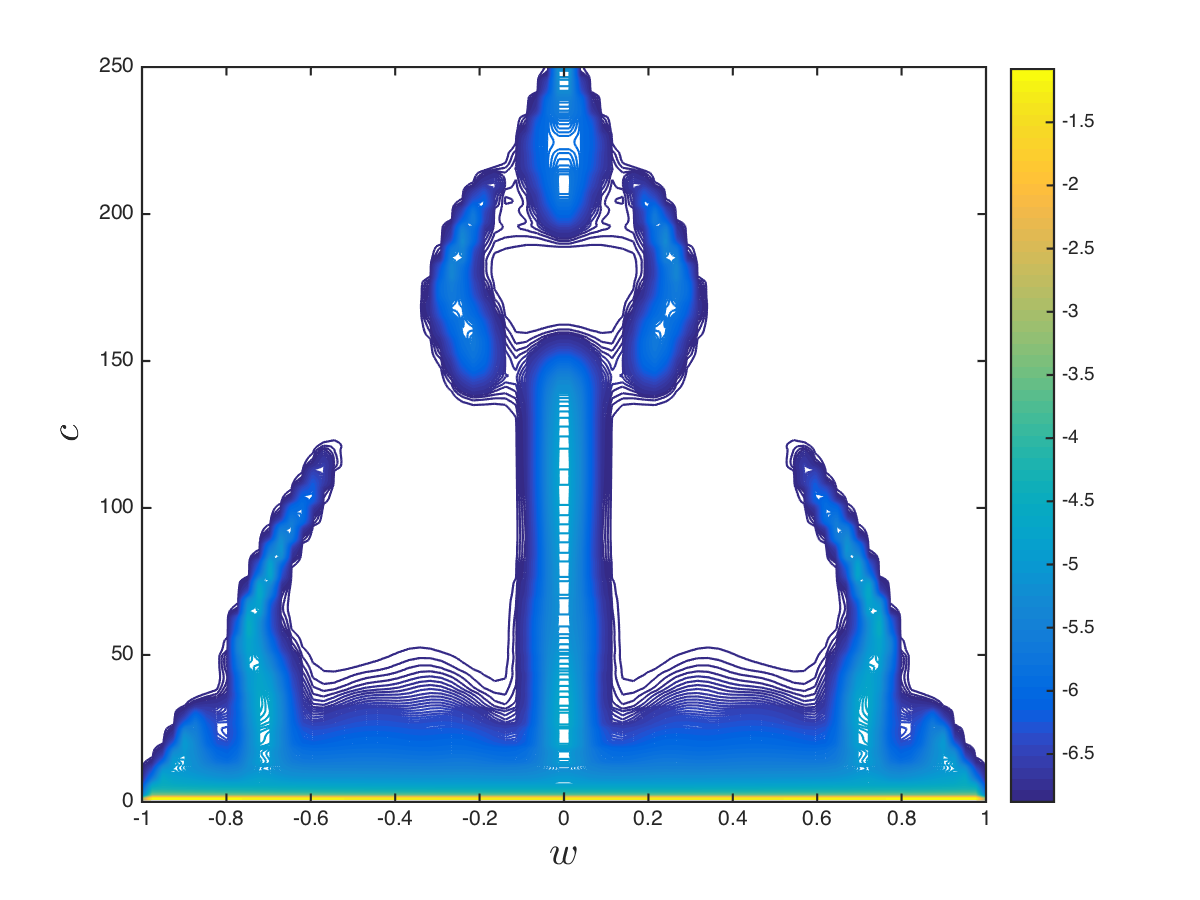}}
\caption{Test \#4. Evolution of the solution of the Fokker-Planck model \eqref{eq:FP}, where the interaction are described by \eqref{eq:HK}, with $\Delta(c)=d_0c/\cm$, and $d_0=1.01$, in the time frame $[0,T]$, with $T = 100$. The choice of $\Delta(c)$ reflects in the heterogeneous emergence of clusters with respect to the connectivity level: for higher level of connectivity consensus is reached, instead for lower levels of connectivity multiple opinion clusters are present. 
(Note: In order to better show its evolution, we represent the solution as $\log(f(w,c,t)+\epsilon)$, with $\epsilon = 0.001$.)
}\label{Fig:F6}
\end{figure}

Next, we perform a second computation where  the confidence bound depends on the number of connections as follows
\begin{align}\label{eq:Deltavar}
\Delta(c) = d_0\frac{c}{\cm}.
\end{align}
This choice reflects a behavior where agents with higher number of connections are prone to larger level of confidence.  We report in Figure \ref{Fig:F6} the evolution of \eqref{eq:initD}, where $\Delta(c)$ creates an heterogeneous emergence of clusters with respect to the connectivity level: for higher level of connectivity consensus is reached, since the bounded confidence level is larger, instead for lower levels of connectivity multiple clusters appears, up to the limiting case $c=0$, where the opinions are not influenced by the consensus dynamics.


\section{Conclusions}
The construction of kinetic models and numerical methods for the spreading of opinions over time dependent large scale networks has been considered. First we have introduced a Boltzmann model for the opinion interactions based on a preferential attachment process for the creation of new connections between agents. If the preferential attachment is independent from the agents' opinion the large time behavior of the network can be described analytically and originates both Poisson type distributions as well as truncated power laws. Next we derived the corresponding mean-field approximation which permits to have a deeper understanding of the asymptotic behavior of the opinion dynamics and to compute analytically stationary states in simplified situations. Robust numerical methods, based on stochastic as well as deterministic techniques have been introduced and their property discussed. The results, for various test cases, show the validity of the present approach. Several extensions of the present approach are possible. First one may consider the case where the number of agents in the network is not conserved, as it happens in a real social network. Moreover, control problems with the aim to force consensus over the network may be introduced.   

\newpage
\newpage
\appendix
\section{Properties of the network}\label{app:A}
In this section we report explicit computations concerning the properties of the network operator $\N[\cdot]$.
\subsection{Conservation of the total number of agents}\label{app:mass}
First we show that 
\begin{equation}\label{eq:sum_mass}
\sum_{c=0}^{\cm} \N[f(w,c,t)]\,dw=0.
\end{equation}
From \eqref{eq:master2} we have
\begin{equation*}\begin{split}
\sum_{c=1}^{\cm-1}\N[f(w,c,t)] =&  -\dfrac{2\Ur(f;w)}{\gamma+\beta} \sum_{c=1}^{\cm-1}[(c+1+\beta)f(w,c+1,t)-(c+\beta)f(w,c,t)] \\
&-\dfrac{2\Ua(f;w)}{\gamma+\alpha}\sum_{c=1}^{\cm-1}[(c-1+\alpha)f(w,c-1,t)-(c+\alpha)f(w,c,t)] \\
=&- \dfrac{2\Ur(f;w)}{\gamma+\beta}\left[(\cm+\beta)f(w,\cm,t)-(\beta+1)f(w,1,t)\right]\\
&+\dfrac{2\Ua(f;w)}{\gamma+\alpha}\left[(\cm-1+\alpha)f(w,\cm-1,t)-\alpha f(w,0,t)\right].
\end{split}\end{equation*}
Using the boundary conditions \eqref{eq:BD_master2} we have the desired property.
As a consequence we obtain the conservation of the total number of agents
\be
\dfrac{d}{dt}\sum_{c=0}^{\cm}\p(c,t) =-\int_{I}\sum_{c=0}^{\cm} \N[f(w,c,t)]\,dw =0.
\ee

\subsection{Mean density of connectivity}\label{app:mean_evo}
Next we consider the evolution of the mean density of connectivity $\gamma(t)$. 
We prove that
\begin{equation}
\begin{split}
\sum_{c=0}^{\cm}c \N[f(w,c,t)]&=2\Ur(f;w)\frac{\gamma_f+\beta g(w,t)}{\gamma+\beta}-2\Ua(f;w)\frac{\gamma_f+\alpha g(w,t)}{\gamma+\alpha}\\
&-\dfrac{2\Ur(f;w)}{\gamma+\beta} \beta f(w,0,t)
+\dfrac{2\Ua(f;w)}{\gamma+\alpha}(\cm+\alpha)f(w,\cm,t).
\end{split}
\label{eq:gamma_ev}
\end{equation}
In fact, thanks to \eqref{eq:master2}, in the internal points we have 
\begin{equation}\begin{split}\label{eq:sum_1}
-\sum_{c=1}^{\cm-1}c\N[f(w,c,t)] = &\dfrac{2\Ur(f;w)}{\gamma+\beta}\sum_{c=1}^{\cm-1}c\left[(c+1+\beta)f(w,c+1,t)\right.\\
&\left.-(c+\beta)f(w,c,t)\right]\\
+&\dfrac{2\Ua(f;w)}{\gamma+\alpha}\sum_{c=1}^{\cm-1}c\left[(c-1+\alpha)f(w,c-1,t)\right.\\
&\left.-(c+\alpha)f(w,c,t)\right].
\end{split}\end{equation}
We observe that the first sum in \eqref{eq:sum_1} is equal to
\begin{equation}
\begin{split}
&\dfrac{2\Ur(f;w)}{\gamma+\beta} \sum_{c=1}^{\cm-1}[c(c+1)f(w,c+1,t)-c^2 f(w,c,t)] \\
&\qquad\qquad+\dfrac{2\Ur(f;w)}{\gamma+\beta}\beta \sum_{c=1}^{\cm-1}[c f(w,c+1,t)-c f(w,c,t)],\\
&= \dfrac{2\Ur(f;w)}{\gamma+\beta}\left[ \cm(\cm+\beta) f(w,\cm,t) 
-(\gamma_f+\beta g(w,t)) + \beta f(w,0,t)\right].
\end{split}
\end{equation}
Similarly the second sum in \eqref{eq:sum_1} is equal to 
\begin{equation}
\begin{split}
&\dfrac{2\Ua(f;w)}{\gamma+\alpha} \sum_{c=1}^{\cm-1}[c(c-1)f(w,c-1,t)-c^2 f(w,c,t)] \\
&\qquad\qquad+\dfrac{2\Ua(f;w)}{\gamma+\alpha}\alpha \sum_{c=1}^{\cm-1}[c f(w,c-1,t)-c f(w,c,t)],\\
&= \dfrac{2\Ua(f;w)}{\gamma+\alpha}\left[\gamma_f(w,t)-\cm(\cm-1) f(w,\cm-1,t)-\cm f(w,\cm,t) \right]\\
&\qquad\qquad+\dfrac{2\Ua(f;w)}{\gamma+\alpha}\alpha\left[g(w,t)-\cm f(w,\cm-1,t)-f(w,\cm,t) \right].
\end{split}
\end{equation}
Using the boundary condition for $c=\cm$ (since the one at $c=0$ does not play any role here) we have 
\begin{equation}\begin{split}
 -\cm \N[f(w,\cm,t)]=&-\cm\dfrac{2\Ur(f;w)}{\gamma+\beta}(\cm+\beta)f(w,\cm,t)\\
 & +\cm\dfrac{2\Ua(f;w)}{\gamma+\alpha}(\cm-1+\alpha)f(w,\cm-1,t),
\end{split}\end{equation}
which together with the above computations yields (\ref{eq:gamma_ev}).

As a consequence we have
\begin{equation}
\begin{split}
&\dfrac{d}{dt}\gamma(t)=-2\int_I\Ur(f;w)\frac{\gamma_f+\beta g(w,t)}{\gamma+\beta}dw+2\int_I\Ua(f;w)\frac{\gamma_f+\alpha g(w,t)}{\gamma+\alpha}dw\\
&\,\,\,\,\,+\dfrac{2\beta}{\gamma+\beta}\int_I \Ur(f;w) f(w,0,t)\,dw
-\dfrac{2(\cm+\alpha)}{\gamma+\alpha}\int_I \Ua(f;w) f(w,\cm,t)\,dw
\end{split}
\end{equation}
\subsection{Asymptotic behavior}\label{app:st}
In the following we compute the explicit stationary solution $\ps(c)$ for the evolution of $\p(c,t)$ in the linear case with $\Ua=\Ur$, $\beta=0$ and assuming
\[
\sum_{c=0}^{\cm} \ps(c) = 1, \qquad \sum_{c=0}^{\cm} c\ps(c) = \gamma_{\infty}.
\]
Note that in the sequel, for notation simplicity, we denote by $\gamma=\gamma_{\infty}$ the asymptotic stationary value reached by the mean density of connectivity.
\begin{proposition}
For each $c\in {\mathcal C}$ the stationary solution to (\ref{eq:Ldef}) or equivalently
\be\label{prop:st}
(c+1)\ps(c+1) = \dfrac{1}{\gamma+\alpha}\left[(c(2\gamma+\alpha)+\gamma\alpha)\ps(c)-\gamma(c-1+\alpha)\ps(c-1)\right]
\ee
is given by 
\be\label{prop:sol}
\ps(c) = \left(\dfrac{\gamma}{\gamma+\alpha}\right)^c \dfrac{1}{c!}\alpha(\alpha+1)\cdots (\alpha+c-1)\ps(0)
\ee
where
\be
\ps(0) = \left(\dfrac{\alpha}{\alpha+\gamma}\right)^{\alpha}.
\ee
\end{proposition}
\begin{proof}
Let us show (\ref{prop:sol}) by induction. First, from the boundary condition (\ref{eq:BD_master}) at $c=0$ we immediately have
\be
\ps(1) = \left(\dfrac{\gamma}{\gamma+\alpha}\right)\alpha\ps(0).
\ee
Now let us assume that \eqref{prop:sol} holds true for $c$, we want to prove that
\be
\ps(c+1) = \left(\dfrac{\gamma}{\gamma+\alpha}\right)^{c+1} \dfrac{1}{(c+1)!}\alpha(\alpha+1)\cdots(\alpha+c)\ps(0).
\ee
From \eqref{prop:st} we have
\begin{equation*}\begin{split}
(c+1)\ps(c+1) =& \dfrac{1}{\gamma+\alpha}\left[(c(2\gamma+\alpha)+\gamma\alpha)\left(\dfrac{\gamma}{\gamma+\alpha}\right)^c \dfrac{1}{c!}\alpha(\alpha+1)\cdots (\alpha+c-1)\ps(0) \right. \\
& \left. -\gamma(c-1+\alpha)\left(\dfrac{\gamma}{\gamma+\alpha}\right)^{c-1}\dfrac{1}{(c-1)!}\alpha\cdots (\alpha+c-2)\ps(0)\right] \\
=& \left(\dfrac{\gamma}{\gamma+\alpha}\right)^{c}\dfrac{1}{(c-1)!}\alpha\cdots(\alpha+c-1)\left[\dfrac{c(2\gamma+\alpha)+\gamma\alpha}{c(\gamma+\alpha)}-1\right]\ps(0) \\
=& \left(\dfrac{\gamma}{\gamma+\alpha}\right)^{c+1}\dfrac{1}{c!}\alpha\cdots(\alpha+c-1)(\alpha+c)\ps(0).
\end{split}\end{equation*}
By direct inspection one verifies that also the boundary condition (\ref{eq:BD_master}) at $c=\cm$ is verified. 
\end{proof}

\section{Properties of the implicit-explicit scheme}\label{app:B}
Let us consider the following implicit-explicit discretization of \eqref{eq:FP} 
\be
\dfrac{f_i^{n+1}-f_i^n}{\Delta t}+\N[f^{n+1}_i]=\dfrac{\mathcal{F}_{i+1/2}^n-\mathcal{F}_{i-1/2}^n}{\Delta w},
\label{eq:imex}
\ee
where $f_i^{n}=f_i^n(c)$, endowed with a positive initial condition $f_i^0(c)={f}_{i}(c,0)$. The main motivation for the time discretization above is related to the severe stability constraints of an explicit scheme applied to the network operator which would require the time step to be $O(1/\cm)$ where $\cm \gg 1$.   
\subsection{Positivity}
In order to study the nonnegativity property of scheme (\ref{eq:imex}) it is convenient to rewrite it as a sequence of two steps
\be
\begin{split}
f_i^{n+1/2} &= f_i^n+\Delta t \dfrac{\mathcal{F}_{i+1/2}^n-\mathcal{F}_{i-1/2}^n}{\Delta w}\\
f_i^{n+1}&=f_i^{n+1/2}-\Delta t \N[f^{n+1}_i].
\end{split}
\label{eq:imex2}
\ee
The first step involves the Chang-Cooper type scheme and reads 
\begin{equation}\begin{split}\label{eq:system_2_disc}
f_i^{n+1/2} &= f_i^n+\dfrac{\Delta t}{\Delta w}\Bigg[\Big ( (1-\delta_{i+1/2})B_{i+1/2}^n+\dfrac{1}{\Delta w}C_{i+1/2}  \Big) f_{i+1}^n\\
& - \Big ((1-\delta_{i-1/2})B_{i-1/2}^n -\delta_{i+1/2}B_{i+1/2}^n \Big)f^n_i-\dfrac{1}{\Delta w}\Big(C_{i+1/2}+C_{i-1/2}\Big)f^n_i \\
& -\Big(\delta_{i-1/2}B_{i-1/2}^n-\dfrac{1}{\Delta w}C_{i-1/2}\Big)f_{i-1}^n\Bigg],
\end{split}\end{equation}
where $B_{i+ 1/2}^n, C_{i+ 1/2}$ are given by  
\be\begin{split}
B_{i+1/2}^n(c) &= \dfrac{D^2_{i+1/2}}{\Delta w}\int_{w_i}^{w_{i+1}}\frac{1}{D(w,c)^2}( \mathcal{P}[f](w,c,t^n)+\sigma^2 D'(w,c)D(w,c))dw, \\
C_{i+1/2}  &= \dfrac{\sigma^2}{2}D_{i+1/2}^2\ge 0.
\end{split}\ee
From the definition of the weight functions $\delta_{i+1/2}$ in (\ref{eq:weights}), the coefficients of $f^n_{i+1},f_{i-1}^n$, satisfy
\be\begin{split}
(1-\delta_{i+1/2})B_{i+1/2}^n+\dfrac{1}{\Delta w}C_{i+1/2} \geq 0,  \\
-\delta_{i-1/2}B_{i-1/2}^n+\dfrac{1}{\Delta w}C_{i-1/2} \geq 0. 
\end{split}\ee
In fact, setting $x=B_{i+1/2}^n\Delta w/C_{i+1/2}$, $y=B_{i-1/2}^n\Delta w/C_{i-1/2}$  the two inequalities are equivalent to show that $\forall\, x,y \in \mathbb{R}$
\be
x\left(1-\frac{1}{1-e^x}\right) \geq 0,\qquad
\frac{y}{e^y-1} \geq 0, 
\ee
which follow from the properties of the exponential function.

Then, in order to ensure the nonnegativity of the scheme the time step must satisfy the restriction
\be\label{eq:cond_dt_1}
\Delta t\le \frac{\Delta w}{\nu^n},
\ee
where
\be
\nu^n=\max_i\Big\{(1-\delta_{i-1/2})B^n_{i-1/2}-\delta_{i+1/2}B^n_{i+1/2}+\dfrac{1}{\Delta w}C_{i+1/2}+\dfrac{1}{\Delta w}C_{i-1/2}\Big\}.
\ee
Now, since the functions $D(w,c),P(w,w_*;c,c_*)$ are bounded for all $w\in I, c\in\mathcal{C}$  we have that $$|B^n_{i+1/2}|\le 2+\sigma^2 M,\qquad C_{i+1/2} \leq \sigma^2/2$$ 
where $M=\max_i|D'_{i+1/2}|$, and the condition \eqref{eq:cond_dt_1} simplifies to
\be
\Delta t \le \frac12 \frac{\Delta w}{\left(2+\sigma^2 M+\frac{\sigma^2}{2\Delta w}\right)}.
\label{eq:time1}
\ee
Therefore we have shown
\begin{proposition}
Under the time step restriction (\ref{eq:time1}) the first step in (\ref{eq:imex2}) preserves nonnegativity, namely $f_i^{n+1/2}(c)\geq 0$ if $f_i^n(c)\geq 0$, $i=1,\ldots,N$, $c \in {\mathcal C}$.
\end{proposition}

Typically when $\sigma^2$ is large this will originate a parabolic stability condition that requires $\Delta t=O(\Delta w^2)$. This can be avoided taking the diffusive part implicitly, however, since we were mostly interested in the case of small values of $\sigma^2$ we will not pursue this direction here.

Next, we consider the second step 
\be
f_i^{n+1}(c) = f^{n+1/2}_i(c)-\Delta t \N[f_i^{n+1}(c)].
\ee
Note that in general the fully implicit evaluation of $\N[\cdot]$ would require the use of a suitable iterative solver due to the nonlinearity in $f_i^{n+1}$. We therefore will consider a semi-implicit linearized version of the operator.

The scheme can be written as 
\be\begin{split}\label{eq:transp_1}
&\left[ 1+d^{n+1/2}(c)+a^{n+1/2}(c)+b^{n+1/2}(c) \right]f_i^{n+1}(c)\\
&-a^{n+1/2}(c) f_i^{n+1}(c+1) -b^{n+1/2}(c) f_i^{n+1}(c-1) = f^{n+1/2}_i(c),
\end{split}\ee
where 
\be
\begin{split}
&a^{n+1/2}(c)=\Delta t v^{n+1/2}_r (c+1+\beta),\,\qquad c=0,\ldots,\cm-1\\
&b^{n+1/2}(c)=\Delta t v^{n+1/2}_a (c-1+\alpha),\,\,\qquad c=1,\ldots,\cm\\
&d^{n+1/2}(c)=-\Delta t v^{n+1/2}_r +\Delta t v^{n+1/2}_a,\quad c=1,\ldots,\cm-1\\
&a^{n+1/2}(\cm)=0,\qquad b^{n+1/2}(0)=0,\\
&d^{n+1/2}(0)=b^{n+1/2}(1)-a^{n+1/2}(0),\\
&d^{n+1/2}(\cm)=-b^{n+1/2}(\cm)+a^{n+1/2}(\cm-1),
\end{split}
\label{eq:quant}
\ee
and we have set $v^{n+1/2}_r={2V^{n+1/2}_r}/{(\gamma^{n+1/2}+\beta)}$ and $v^{n+1/2}_a={2V^{n+1/2}_a}/{(\gamma^{n+1/2}+\alpha)}$.
Since alle quantities $a^{n+1/2}(\cdot)$, $b^{n+1/2}(\cdot)$ defined in (\ref{eq:quant}) are nonnegative, equations \eqref{eq:transp_1}-\eqref{eq:quant} define a diagonally dominant matrix of size $(\cm+1)\times (\cm+1)$ if
\begin{align}
\nonumber
&\Delta t\le \frac1{v^{n+1/2}_r-v^{n+1/2}_a}, \,\,\,\,\quad\qquad\qquad\qquad\qquad\qquad \frac{v^{n+1/2}_r}{v^{n+1/2}_a} > 1,\\ 
&\Delta t\le \frac1{v^{n+1/2}_r(1+\beta)-v^{n+1/2}_a\alpha}, \qquad\qquad\qquad\qquad \frac{v^{n+1/2}_r}{v^{n+1/2}_a} > \frac{\alpha}{(1+\beta)},
\label{eq:time2}
\\ 
\nonumber
&\Delta t\le \frac1{v^{n+1/2}_a(\cm-1+\alpha)-v^{n+1/2}_r(\cm+\beta)}, \quad \frac{v^{n+1/2}_a}{v^{n+1/2}_r}>\frac{(\cm+\beta)}{(\cm-1+\alpha)}. 
\end{align}
Note that when the above conditions on $v^{n+1/2}_r$ and $v^{n+1/2}_a$ are not satisfied, no time step restriction occurs. Conditions (\ref{eq:time2}) are not restrictive since in practice $\gamma^{n+1/2} \gg 1$ and so $v^{n+1/2}_a \ll 1$ and $v^{n+1/2}_r \ll 1$. Thus we have 
\begin{proposition}
Under the time step restriction (\ref{eq:time2}) the second step in (\ref{eq:imex2}) preserves nonnegativity, namely $f_i^{n+1}(c)\geq 0$ if $f_i^{n+1/2}(c)\geq 0$, $i=1,\ldots,N$, $c \in {\mathcal C}$.
\end{proposition}

\begin{remark} In particular, in the case where the rates are defined by (\ref{eq:rates}) since 
\[
g^{n+1}_i = \sum_{c=0}^{\cm}  f_i^{n+1}(c) = \sum_{c=0}^{\cm}  f_i^{n+1/2}(c) = g^{n+1/2}_i,
\]
the previous arguments applies to the fully implicit evaluation of $V^{n+1}_a=V_a(f_i^{n+1};w_i)$ and $V^{n+1}_r=V_r(f_i^{n+1};w_i)$.
\end{remark}

\subsection{Conservations and stability}
Let us consider the conservation properties of the scheme with respect to the variable $w$. Let us observe that from scheme \eqref{eq:imex} we get
\begin{equation}
\sum_{i=0}^N f_i^{n+1}(c)=\sum_{i=0}^{N}f_i^{n}(c)-\Delta t \sum_{i=0}^{N}\N[f_i^{n+1}]+
\dfrac{\Delta t}{\Delta w}\sum_{i=0}^{N}\left(\mathcal{F}_{i+1/2}^n-\mathcal{F}_{i-1/2}^n\right).
\end{equation}
Now since
\begin{equation*}\begin{split}
\sum_{i=0}^{N}\left(\mathcal{F}_{i+1/2}^n-\mathcal{F}_{i-1/2}^n\right)&=\sum_{i=0}^{N-1}\mathcal{F}_{i+1/2}^n-\sum_{i=1}^N\mathcal{F}_{i-1/2}^n+\mathcal{F}_{N+1/2}^n-\mathcal{F}_{-1/2}^n \\
&= \mathcal{F}_{N+1/2}^n-\mathcal{F}_{-1/2}^n,
\end{split}\end{equation*}
by imposing no-flux boundary conditions, i.e.
\be
\mathcal{F}_{N+1/2}^n=0,\qquad\mathcal{F}_{-1/2}^n = 0,
\ee
we obtain that for all $n\ge 0$ the following conservation equation for the density of connections is satisfied
\be
 \rho^{n+1}(c)=\rho^n(c)-\Delta t \sum_{i=0}^{N}\N[f_i^{n+1}].
\ee
Summing over $c$ in the above equation yields the conservation of the total number of agents
\be
\sum_{c=0}^{\cm}\rho^{n+1}(c)=\sum_{c=0}^{\cm}\rho^n(c).
\ee
From this identity we have
\begin{proposition}
Under the time step restrictions (\ref{eq:time1}) and (\ref{eq:time2}), the numerical scheme defined by  \eqref{eq:imex} is stable in the discrete $L_1$-norm.
\end{proposition}


\begin{thebibliography}{99}
\bibitem{Ace} D.~Acemoglu, O.~Asuman. Opinion dynamics and learning in social networks. \emph{Dynamic Games and Applications}, 1, 3--49, 2011.
\bibitem{AB} R.~Albert, A.-L.~Barabàsi. Statistical mechanics of complex networks. \emph{Reviews of modern physics}, 74(1): 1--47, 2002.
\bibitem{AP} G.~Albi, L.~Pareschi. Binary interaction algorithm for the simulation of flocking and swarming dynamics. \emph{SIAM Journal on Multiscale Modeling and Simulations}, 11(1), 1--29, 2013.
\bibitem{AHP} G.~Albi, M.~Herty, L.~Pareschi. Kinetic description of optimal control problems and applications to opinion consensus. \emph{Communications in Mathematical Sciences}, 13(6): 1407--1429, 2015.
\bibitem{APZa} G.~Albi, L.~Pareschi, M.~Zanella. Boltzmann-type control of opinion consensus through leaders. \emph{Philosophical Transactions of the Royal Society of London A: Mathematical, Physical and Engineering Sciences}, 372(2028): 20140138, 2014.
\bibitem{APZc}  G.~Albi, L.~Pareschi, M.~Zanella. On the optimal control of opinion dynamics on evolving networks. \emph{IFIP TC7 2015 Proceedings}, to appear.
\bibitem{ASBS} L.~A.~N.~Amaral, A.~Scala, M.~Bathélemy, H.E.~Stanley. Classes of small-world networks. \emph{Proceedings of the National Academy of Sciences of the United States of America}, 97(21): 11149--11152, 2000.
\bibitem{BA} A.-L.~Barabàsi, R.~Albert. Emergence of scaling in random networks. \emph{Science}, 286(5439): 509--512, 1999.
\bibitem{BAJ} A.-L.~Barabàsi, R.~Albert, H.~Jeong. Mean-field theory for scale-free random networks. \emph{Physica A: Statistical Mechanics and its Applications}, 272(1): 173--187, 1999.
\bibitem{BMS} L.~Boudin, R.~Monaco, F.~Salvarani. Kinetic model for multidimensional opinion formation. \emph{Physical Review E}, 81(3): 036109, 2010.
\bibitem{BrTo} C.~Brugna, G.~Toscani. Kinetic models of opinion formation in the presence of personal conviction. \emph{Physical Review E}, 92, 052818, 2015.
\bibitem{BCDS} C.~Buet, S.~Cordier, V.~Dos Santos. A conservative and entropy scheme for a simplified model of granular media. \emph{Transport Theory and Statistical Physics}, 33(2): 125--155, 2004.
\bibitem{BD} C.~Buet, S.~Dellacherie. On the Chang and Cooper numerical scheme applied to a linear Fokker-Planck equation. \emph{Communications in Mathematical Sciences}, 8(4): 1079--1090, 2010.
\bibitem{CC} J.~S.~Chang, G.~Cooper. A practical difference scheme for Fokker-Planck equation. \emph{Journal of Computational Physics}, 6: 1--16, 1970.
\bibitem{Das} A.~Das, S.~Gollapudi, K.~Munagala. \emph{Modeling opinion dynamics in social networks}, Proceedings of the 7th ACM international conference on Web search and data mining,  ACM New York, 403--412, 2014. 
\bibitem{DL} M.~Dolfin, L.~Miros\l{}av. Modeling opinion dynamics: how the network enhances consensus. \emph{Networks \& Heterogeneous Media}, 10(4): 877-896, 2015.
\bibitem{DMPW} B.~D\"uring, P.~A.~Markowich, J.-F.~Pietschmann, M.-T.~Wolfram. Boltzmann and Fokker-Planck equations modelling opinion formation in the presence of strong leaders. \emph{Proceedings of the Royal Society of London A}, 465(2112): 3687--3708, 2009.
\bibitem{DT} B.~D\"uring, M.-T.~Wolfram. Opinion dynamics: inhomogeneous Boltzmann-type equations modelling opinion leadership and political segregation. \emph{Proceedings of the Royal Society of London A}, 471(2182):20150345, 2015.
\bibitem{HK} R.~Hegselmann, U.~Krause. Opinion dynamics and bounded confidence, models, analysis and simulation. {\em Journal of Artifcial Societies and Social Simulation}, 5(3):2, 2002.
\bibitem{LLPS} E.~W.~Larsen, C.~D.~Levermore, G.~C.~Pomraning, J.~G.~Sanderson. Discretization methods for one-dimensional Fokker-Planck operators. \emph{Journal of Computational Physics}, 61: 359--390, 1985.
\bibitem{N} M. E. J.~Newman. The structure and function on complex networks. \emph{SIAM Review}, 45(2): 167--256, 2003.
\bibitem{MB} M.~Mohammadi, A.~Borz\`{i}. Analysis of the Chang-Cooper discretization scheme for a class of Fokker-Planck equations. \emph{Journal of Numerical Mathematics}, 23(3): 271--288, 2015.
\bibitem{PR} L.~Pareschi, G.~Russo. An introduction to Monte Carlo methods for the Boltzmann equation. \emph{ESAIM: Proceedings}, EDP Sciences. Vol. 10: 35--75, 2001. 
\bibitem{PTa} L.~Pareschi, G.~Toscani. \emph{Interacting Multiagent Systems. Kinetic Equations and Monte Carlo Methods}. Oxford University Press, 2013.
\bibitem{PTb} L.~Pareschi, G.~Toscani. Wealth distribution and collective knowledge: a Boltzmann approach. \emph{Philosophical Transactions of the Royal Society of London A: Mathematical, Physical and Engineering Sciences}, 372(2028): 20130396, 2014.

\bibitem{Patt} S.~Patterson, B.~Bamieh. \emph{Interaction-driven opinion dynamics in online social networks}, Proceedings of the First Workshop on Social Media Analytics, ACM New York, 98--110, 2010 

\bibitem{S} S.~H.~Strogatz. Exploring complex networks. \emph{Nature}, 410(6825): 268--276, 2001.
\bibitem{SWS} K.~Sznajd-Weron, J.~Sznajd. Opinion evolution in closed community. \emph{International Journal of Modern Physics C}, 11(6): 1197--1165, 2000.
\bibitem{T} G.~Toscani. Kinetic models of opinion formation. \emph{Communications in Mathematical Sciences}, 4(3): 481--496, 2006.
\bibitem{WS} D.~J.~Watts, S.~H.~Strogatz. Collective dynamics of 'small-world' networks. \emph{Nature}, 393: 440--442, 1998.
\bibitem{XZW} Y.-B.~Xie, T.~Zhou, B.-H.~Wang. Scale-free networks without growth. \emph{Physica A}, 387: 1683--1688, 2008.
\end{thebibliography}
\end{document}